\documentclass[11pt]{amsart}
\usepackage{graphicx}
\usepackage{amssymb,amsmath,latexsym} 
\newcommand{\mute}[2] {}

\newcommand{\labell}[1] {\label{#1}}

\renewcommand{\Tilde}{\widetilde}
\renewcommand{\Hat}{\widehat}
\newcommand{\QED}{\hfill$\Box$}
\newcommand{\less}{{\smallsetminus}}
\newcommand{\NI}{{\noindent}}
\newcommand{\SSS}{{\smallskip}}
\newcommand{\MS}{{\medskip}}

\newcommand{\la}{{\lambda}}
\newcommand{\De}{{\Delta}}

\newcommand{\Ga}{{\Gamma}}
\newcommand{\lat}{{\ft_\Z}}
\newcommand{\ga}{{\gamma}}

\newcommand{\ML}{\operatorname{ML}}
\newcommand{\iness}{\operatorname{iness}}

\newcommand\fk {{\mathfrak k}}
\newcommand\ft {{\mathfrak t}}
\newcommand\p {{\partial}}

\newcommand{\Oo}{{\mathcal O}}
\newcommand{\Ee}{{\mathcal E}}

\newcommand{\ov}{\overline}

\newcommand{\Uu}{{\mathcal U}}

\newcommand{\Aa}{{\mathcal A}}
\newcommand{\Ss}{{\mathcal S}}
\newcommand{\Cc}{{\mathcal C}}
\newcommand{\Ii}{{\mathcal I}}

\newcommand{\La}{{\Lambda}}
\newcommand{\CP}{{\mathbb CP}}
\newcommand{\C}{{\mathbb C}}

\newcommand{\R}{{\mathbb R}}
\newcommand{\Z}{{\mathbb Z}}

\newcommand{\om}{{\omega}}
\newcommand{\ka}{{\kappa}}
\newcommand{\al}{{\alpha}}
\newcommand{\si}{{\sigma}}
\newcommand{\io}{{\iota}}
\newcommand{\be}{{\beta}}
\newcommand{\Symp}{{\rm Symp}}

\newcommand{\Isom}{{\rm Isom}}

\newcommand{\eps}{{\epsilon}}

\newcommand{\ssminus}{{{\smallsetminus}}}

\newtheorem{theorem}{Theorem}[subsection]
\newtheorem{thm}[theorem]{Theorem}
\newtheorem{prop}[theorem]{Proposition}
\newtheorem{corollary}[theorem]{Corollary}
\newtheorem{cor}[theorem]{Corollary}

\newtheorem{lemma}[theorem]{Lemma}
\newtheorem{proposition}[theorem]{Proposition}
\newtheorem{definition}[theorem]{Definition}
\newtheorem{defn}[theorem]{Definition}
\newtheorem{rmk}[theorem]{Remark}
\newtheorem{example}[theorem]{Example}
\newtheorem{quest}[theorem]{Question}
\numberwithin{figure}{section}
\numberwithin{equation}{section}

\begin{document}

\title{Polytopes with mass linear functions II: the $4$-dimensional case}
\author{Dusa McDuff}\thanks{First author partially supported by NSF grant DMS 0905191, and second by NSF grant DMS 0707122.}
\address{Mathematics Department,  Barnard College, Columbia University
NY, USA}
\email{dmcduff@barnard.edu}
\author{Susan Tolman}
\address{Department of Mathematics,
 University of Illinois at Urbana--Champaign, 
IL, USA}
\email{tolman@math.uiuc.edu}
\date{May 18, 2011}

\begin{abstract} 
This paper continues the analysis 
begun in {\it Polytopes with mass linear functions, Part I}
 of the structure of smooth moment polytopes 
$\De \subset \ft^*$
that support a mass linear function
$H \in \ft$.
As explained there, besides its purely  combinatorial interest, 
this question is relevant to the study of the homomorphism
$\pi_1(T^n)\to \pi_1\bigl(\Symp(M_\De, \om_\De)\bigr)$ from the 
fundamental group of the torus $T^n$ to that of the  group of  
symplectomorphisms
of the
$2n$-dimensional symplectic toric manifold $(M_\De, \om_\De)$
associated to $\De$.
 
 In Part I, we made a general investigation of this question and classified 
 all mass linear pairs $(\De, H)$ in dimensions up to three.  
The main result of the current paper
is a classification of all $4$-dimensional examples.
Along the way, we investigate the properties 
of general constructions such as fibrations, blowups and expansions
(or wedges), 
describing their
effect
both on moment polytopes and
on mass linear functions.

We end by discussing the relation of mass linearity to Shelukhin's notion of full mass linearity.  The two concepts agree in dimensions up to and including $4$.  However full mass linearity may 
be
the more natural concept when considering the question of which blow ups preserve mass linearity. 

\end{abstract}

\maketitle

\tableofcontents

\section{Introduction}

\subsection{Statement of main results}\labell{ss:intro}

This paper continues the analysis begun in 
\cite{MTI} of the structure of smooth 
polytopes $\De$ that support 
an essential 
mass linear function $H$.
As we show 
there, besides its purely  combinatorial interest, 
this question is relevant to the
 study 
 of the homomorphism
$\pi_1(T)\to  
\pi_1\bigl(\Symp(M_\De, \om_\De)\bigr)$
from the fundamental group of 
the torus 
$T$ to that of the  group of 
symplectomorphisms
of the 
symplectic toric manifold 
 $(M_\De,  \om_\De, T)$
associated to $\Delta$.
The paper \cite{Mct} describes other applications, 
such as
understanding when a product manifold of the form $(M\times S^2, \om +\si)$   has
more than one toric structure.

In \cite{MTI} (from now on called Part I), 
we made a 
general investigation of
the properties of mass linear functions
and classified 
all  
essential
mass linear pairs $(\De, H)$ in dimensions up to three.  
The main result of the current paper
is a classification of all $4$-dimensional examples.  We also develop new
 tools for understanding the topological properties 
of symplectic toric manifolds.

Before stating our results we shall remind the reader of some of the basic concepts 
introduced in Part I; more details can be found there.

Let $\ft$ be a real vector space
with integer lattice 
$\lat \subset \ft$.
Let $\ft^*$ denote the dual space and 
let
$\langle \, \cdot \, , \, \cdot \, \rangle \colon \ft \times \ft^* \to \R$ 
denote
the natural pairing.
A {\bf (convex) polytope} $\Delta \subset \ft^*$ is the bounded intersection of a finite
set of affine half-spaces.
In this paper, we shall always write $\De$ in the form
\begin{equation}\labell{eq:De}
\De =   
\bigcap_{i = 1}^N \bigl\{ x \in \ft^* \mid \langle \eta_i, x\rangle \leq \kappa_i  \bigr\},
\end{equation}
where the {\bf outward conormals} $\eta_i$ lie in $\ft$ and
the {\bf support numbers} $\kappa_i$ lie in $\R$ for all $1 \leq i \leq N$.   
We 
always assume that
$\Delta$ has  a nonempty interior,
and that  the affine
span of each {\bf facet}  $F_i := \Delta \cap 
 \{ x \in \ft^* \mid \langle \eta_i, x\rangle = \kappa_i  \}$ is a hyperplane.
Further, we assume that $\De$ is {\bf smooth}, 
that is,
for each vertex $v$ of $\De$
the primitive
outward conormals 
to the facets which meet at $v$ 
form a basis for 
the integral lattice
$\lat$ of $\ft$.
  In particular, a smooth polytope is 
{\bf simple},
that is, $\dim \ft$ 
facets meet at every vertex.

Given a 
polytope $\De \subset \ft^*$,  
let $c_\De$ denote its center of mass, 
considered as 
a
function of
the support constants $\ka$.
An element $H\in \ft$ is said to be {\bf mass linear} on $\De= \De(\ka)$
 if
the function 
$$
 \Hat H \colon  \ka'\mapsto \langle H,c_\De(\ka')\rangle
$$
is linear for all $\ka'$ near $\ka$;   
cf.\  [I, Definition 1.2 and Lemma 2.3].
In this case there are 
real
numbers $\be_i$, called the {\bf coefficients} of $H$ such that
$\langle H,c_\De(\ka')\rangle = \sum \be_i\ka_i'$ 
for all $\ka'$ near $\ka$.

 To explain 
the important distinction between 
{\em essential} and {\em inessential}
mass linear  functions,
we introduce an equivalence relation on the facets.  Following 
[I, Definition~1.12]
(and [I, Corollary~3.5 and Remark~1.6]),
we say that
two distinct facets $F_i$ and $F_j$ are {\bf equivalent},
denoted $F_i \sim F_j$,
exactly if there is an 
integral
affine transformation
of $\De(\ka)$ that 
takes $F_i$ to $F_j$ and is {\bf robust}, 
in the sense that it persists 
when $\ka$ is perturbed.
Let $\Ii$ denote
the set of equivalence
classes of facets.  
We say that $H \in \ft$ is {\bf inessential} iff
\begin{equation*}
H = \sum \be_i \eta_i, \quad \mbox{where} \quad
\be_i\in \R
\ \forall \, i \quad \mbox{and}
\quad \sum_{i \in I} \beta_i = 0
\quad \forall \, I \in \Ii. 
\end{equation*} 
Otherwise, we say that $H$ is {\bf essential}.
By  Proposition~\ref{prop:inessential}, every inessential function is mass linear. 

As an example  consider the standard $k$-simplex $\De_k$, 
that is,
$$
\De_k = \Bigl\{x\in \R^k\,\Big|\, 0\le x_i  \ \forall \, i \mbox{ and } \sum_{i+1}^k x_i\le 1\Bigr\}.
$$
Any pair of facets of 
$\De_k$ is
 equivalent, and so every
$H \in \ft$ is inessential.

To understand 
the implications of these definitions,
consider the symplectic toric manifold $(M_\Delta,\omega_\Delta,\Phi_\Delta)$ with
moment image $\Phi_\Delta(M_\Delta) = \Delta$.
Let 
$\Symp(M_\Delta,\omega_\Delta)$ denote 
the group of symplectomorphisms
of $(M_\Delta,\omega_\Delta)$, and let 
$\Isom(M_\Delta,\omega_\Delta)$ denote the group of
K\"ahler isometries, that is, the subgroup of symplectomorphisms that also preserve the canonical complex structure on $M$.
As we showed in [I, \S 1.2], 
if
the circle $\Lambda_H$ generated by $H \in \lat$ has  finite order in 
$\pi_1\bigl(\Symp(M_\Delta,\omega_\De)\bigr)$,
then $H$ is mass linear.\footnote
{
In fact, McDuff shows in \cite[\S4]{Mct} that $H$ is mass linear precisely if the rational cohomology ring of the toric bundle $M_\De\times_{\La_H} S^3\to S^2$ is isomorphic to the product
ring $H^*(M_\De)\otimes H^*(S^2)$.}
Moreover, $\Lambda_H$ has finite order in 
$\pi_1\bigl(\Isom(M_\De,\omega_\De)\bigr)$ exactly if
$H$ is inessential. Finally,  
if there are no essential mass linear functions on $\Delta$, then the natural map
$\pi_1\bigl(\Isom(M, \om)\bigr)\to\pi_1\bigl(\Symp(M, \om)\bigr)$ is an injection.
For more details, see [I, \S 1.2].

Most polytopes do not admit nonzero mass linear functions. We showed in Part I that 
in dimension two the only ones that do are 
the triangle, the parallelogram, and trapezoids,
corresponding respectively to the projective plane $\C P^2$, the product $S^2\times S^2$ and the different  Hirzebruch  surfaces ($S^2$-bundles over $S^2$).
Moreover, in each  
case all mass linear functions are inessential.

In dimension three, although there are more examples of polytopes with
mass linear functions 
(see Proposition \ref{prop:3d}), there are very few   
with essential mass linear functions.
To describe these,
we need the notion of ``bundle'', which is given in  
Definition \ref{def:bund} below.
One key fact
about bundles
 is  that if 
a smooth polytope $\Delta$
is a bundle over $\Hat \De$ with fiber $\Tilde\De$,
then the corresponding toric manifold 
$M_\De$ is a  bundle  over $M_{\Hat\De}$ with fiber $ M_{\Tilde\De}$;
see  [I, Remark 5.2].
In Part I, we showed that every smooth $3$-dimensional polytope which admits an essential mass linear function 
is a $\Delta_2$ bundle over $\Delta_1$. Since the moment image of $\CP^n$ is $\Delta_n$, this
implies that, $M_\Delta$ 
is
 a $\CP^2$ bundle over $\CP^1$.

The
analogous
 statement in dimension $4$ is more complicated
because there is a much greater variety of examples.
Correspondingly, we need to introduce new terminology.
Blowups are defined in Definition~\ref{def:blowup}.
As the name suggests, 
if $\Delta'$ is the blowup of 
a polytope
$\Delta$,
then 
the corresponding toric manifold $M_{\Delta'}$ is the blowup of 
$M_\Delta$; see 
Remark~\ref{rmk:geoblow} (i).
Double expansions are defined in 
Definition~\ref{def:dexpand}.
By [I, Remark 5.4],
a polytope $\De$ that is an expansion of $\Tilde\De$ corresponds to  a toric manifold that is a nonsingular pencil with fibers 
$M_{\Tilde\De}$.  
Thus if $\De$ is a double expansion,
the corresponding toric manifold
is a \lq\lq double pencil".

Additionally,
fix $H \in \ft$ and a 
polytope $\Delta \subset \ft^*$.
We say that a facet is {\bf symmetric} 
 (or  {\bf $\mathbf H$-symmetric}),
if $\langle H,c_\De(\ka)\rangle$ does not change when 
that facet is moved. Otherwise, we say that the facet is {\bf asymmetric} 
(or {\bf $\mathbf H$-asymmetric}). 
Further, we say that a facet is {\bf pervasive} if it meets all other facets.

Finally, a {\bf face} of $\Delta$ is a (nonempty) intersection of
some collection of facets of $\De$; 
a {\bf $\mathbf k$-face} 
is a face of dimension $k$.    
We
denote the faces of $\De$ by
$F_I: = \cap_{i\in I} F_i$, where $I\subset \{1,\dots,N\}$.
Then  we say that 
the face
$F_I$ is {\bf symmetric}
if the facet 
$F_i$ is symmetric for each $i\in I$.

\begin{thm}\labell{thm:4d}
Let $H \in \ft$ be an essential mass linear function
on a smooth $4$-dimensional polytope
$\Delta \subset \ft^*$. 
There exists a smooth 
$4$-dimensional polytope $\ov\Delta \subset \ft^*$ so that either:
\begin{itemize}
\item [(a)] $H$ is an essential mass linear function on $\ov\Delta$
and at least 
one\footnote{For most polytopes $\ov\De$, only one of these statements is 
true.  The one exception is $\Delta_2 \times \Delta_1$ bundles over $\Delta_1$,
which belong in case (a2) and case (a3).}
of the 
 following statements is true:
\begin{enumerate}
\item[(a1)] $\ov\Delta$ is a $\Delta_3$ bundle over $\Delta_1$,
\item[(a2)] 
$\ov\Delta$  is a 
$\Delta_1$ bundle over a
polytope  which is 
a $\Delta_2$ bundle over $\Delta_1$, or
\item[(a3)]  $\ov\Delta$ is a $\Delta_2$ bundle over a 
polygon $\Hat\De$; or
\end{enumerate}
\item [(b)] 
$H$ is inessential on $\ov\Delta$, 
the polytope
$\ov\Delta$ is the double expansion of a 
polygon 
$\Tilde{\De}$, 
and the asymmetric facets are the four base-type facets.
\end{itemize}
Moreover,
$\Delta: = \ov \De(m)$ can be obtained from $\ov\Delta: = \ov \De(0)$ by a series of blowups. For each $k = 1, \dots, m$, the polytope
$\ov \De(k)$ is obtained from $\ov \De(k-1)$ by blowing up either along a 
symmetric $2$-face  or along
an edge of the form
$\ov F_{ij}\cap \ov G: = \ov{F}_i \cap \ov{F}_j \cap \ov{G}$, where $\ov{G}$ is a symmetric facet
of $\ov\De(k-1)$,
$\ov F_{ij}\cap \ov G$ intersects every asymmetric facet,
and $\gamma_i + \gamma_j = 0$.
Here $\gamma_\ell$ is
the coefficient of the support number of the facet $\ov{F}_\ell$
in the linear function 
$\langle H, c_{\ov{\Delta}} \rangle.$
\end{thm}

Combining this with the results of [I,\S1.2],  
we obtain the following corollary.

\begin{cor}  Let $(M,\om)$ be an  $8$-dimensional symplectic toric manifold.  
The natural map 
$$
\pi_1\bigl(\Isom(M,\om)\bigr)\to\pi_1\bigl(\Symp(M,\om)\bigr)
$$
is injective unless $M$ is a very special blowup of either  a 
double pencil or a bundle. Moreover the bundle either  has 
 $\C P^2$ or $\C P^3$ as its fiber or 
has a $\C P^2$ bundle over $\C P^1$ as its base. 
\end{cor}

The following results elaborate 
the statement of Theorem~\ref{thm:4d}.
\begin{itemize}
\item
We give a complete description of
all the essential mass linear functions $H$ on polytopes $\ov \Delta$
satisfying  
conditions (a1), (a2), and (a3)
in
Corollary~\ref{cor:Mabc}, Proposition~\ref{prop:doublebundle},  
and Proposition~\ref{prop:polybundle}, 
respectively. 
The inessential functions 
on a polytope
$\ov\De$ 
satisfying condition (b) 
are described in Lemma \ref{le:dexpan0}.

\item 
By Lemma~\ref{le:symblow} 
and Proposition~\ref{prop:blow4},
a  mass linear 
function on a polytope 
will still be mass linear
if the polytope is blown up by the types of blowups described above; 
moreover,  an essential function will still be essential on the blowup.  In fact, 
the linear  function
$\langle H, c_{\ov{\Delta}(k)} \rangle$ is unchanged under blowup:
the exceptional divisor of  each blowup is symmetric and so has zero coefficient,
while the coefficients of the other facets remain the same.

\item  
We explain in Proposition~\ref{prop:essblow2} exactly which blowups  on the polytope
$\ov \Delta$ described in (b) convert $H$ from 
an inessential to an essential mass linear function. 
\end{itemize}

Combining the above results,  we can draw a number of conclusions.

\begin{enumerate}
\item  In each case of 
part (a) of
Theorem \ref{thm:4d}
some, but not all, 
polytopes $\ov\De$ of the given form
support essential mass linear functions.
In case (a1)  and (a2) 
one can take 
the bundle
$\ov\De$ to be generic.
However, in case (a3),
while the base $\Hat\De$ 
can be 
any polygon
except a triangle,
the bundle itself must satisfy some very special conditions 
that have a topological interpretation; 
see Proposition~\ref{prop:geo3dep} 
Similarly, 
in case (b)  the polygon
$\Tilde \De$
can be anything except a triangle.

\item 
For any essential mass linear function 
of type (a),
the polytope $\ov \Delta$ has
between $3$ and $7$ asymmetric
facets; see Remark~\ref{rmk:asym}.  
In all cases $\ov\De$
can have $3$ asymmetric facets. 
However,
it is only possible to have 
$4$ or $6$
asymmetric facets  in the case 
(a1), and the  
only case with $7$ asymmetric facets is 
the product of $\Delta_1$ with a $\Delta_2$ bundle over $\Delta_1$.
Further in case (b) $\ov \De$ always has 
$4$ asymmetric facets.
Since,  
by Lemma~\ref{le:blowcent}, 
blowing down the facets of a polytope with
a mass linear function 
does not change the number of asymmetric facets, 
the original polytope $\De$ has
the same number of asymmetric facets.

\item  If the   polytope 
$\ov\De$ 
has more than four asymmetric facets 
then in all cases
$\ov \De =\De$,
that is, no blowups are 
possible.
Therefore,  every polytope that supports an essential mass linear
function with more than $6$ asymmetric facets is the product of $\Delta_1$
with a $\Delta_2$ bundle over $\Delta_1$;  the only examples with $6$ asymmetric
facets are $\Delta_3$ bundles over $\Delta_1$.

If it has three 
or four
asymmetric facets then
the situation is more complicated.
However, by Proposition \ref{prop:blow4},
 the edge blowups 
described in Theorem~\ref{thm:4d} 
are only possible when there are 
four asymmetric facets.
For more details see
 Remark~\ref{rmk:blowable}.

\item In all cases $\sum\ga_i = 0$; 
see Corollary \ref{cor:full0} and Remark~\ref{rmk:sum}.

\item
The mass linear functions on 
a given
polytope $\De$ form a vector space $V_{\ML}$, 
with a subspace $V_{\iness}$ consisting of inessential functions.  
If $\Delta$ is $4$-dimensional,
then
the quotient $V_{\ML}/V_{\iness}$ has dimension at most $1$ unless $\Delta$ is a
$\Delta_3$ bundle over $\Delta_1$, in which case it has dimension at most $2$.
\end{enumerate}

\begin{rmk}\labell{rmk:intmin} \rm
Suppose that $H$ is a mass linear function on 
$\De$ and that $\ov \De$ is obtained from $\De$ by any sequence of blowdowns.   Then Lemma \ref{le:blowcent}
shows that  $H$ is mass linear on $\ov \De$.  Therefore, when classifying mass 
linear functions one may 
assume that the underlying  polytope is {\bf minimal}, 
i.e., that none of its facets can be blown down.

The results in \S\ref{ss:min}  show 
that Theorem~\ref{thm:4d} still holds with the additional requirement that $\ov \Delta$ is minimal
as long as we omit the last two 
sentences, which restrict the kinds of blowups allowed.
For further details, see Remark~\ref{rmk:blowess} (i).

On the other hand, allowing
arbitrary blowups does not allow us to reduce the list of examples;
the results 
of \S\ref{ss:min} also 
show that 
there  exist minimal polytopes of each type.
In fact, for any  
sufficiently large $N$
there exists  a minimal polytope $\ov \De$ with $N$ facets
which satisfies the conditions of case (a3) or (b) of Theorem~\ref{thm:4d}; see 
Propositions \ref{prop:3depblowN} and \ref{prop:4indepblowN}.
In contrast, the polytopes $\ov\Delta$ described in 
(a1) and (a2)  have  $6$ and $7$ facets, respectively. 
\end{rmk}

\begin{rmk}\labell{rmk:smooth} \rm
(i)  
Many of our  constructions
and intermediate results are valid for   
all simple polytopes.
We wrote much of Part I in this generality, although 
our main classification theorem was only for smooth polytopes.  
In this paper  we assume  throughout  that $\De$ is smooth.\MS

\NI (ii) 
We have not
explicitly
written down a list of all the
$4$-dimensional smooth polytopes that admit nonzero mass linear functions
because the answer is too messy to be very enlightening. 
However, this information is easy to extract from our paper.
On the one hand, Propositions~\ref{prop:inessential} and \ref{prop:Iexpan},
allow us to classify  polytopes which admit inessential mass
linear functions with at most three asymmetric facets.
On the other hand, Propositions~\ref{prop:3dep} and \ref{prop:3indep}
classify all $4$-dimensional smooth polytopes which admit essential mass 
linear
functions with exactly three asymmetric facets, and 
Proposition~\ref{prop:4dep}, \ref{prop:4indep},
\ref{prop:5plus}, and \ref{prop:nonperv}  
classify all $4$-dimensional smooth polytopes which admit mass 
linear
functions with at least four asymmetric facets. 
\end{rmk}

We end this section with an 
example, which
demonstrates how blowups can transform an inessential function into an essential mass linear
function. Hence, it
is an example of case (b) of Theorem~\ref{thm:4d}, and not case (a1).

\begin{example}\labell{ex:7}\rm 
Let $\ov{\De} \subset (\R^4)^*$ 
be the $\De_3$ bundle over $\De_1$ with conormals
\begin{gather*}
\eta_1=(-1,0,0,0),\;\eta_2=(0,-1,0,0),\;\eta_3=(0,0,-1,0),\;
\eta_4=(1,1,1,0),\\
\al_1=(0,0,0,-1),\; \mbox{and} \;\al_2=(-1,-1,0,1).
\end{gather*}
The polytope $\ov\De$
is also 
the double expansion of the trapezoid
with conormals
$$
(-1,0), \; (0,-1), \; (1,1),\; \mbox{and}  \;(-1,-1)
$$
along its two parallel facets.

Denote by $\ov F_i$ and $\ov G_j$  the facets with conormals $\eta_i$ and $\al_j$, respectively.   
By Lemma~\ref{le:equiv}, two facets are equivalent exactly if the conormals of all the
other facets lie in a $3$-dimensional subspace.
Hence, $\ov F_1 \sim \ov F_2 \not\sim \ov F_3 \sim \ov F_4$, and so 
the function
$$
H: = \eta_1-\eta_2-\eta_3+\eta_4
$$
is inessential
on $\ov\De$.  
By Proposition~\ref{prop:inessential}, this implies that $H$ is mass linear on $\ov \De$; in fact,
\begin{equation}\labell{excoeff}
\langle H, c_{\De}(\ka) \rangle = \ka_1-\ka_2-\ka_3+\ka_4.
\end{equation}

Now consider the blowup $\De$ of $\ov{\De}$ along the edge 
$\ov F_{24}\cap \ov G_1$. 
This has a new 
facet $G_0$ (the exceptional divisor) with conormal 
$$
\al_0 = \eta_2+\eta_4+\al_1 = (1,0,1,-1).
$$
None of the  facets $\ov  F_i$ are equivalent in $\De$.  However,  
Proposition \ref{prop:blow4} and \eqref{excoeff} together imply
that $H$ is still mass linear on $\De$.
Therefore,
$H$ is essential 
on $\De$.

The corresponding toric manifold 
$M_{\ov \De}$ is a $\C P^3$
bundle over $\C P^1$;
in fact, it is
the projectivization of the vector bundle
$\Oo(-1)\oplus \Oo(-1)\oplus \Oo\oplus\Oo \to \C P^1$.
The toric manifold  $M_{\De}$ is the blowup of $M_{\ov \De}$ along a line in one of the fibers.
\end{example}

\subsection{Proof of Theorem~\ref{thm:4d}}\labell{ss:out}

This section explains the proof of Theorem \ref{thm:4d}.
Since we use results from part I without comment, readers might find it useful
look over  the beginning of \S\ref{ss:review}
where we summarize its main results.

We divide the proof of  Theorem \ref{thm:4d} into four steps. 
\MS

\NI {\bf Step 1:}  {\it Theorem \ref{thm:4d} holds 
if $\Delta$ has a nonpervasive asymmetric facet.}\MS

\NI {\it Proof.}
Let $H \in \ft$ be an essential mass linear function on a 
$4$-dimensional polytope $\De \subset \ft^*$.  
If one of the asymmetric facets $F$ is not pervasive, 
then  Proposition~\ref{prop:nonperv}
implies that
$\Delta$ is a bundle
of one of the three types mentioned in part (a).

The proof of Propositition~\ref{prop:nonperv} uses the fact
that, by Propositions~\ref{prop:flat} and \ref{prop:bund}, 
it is enough to prove that the claim holds for bundles 
over $\De_1$ with symmetric base facets.
We solve this case using the classification  of mass linear functions on 
$3$-dimensional polytopes given in Proposition 
\ref{prop:3d}.

\MS
\NI {\bf Step 2:} {\it Theorem \ref{thm:4d} holds if  
$\Delta$ has 
more than four asymmetric 
facets.} \MS

\NI {\it Proof.}
By Step 1, we may assume that all the asymmetric facets are pervasive.
Hence, Proposition~\ref{prop:5plus} implies that $\De$ is either
$\De_4$ or a $\De_2$ bundle over $\De_2$.
But $\Delta_4$ has no essential mass linear functions.
Therefore we are in case (a3).  
(In fact, this case 
does not occur for essential $H$; 
see Corollary~\ref{cor:22bundle}.)

The proof of Proposition~\ref{prop:5plus} 
relies on the fact, proved in [I, Corollary A.8],
that if all
the
facets of $\De$ are asymmetric then  
$\De$  is combinatorially equivalent to a product of simplices. 
Since all facets are pervasive, it must 
therefore
be combinatorially 
equivalent to $\Delta_4$ or $\Delta_2 \times \Delta_2$.
Moreover, because $\De$ is smooth, in the latter case
Lemma~\ref{prodsimp} implies that
it is a $\De_2$ bundle over $\De_2$.  
On the other hand, if
$\De$ has at least one symmetric facet $G$, then 
Proposition \ref{prop:symcent} implies that the 
restriction of $H$ to $G$ is mass linear. 
By the $3$-dimensional classification, this  implies that there are 
only a few possibilities for $G$.  
The proof is completed by analyzing these.

\MS

\NI {\bf Step 3:}  {\it Theorem \ref{thm:4d} holds if  
$\Delta$ has  four asymmetric 
facets.} \MS

\NI{\it Proof.}
By Step 1, we may assume that all the asymmetric facets are pervasive.
If their  conormals are linearly dependent, then by 
Proposition \ref{prop:4dep}, 
$\De$ is the blowup of a 
 $\Delta_3$ bundle over $\Delta_1$  by a series
of blowups of the type described in Theorem~\ref{thm:4d}.
If $H$ is essential on $\ov\De$ then we are in case (a1);
if $H$ is inessential on $\ov\De$ then by Proposition~\ref{prop:essblow} we are in the special case
of  (b) 
in which the double 
expansion is along two parallel edges of a quadrilateral $\Tilde\De$
(as in Example \ref{ex:7}). 
On the other hand, 
if their conormals are linearly independent then
we are in case (b) by
Proposition~\ref{prop:4indep}.

In each case, 
the classification is established by 
exploiting the 
 classification of polygons and $3$-dimensional
polytopes with four asymmetric facets to analyze the
set of symmetric facets of $\De$; see for example
Lemmas \ref{le:4comb},  
\ref{le:blowblow},
and
\ref{4comp}. 
\MS

\NI {\bf Step 4:} {\it Completion of the proof.} 
\MS

\NI {\it Proof.} 
It remains to consider the case when
$\Delta$ has fewer than four asymmetric facets.
By Proposition~\ref{prop:2asym}, $\Delta$ must 
have
exactly  three 
asymmetric  facets.  
If
their conormals are linearly dependent, then
Proposition~\ref{prop:3dep} implies that
$\De$ itself is a $\Delta_2$ bundle over a  
polygon and the asymmetric facets are the fiber facets;
hence, we are in case (a3).
If their conormals are linearly independent, then 
  Proposition~\ref{prop:3indep} implies  either that
the triple intersection  $F_{123} $ is empty, we are in case (a2), and the asymmetric facets
 correspond
to the facets of $\De_2$, or
 that
$F_{123}$ is nonempty, we are in case (a1), and three of the four fiber facets 
are asymmetric. 
In all these cases we analyze the structure of $\De$ by exploiting the fact that
the symmetric faces of smallest dimension are $2$-dimensional triangles
with edges $F_1\cap g$, $F_2  \cap g$, and $F_3 \cap g$.

\begin{rmk}\rm
As we explain above, the proof of Theorem~\ref{thm:4d} depends
on the number of asymmetric facets.
In particular, as is shown in Steps 3 and 4, the arguments needed if there are 
four pervasive asymmetric facets are very different from those needed if there 
are three.   
However, these cases are not so distinct as they might seem.
By Corollary~\ref{cor:Mabc}, a generic $\Delta_3$ bundle over $\Delta_1$ 
admits essential mass linear functions with either three or four 
pervasive asymmetric facets.
As the proof above shows,
in case  (a2) and (a3) the polytope necessarily has three  
pervasive asymmetric
facets, and in case (b) it has four. 
\end{rmk}

\subsection{Questions and comments}

 Many results and techniques used in this paper extend to higher dimensions. 
 However, the proof of Theorem~\ref{thm:4d} relies on a
very detailed result  about 
polygons (Lemma~\ref{le:2blowdown}) as well as
the description in Proposition \ref{prop:3d} below of all 
$3$-dimensional polytopes that have nonzero mass linear functions.
Since  higher dimensional 
 polytopes are not yet so well understood,
one cannot expect such a complete classification in higher dimensions. 
Additionally, 
to make further progress with our current methods
we would  first need to answer the following 
question 
since, as explained in  Step 2
of \S\ref{ss:out} above,
 this is the basis of our inductive argument: 
 once one has a symmetric facet  $G$ one can analyze the structure of 
 the mass linear pair $(\De,H)$
 by using information on the lower dimensional pair $(G,H|_G)$.

\begin{quest} \labell{q:1}  If 
$\De$  has a mass linear function such that all facets are asymmetric,
is the equivalence relation on the facets of $\De$ nontrivial?  
\end{quest}

If the answer were yes,  then 
by Lemma \ref{le:ea} 
there would be an inessential $H'$ 
such that $H-H'$  has a symmetric facet.  Further
Proposition \ref{prop:Iexpan}
would imply that $\De$ must be either
 a bundle over a simplex or an expansion.
 In fact, it seems quite likely that in this situation $\De$
 must be 
combinatorially equivalent to a product of simplices, and hence,
by an extension of  [I, Lemma 4.10], an iterated  simplex bundle.
See Corollary~\ref{cor:asym3} for the $4$-dimensional case.

Together with Timorin, we found a purely combinatorial argument 
that showed  Question \ref{q:1} has
a  positive answer in dimensions $\le 4$; see [I, Appendix]. 
As pointed out in \cite[Lemma~2.4]{Mct}, this
argument does not extend to higher dimensions.
However,  Chen  \cite{Chen}
showed that the answer is again positive for  
$5$-dimensional polytopes
with at most $9$ facets.

Question \ref{q:1} seems hard, though very interesting.  An easier task 
would be to analyze properties of particular kinds of polytopes. 
The obvious examples of polytopes  with a mass linear function  for 
which all facets are asymmetric are products of simplices.  As we showed in
  [I, Theorem~1.20] these polytopes have several other interesting
  characterizations: they are the only polytopes for which 
 every $H\in \ft$ is mass linear, and also the only polytopes such
that  for each equivalence class $I$ of facets
 the intersection $F_I: = \cap_{i\in I} F_i$ is empty.  The latter condition implies that $\De$ has no singleton facets, i.e., 
 that $|I|>1$ for all equivalence classes $I$.  
Polytopes with this property are analyzed, though not fully classified, in the proof of \cite[Lemma~3.7]{Mct};
they are a particular kind of expansion.  Here are some questions.

\begin{quest}\labell{q:d-1}   If $\De$ is an expansion 
(but not a bundle over a simplex), can it support a mass linear function 
for which all facets are asymmetric? 
Which polytopes 
 support
 an $(n-1)$-dimensional family of mass linear functions, where $n: =\dim \De$?
\end{quest}

Note that  by  
Propositions 
\ref{prop:lift} 
and 
\ref{DkoverD1} 
every $\De_{1}$-bundle over $\De_{n-1}$
and 
every $\De_{n-1}$ bundle over $\De_1$ 
has an  
$(n-1)$-dimensional family of mass linear functions, while, by 
Corollary~\ref{cor:22bundle},
generic
$\De_2$ bundles over $\De_2$ have only a $2$-dimensional family of  
mass linear functions,  and these
are all inessential.
Moreover,   
Proposition~\ref{DkoverD1}
shows that 
each generic
$\De_{n-1}$ 
bundle
over $\De_1$  
has 
mass linear functions for which every facet is asymmetric, 
while by Proposition 
\ref{prop:asym}
and Corollary~\ref{cor:22bundle}, generic bundles of the other two types must have symmetric fiber facets.
In this paper we do not study   mass linear functions on
$\De_s$ bundles over $\De_{n-s}$ for general $s$.

We show in Corollary \ref{cor:asym2} that 
the only $4$-dimensional polytope 
that has a mass linear function with $8$ asymmetric facets is the product $(\De_1)^4$.
In fact it is easy to see that {\it  a mass linear function on an $n$-dimensional polytope has at most $2n$ asymmetric facets.}   
This holds because by [I, Proposition~A.2] every asymmetric facet
$F$ is powerful, i.e. it is connected to every vertex of 
 $\De\less F$ by an edge.  

\begin{quest}\labell{q:3}  If $H$ is a mass linear function on a smooth $n$-dimensional polytope $\De$ with $2n$ asymmetric facets, must $\De$ be the product $(\De_1)^n$?
\end{quest}

Another interesting question concerns which blowups preserve mass linearity. 
It is  easy to see  that mass linearity is destroyed if one  blows up along a face
$f$ that does not {\em meet} 
all asymmetric facets; see Lemma \ref{le:blowcent}.  
 Lemma \ref{le:inessblow} is another straightforward result 
 showing that if $\De'$ is the blowup of the polytope $\De$ 
along a face $f$ that {\em lies in} 
all asymmetric facets then  every inessential function on $\De$ is inessential on $\De'$.
However, it is not clear whether this condition on $f$  is sufficient for  mass linearity
to be preserved.

\begin{quest}\labell{q:4}  Suppose that $H$ is a mass linear function
on $\De$ with asymmetric facets $F_j, j\in J$. Let $\De'$ be the blowup of $\De$ along the face $f= F_I$ where
$J\subseteq I$.  Is $H$ mass linear on $\De'$?
More generally, describe all blowups  that preserve mass linearity. 
\end{quest}

 In dimension $4$, besides the blowup operations   described in Theorem \ref{thm:4d} two others occur during the classification proof; 
namely, blowing 
up at a vertex or
  edge that meets all asymmetric facets.
Corollary \ref{cor:blowv} shows that the vertex blowup preserves mass linearity in any dimension. 
 However, if one blows up along an edge $e=F_I$, additional 
 conditions are needed.  
 One of these is very natural, namely that $\sum_{i\in I}\ga_i = 0$.
   (By Remark \ref{rmk:sum} this  holds if  $J\subseteq I$ as in Question \ref{q:4}.)   Corollary \ref{cor:blowedge}  shows that in dimension $4$ this extra condition 
 suffices.  Our proof also  suggests  that the natural framework in which to consider the effect of blowing up may not be the set of mass linear functions, but rather the set of fully mass linear functions that we now discuss.

 Our  analysis of the image of $\pi_1(T)$ in $\pi_1\bigr(\Symp(M,\om)\bigr)$ 
 is based on the properties of Weinstein's action homomorphism $\Aa_\om$; 
 see [I, \S5].  In  \cite{Shel}, Shelukhin  defined a series of related
 homomorphisms that allow one to formulate properties, 
in principle stronger than mass linearity, that must be satisfied whenever 
  the loop $\La_H$ generated by an integral $H\in \lat$ 
 contracts in $\Symp(M,\om)$.
These are discussed further in \cite[\S4]{Mct} where we called Shelukhin's conditions {\bf full mass linearity};
see also \S\ref{ss:fullML} below.
 
\begin{quest}\labell{q:ml}  Does every mass linear function 
satisfy Shelukhin's additional conditions?
\end{quest}

The results of this paper imply that the answer is
 yes in dimensions $\le 4$; 
see
Proposition \ref{prop:fullML}.
   However, although the mass linear condition seems to be very strong, it is not clear whether it is equivalent to full mass linearity
  in higher dimensions.
  If not, many of the above questions might be better investigated for
  fully mass linear functions.   
\MS

\NI {\bf Organization of the paper.}\,\,
Section 
\ref{sec:con} begins with a 
review of the results from Part I that we use most often.  It then
describes in detail some general ways to construct 
polytopes, namely bundles, 
expansions, and 
blowing up and down.
In each case, we describe the behavior of mass linear functions under these operations.  We also develop criteria
for recognizing when a 
facet can be blown down  
(Lemma~\ref{le:blowtech}) and for recognizing when a polytope is a double expansion (Lemma \ref{lemma:dexpan}).

In section \ref{sec:ex}, we construct all the (essential) mass
linear functions on the polytopes described in Theorem~\ref{thm:4d}.
\S\ref{ss:calcul}  gives detailed information on the three kinds of 
bundles 
in case (a) of Theorem \ref{thm:4d}, while \S\ref{ss:blowex} 
discusses double expansions, 
showing precisely how blowing up an inessential function on a 
 double expansion can convert it into an essential function.   

Section \ref{sec:4dim} 
finishes the proof of Theorem \ref{thm:4d},
thus showing that our list of examples is complete. 
 \S\ref{ss:3perv} deals with the case when there are three asymmetric facets, 
 and
  \S\ref{ss:4asym} with the case of four asymmetric
   and pervasive 
   facets.  These arguments are quite different, because by Proposition 
  \ref{prop:symcent} the symmetric $2$-faces are triangles in the first case and are rectangles in the second. 
 The final subsection \S\ref{s:more4}
 discusses the case when there are more than four asymmetric facets. 

   The last section contains a variety of further results.
  \S\ref{ss:min} explains  exactly when the polytopes $\ov\De$ in Theorem 
\ref{thm:4d} are minimal. In \S\ref{ss:fullML} we use Theorem
\ref{thm:4d} to show that in dimensions $\le 4$ every mass linear function is fully mass linear.  Finally, in 
\S\ref{ss:blml} we consider the question of which  blowup operations preserve mass linearity. 
\MS
 
\NI {\bf Acknowledgements.}  
Both authors are very grateful 
to MSRI for its hospitality in Spring  2010; the first author also
thanks the Simons Foundation for its support via an Eisenbud Professorship.

\section{Constructions}\labell{sec:con}

After a review of basic results, this section describes in
detail some general ways to construct 
polytopes: 
bundles, expansions, blowups, and blowdowns.  
We also
analyze certain natural mass linear functions
on each type of polytope.

\subsection{Review of basic results}\labell{ss:review}

For the convenience of the reader we 
begin by assembling
the  results from Part I that
will be used most often in this paper;
in particular, we describe all 
smooth 
polytopes of dimension at most three
that have mass linear functions.
In the process, we give the definition of a bundle.
Many of the results quoted below are valid for simple polytopes; 
however we restrict to the smooth case  
for simplicity.
Thus, even if it is not stated explicitly, we assume that 
every polytope
is smooth.

We begin by noting that the definition of mass linearity given in \S\ref{ss:intro} is slightly different 
from, but equivalent to, the definition used in Part I.
Given a 
smooth 
polytope $\Delta \subset \ft^*$,
the {\bf chamber} $\Cc_\De$ of $\Delta := \Delta(\kappa)$ is the connected component that contains $\kappa$ of
the set of all $\kappa' \in \R^N$ such that $\Delta(\ka')$ is 
smooth.
Note that, for every $\kappa' \in \Cc_\Delta$, the polytope $\Delta(\ka')$ is 
{\bf analogous}
 to $\Delta$,
that is, for all $I \subset \{1,\dots,N\}$  
the intersection 
$\bigcap_{i \in I} F_i'$ is empty exactly if the intersection $\bigcap_{i \in I} F_i$ is empty.
In \S\ref{ss:intro}, we gave a local definition of mass linearity, that is,
we only required the function $\Hat H$ which takes $\ka'$ to 
$\langle H,c_\De(\ka')\rangle$  to be linear
for $\ka'$ in some open neighborhood of $\ka$.  In contrast, in Part I we
required the function to be linear on all of $\Cc_\De$.
However, as we show in 
[I, Lemma 2.3], 
these two definitions are equivalent because $\Hat H$ is always a rational function.
A similar remark applies to the definition of equivalent facets; see [I, Corollary 3.5].

One
extremely useful fact -- 
which follows quite easily from the definitions --
is that inessential functions are mass linear.

\begin{proposition}[I, Proposition 1.18] 
\labell{prop:inessential}
Fix $H \in \ft$ and a
polytope 
$\Delta \subset \ft^*$.
Let $\Ii$ denote the set of equivalence
classes of facets 
of $\Delta$.
If  $H$ is inessential, 
write
$$
H =  \sum \be_i \eta_i, \quad \mbox{where} \quad
\be_i\in \R \  \forall \, i \quad \mbox{and}
\quad \sum_{i \in I}  \be_i =  0 \quad \forall \, I \in \Ii. 
$$
Then
$$
\langle H, c_{\De}(\ka) \rangle = \sum \be_i \kappa_i.
$$
\end{proposition}

It is 
straightforward to use the formula 
above
to show that  we can reduce the number of asymmetric facets if some of them are equivalent.

\begin{lemma} [I, Lemma 3.19] 
\labell{le:ea}
Let $H \in \ft$ be a mass linear function on a
polytope
$\Delta \subset \ft^*$. 
If $F_1,\ldots,F_m$ are equivalent facets,
there exists an inessential function $H' \in \ft$ 
so that the mass linear function
$\Tilde{H} = H - H'$  
has the following properties:
\begin{itemize}
\item 
For all $i < m$,
the facet $F_i$ is  $\Tilde{H}$-symmetric. 
\item   For all $i > m$,  
the facet $F_i$ is $\Tilde{H}$-symmetric iff it is $H$-symmetric. 
\end{itemize}
\end{lemma}

Note that, in general,  even
if $H=\sum_i\be_i\eta_i$  is mass linear,
$\langle H, c_\Delta(\kappa) \rangle $ need not equal $ \sum \beta_i \kappa_i$
since the 
 $\be_i$ are not uniquely determined by $H$.
In contrast, the next lemma shows that the coefficients of a mass linear function $H \in \ft$ always determine
the function $H$ itself.

\begin{lemma}[I, Lemma 2.6] 
\labell{le:Hsum}
Fix $H \in \ft$ and 
a polytope
$\Delta \subset \ft^*$. 
If
$\langle H, c_\Delta(\kappa) \rangle = \sum \beta_i \kappa_i,$
then $$H = \sum \beta_i \eta_i.$$
\end{lemma}

In part I, the proof of the lemma above relies on the following 
fact.

\begin{rmk}\labell{translate}\rm
Given 
a polytope 
$\Delta = \bigcap_{i=1}^N \{ x \in \ft^* \mid \langle \eta_i,x \rangle \leq \kappa_i \}  \subset \ft^*$
and  $\xi \in \ft^*$, consider $\Delta' = \Delta + \xi$, the
translate of $\Delta$ by $\xi$.  
Then
\begin{equation*}
\Delta' = \bigcap_{i=1}^N \{ x \in \ft^* \mid \langle \eta_i,x \rangle \leq \kappa_i' \}, 
\quad \mbox{where} \quad \kappa_i' = \kappa_i +\langle \eta_i,\xi \rangle \quad \forall \ i.
\end{equation*}
\end{rmk}

We continue 
with  some very useful 
results about symmetric faces, taken from 
[I, \S 2.3].
They imply that if 
$H$ is a mass linear function on $\De$ and
$(\De,H)$ has a symmetric face $f$ then the pair
$(f,H|_f)$ is also  mass linear.  
(Here, we consider the face $f$ as a polytope in $P(f) \subset \ft^*$, the smallest
affine plane containing  $f$.)
Hence one can use knowledge of the structure of the lower dimensional $(f,H|_f)$ to analyse $(\De,H)$.

\begin{prop}
[I, Lemma 2.7  and Proposition 2.9]
\labell{prop:symcent} 
Let $H \in \ft$ be a mass linear function on 
a  polytope 
$\Delta \subset \ft^*$.
Let $f$ be a symmetric face of $\Delta$. 
Then the following hold:
\begin{itemize}
\item 
$
\langle H,c_f(\kappa) \rangle  = \langle H, c_{\Delta}(\kappa) \rangle $ for all $ \kappa \in \Cc_\Delta,
$
where $c_f$ denotes the center of mass of $f$ in $P(f)$.
\item  
The restriction of $H$ to $f$ is  mass linear.
\item  Intersection induces a one-to-one
correspondence between the asymmetric facets of $\Delta$
and the asymmetric facets of $f$.
\item The coefficient of the support number of a facet $F$ in $\langle H,
c_\Delta \rangle$ is the coefficient of the support number of $f \cap F$
in $\langle H,c_f \rangle$.
\end{itemize}
\end{prop}

\begin{rmk}
\labell{rmk:symcent} \rm  
In fact, it is not hard to prove  the following slightly stronger claims:
\begin{itemize}
\item  If $H$ is inessential on $\De$ then it is inessential on $f$ (but not conversely).
\item  $H$ is mass linear on $\Delta$ exactly if
the restriction of $H$ to $f$ is  mass linear.
\end{itemize}
\end{rmk}

On the other hand, asymmetric facets 
have special properties.
Recall that a 
facet $F$ of 
a  polytope
 $\Delta \subset \ft^*$ is called  {\bf pervasive}
if it has nonempty intersection with every other facet  of $\Delta$. 
Further, we say that $F$ is {\bf flat} if there is a hyperplane in $\ft$ that contains
the conormal of every other facet (other than $F$ itself) that meets $F$.

\begin{prop} [I, Proposition 2.11]  
\labell{prop:asym}
Let $H \in \ft$ be a mass linear function
on  
a  polytope
 $\De \subset \ft^*$.
Then every asymmetric facet is 
pervasive or  flat (or both).
\end{prop}

The following statement combines [I, Lemma 2.13] 
with results from [I, \S4.2].

\begin{prop}\labell{prop:2asym}
Let $H \in \ft$ be a 
nonzero
mass linear function on 
a  polytope 
$\Delta \subset \ft^*$. 
Then
 $\Delta$ has  at least two asymmetric facets.    
Moreover, if it has exactly two,
then they are equivalent and $H$ is inessential. 
\end{prop}
Next, we give another characterization of 
the equivalence relation on the facets.

\begin{lemma}[I, Lemma 3.7]\labell{le:equiv}
Let $\Delta \subset \ft^*$ 
be a 
smooth  polytope.
Given a subset $I \subset \{1,\ldots,N\}$,
we have
$F_i \sim F_j$ for all $i$ and $j$ in $I$
exactly if the plane 
$V\subset \ft$ spanned
by the outward conormals $\eta_k$ for $k \not\in I$
has codimension $|I| - 1$. 
Moreover, in this case 
the linear combination
$\sum_{i \in I} c_i \eta_i$ lies in $V$ 
if and only if $c_i = c_j$ for all $i$ and $j$.
\end{lemma}

\begin{rmk}\labell{rmk:equiv}\rm   
In particular,  $F_i\sim F_j$ exactly if there is a vector $\xi\in \ft^*$ that is parallel to all the facets except $F_i$ and $F_j$.
This condition is very easily tested. The vector $\xi$ is preserved by the corresponding  reflection symmetry; 
in Masuda  \cite{Mas} it is called a root. 
Since $\Delta$ is smooth, there
is also a homological interpretation of this equivalence relation:  by [I, Remark~5.8], $F_i\sim F_j$ exactly if
the corresponding submanifolds $\Phi^{-1}(F_i)$ and $\Phi^{-1}(F_j)$ are homologous
in $H_{2n-2}(M_\De)$.
\end{rmk}

Our next aim is to describe all the
$2$- and $3$-dimensional polytopes that have mass linear functions.  
 Since many of these polytopes are bundles, we start with the formal definition.   
Two polytopes $\Delta$ and $\Delta'$ are said to be {\bf combinatorially
equivalent} if there exists a bijection of facets $F_i \leftrightarrow F_i'$
so that $\cap_{i \in  I} F_i \neq \emptyset$ exactly if $\cap_{i \in I} 
F_i' \neq \emptyset.$

\begin{defn}\labell{def:bund} 
Let $\Tilde{\Delta} =
\bigcap_{j = 1}^{\Tilde{N}} \{ x \in \Tilde{\ft}^* \mid \langle \Tilde{\eta}_j,x
\rangle \leq \Tilde{\kappa}_j \}$ and
 $\Hat{\Delta} =
\bigcap_{i = 1}^{\Hat{N}} \{ y \in \Hat{\ft}^* \mid \langle \Hat{\eta}_i,y
\rangle \leq \Hat{\kappa}_i \}$ be 
smooth
 polytopes.
We say that a 
smooth
polytope 
$\Delta \subset \ft^*$
is a {\bf bundle} with {\bf fiber} $\Tilde{\Delta}$ over the {\bf
base} $\Hat{\Delta}$ if 
there exists  a short exact sequence
$$ 
0 \to \Tilde{\ft} \stackrel{\iota}{\to} \ft \stackrel{\pi}{\to} \Hat{\ft} \to 0
$$
so that
the following hold:
\begin{itemize}
\item 
$\Delta$ is combinatorially equivalent to the product $\Tilde{\Delta}
\times \Hat{\Delta}$.
\item 
If $\Tilde{\eta}_j\,\!'$   
 denotes the outward conormal to
the  facet 
$\Tilde{F}_j\,\!'$
of $\Delta$ which corresponds to  $\Tilde{F}_j \times
\Hat{\Delta} \subset \Tilde{\Delta} \times \Hat{\Delta}$,
then $\Tilde{\eta}_j\,\!' = \iota(\Tilde{\eta}_j)$ 
for all $1 \leq j \leq \Tilde{N}.$
\item 
If $\Hat{\eta}_i\,\!'$  
denotes the outward conormal to
the  facet 
$\Hat F_i\,\!'$
of $\Delta$ which corresponds to  $\Tilde{\Delta} \times
\Hat{F_i} \subset \Tilde{\Delta} \times \Hat{\Delta}$,
then $\pi(\Hat{\eta}_i\,\!') = \Hat{\eta}_i$
for all $1 \leq i \leq \Hat{N}.$
\end{itemize}
 The facets 
$\Tilde{F}_1\,\!', \ldots, \Tilde{F}_{\Tilde{N}}\,\!'$ 
will be called {\bf fiber facets},  
and the facets 
$\Hat F_1\,\!' \ldots, \Hat F_{\Hat N} \,\!'$
will be called {\bf base facets}.
\end{defn}

Observe  that if $\De$ is such a bundle
then 
the faces $\Hat F_{I}': = \cap_{i\in I} \Hat F_i'$ of $\De$ corresponding
to 
the vertices $\Hat{F}_{I}$  of the base $\Hat\Delta$  are all
affine equivalent.
In contrast,
the faces $\Tilde F_J\,\!'$ of $\De$
corresponding to the vertices  $\Tilde F_J$ of the fiber $\Tilde\De$ 
may not be affine equivalent, but they are {\bf analogous}, that is,
we may identify 
the affine plane
$P(\Tilde F_J\,\!')$
 with $\Hat\ft^*$ so that there is  combinatorial equivalence 
between  $\Tilde F_J\,\!'$
and $\Hat\De$ in which corresponding facets are parallel; see [I, \S1.1].

To help the reader understand this rather complicated definition, here is
a recognition lemma,
which explains how to identify a given polytope $\De$ as a bundle with fiber $\Tilde\De$ and base $\Hat \De$.  
The proof is elementary and is left to the reader.

\begin{lemma}  \labell{le:recog}
A 
smooth polytope  $\Delta$ is a
bundle over $\Hat{\Delta}$ with fiber $\Tilde{\Delta}$ exactly if all 
the following conditions hold:
\begin{itemize}
\item $\De$ is combinatorially equivalent to the product
$\Tilde{\Delta}\times \Hat{\Delta}$.
\item The conormals $\Tilde{\eta}_j\,\!'$ to the fiber facets $\Tilde F_j\,\!'$ lie in a dim $\Tilde\Delta$ subspace.
\item There is a vertex $\Hat{F}_{I}$  of $\Hat\Delta$
so that  the face $\Hat F_{I}'$ of $\De$ is analogous to $\Tilde\Delta$.
\item There is a vertex  $\Tilde{F}_{J}$  of $\Tilde\Delta$
so that the face $\Tilde F_{J}'$ of $\De$ is analogous to $\Hat\Delta$. 
\end{itemize}
\end{lemma}

Later, we will also need the following result from Part I; it  explains why the smooth case is easier than the general one.

\begin{lemma}[I, Lemma 4.10] \labell{prodsimp}
Let $\Delta$ be a smooth polytope which is combinatorially
equivalent to $\Delta_k \times \Delta_n$.
Then $\Delta$ is either a $\Delta_k$ bundle over $\Delta_n$,
or  a $\Delta_n$ bundle over $\Delta_k$.
\end{lemma}

Finally here are some detailed results about 
mass linear pairs in dimensions $2$ and $3$.

\begin{prop}[I, Proposition 4.2 and Corollary 4.3]\labell{prop:2dim}  
Let $H \in \ft$ be a nonzero mass linear function on a 
smooth
polygon $\De \subset \ft^*$.
Then one of the following statements holds:
\begin{itemize}
\item $\Delta$ is the simplex $\Delta_2$; at most one edge is symmetric.
\item $\Delta$ is a $\Delta_1$ bundle over $\Delta_1$; the base facets are the asymmetric edges.
\item $\Delta$ is the product $\Delta_1 \times \Delta_1$; each edge is asymmetric.
\end{itemize}
In any case, $H$ is inessential.
Moreover, if two edges $F_i$ and $F_j$ do not intersect
then $\gamma_i + \gamma_j = 0$, where $\gamma_k$ is the coefficient
of the support number of $F_k$ in the linear function $\langle H, c_\Delta \rangle$.
\end{prop}

\begin{prop}[I, Theorem 1.4, 
Proposition 4.14, and Lemma 4.15]\labell{prop:3d}
Let $H \in \ft$ be a mass linear function on a smooth
$3$-dimensional polytope $\Delta \subset \ft^*$.
If $\Delta$ has more than two asymmetric facets,
then one of the following statements holds:
\begin{itemize}
\item $\Delta$ is the simplex $\Delta_3$.
\item $\Delta$ is a $\Delta_1$ bundle over $\Delta_2$; the base facets
are the asymmetric facets.
\item $\Delta$ is a $\Delta_2$ bundle over $\Delta_1$; if either base facet
is asymmetric then both are.
\item $\Delta$ is a $\Delta_1$ bundle over $\Delta_1 \times \Delta_1$;
the base facets are the asymmetric facets.
\item $\Delta$ is the product $\Delta_1 \times \Delta_1 \times \Delta_1$;
every facet is asymmetric.
\end{itemize}
Moreover, $H$ is inessential unless
$\Delta$ is a $\Delta_2$ bundle over $\Delta_1$.
Finally,  $$\sum_{i=1}^N \gamma_i = 0,$$
where $\gamma_i$ is the coefficient of the support number of the
facet $F_i$ in the linear function $\langle H, c_\Delta \rangle.$
\end{prop}

\subsection{Bundles}\labell{ss:bund}

In this subsection, we give a new way to construct 
mass linear functions on  bundles.
More precisely, we show that there is a one-to-one
correspondence between (essential) mass linear functions on the
base and  (essential) mass linear functions 
with symmetric fiber facets on the bundle.
Let $\Delta \subset \ft^*$ be a $\Tilde\De \subset \Tilde{\ft}^*$
bundle over 
$\Hat\De \subset \Hat{\ft}^*$
and  consider  the associated short exact sequences
\begin{equation}\labell{short}
0 \to \Tilde\ft \stackrel{\iota}{\to}  \ft \stackrel{\pi}{\to} \Hat\ft \to 0
 \quad \mbox{and} \quad 
0 \to \Hat\ft^* \stackrel{\pi^*}{\to} 
 \ft^* \stackrel{\iota^*}{\to} \Tilde\ft^* \to 0.
\end{equation}
We begin our discussion with the following elementary but useful lemma.

\begin{lemma}\labell{le:bund}
Let $\Delta \subset \ft^*$ be a $\Tilde \Delta \subset \Tilde\ft^*$
bundle over $\Hat\De \subset \Hat\ft^*$. 
\begin{itemize}\item[(i)]
Two base facets $\Hat F_i\,\!'$ and $\Hat F_j\,\!'$ of $\Delta$ are equivalent
exactly if the corresponding facets $\Hat F_i$ and $\Hat F_k$ of $\Hat \Delta$ 
are equivalent. 
\item[(ii)]  If two fiber facets of $\Delta$ are equivalent then the 
corresponding facets of the fiber are 
equivalent;  the converse need not hold.
\item[(iii)] A base facet 
$\Hat F_i\,\!'$ of $\Delta$ is never equivalent to
a fiber facet  $\Tilde F_j\,\!'$.
\end{itemize}
\end{lemma}
\begin{proof}
To prove (i),
note that \eqref{short} implies that 
the image of $\pi^* \colon \Hat{\ft}^* \to \ft^*$ 
is the annihilator  of $\iota(\Tilde \ft) \subset \ft$; 
moreover,
$$
\langle \Hat\eta_i\,\!', \pi^*(\Hat\xi) \rangle 
 = \langle \pi(\Hat\eta_i\,\!'), \Hat\xi \rangle 
 = \langle \Hat\eta_i,\Hat\xi \rangle 
\qquad \forall \ \Hat\xi \in \Hat\ft \ \mbox{and} \ \forall  \ i.
 $$ 
Hence, the claim follows from Lemma~\ref{le:equiv}.

The first part of (ii) is easy.  To illustrate  the second, 
consider the  $\De_2$ bundle over $\De_1$ with polytope
$Y$ as in Equation~\eqref{eq:Yb}. 
The fiber facets are not all equivalent unless $a_1 = a_2 = 0$.

To prove (iii),
note that since $\Hat \Delta$ is compact,
the outward conormals to all but one facet of $\Hat \Delta$
still span $\Hat \ft$.  Since the same
holds for $\Tilde \Delta$, the claim follows from
\eqref{short} and Lemma~\ref{le:equiv}.
\end{proof}

We are now ready -- after one last definition -- to state 
the main result of this subsection.

\begin{definition}\labell{def:lift}
Let $\Hat H \in \Hat \ft$ be a  mass linear function on $\Hat \Delta$;
write $\langle \Hat H, c_{\Hat \Delta} \rangle = \sum \beta_i \Hat{\kappa}_i$.
The 
{\bf lift} of $\Hat H$ to $\Delta$ is
$$H =  \sum \beta_i 
\Hat\eta_i\,\!' \in \ft.$$ 
\end{definition}

Note that, by
Lemma~\ref{le:Hsum}, $\Hat H = \sum \beta_i \Hat\eta_i$;
hence, $\pi(H) = \Hat H$.

\begin{prop}\labell{prop:lift}
Let $\Delta \subset \ft^*$ be a $\Tilde\De \subset \Tilde{\ft}^*$ bundle over 
$\Hat\De \subset \Hat{\ft}^*$.
Then the following  hold: 
\begin{itemize}
\item[(i)]
If $H \in \ft$ is mass linear on $\De$ and the fiber facets
are symmetric, then $\Hat H = \pi(H)$ is mass linear on $\Hat\De$.
More specifically, 
if $\langle H, c_\Delta \rangle = \sum \beta_i \Hat\kappa_i\,\!'$,
then 
$\langle \Hat H, c_{\Hat\Delta} \rangle = \sum \beta_i \Hat\kappa_i.$
\item[(ii)] Conversely, 
if $\Hat H \in \Hat\ft$ is  mass linear on $\Hat \De$,
then the 
lift of $\Hat H$ to $\Delta$ is  
mass linear on $\Delta$ and the fiber facets are symmetric.
\item[(iii)]  
In the cases described above, $H$ is inessential
on $\Delta$ exactly if $\Hat H$ is inessential on $\Hat \Delta$.
\end{itemize}
In particular, $\pi$ induces a one-to-one correspondence between
essential mass linear functions on $\Delta$ with symmetric
fiber facets and essential mass linear functions on 
$\Hat \Delta$.
\end{prop}

\begin{figure}[htbp] 
   \centering
\includegraphics[width=3in]{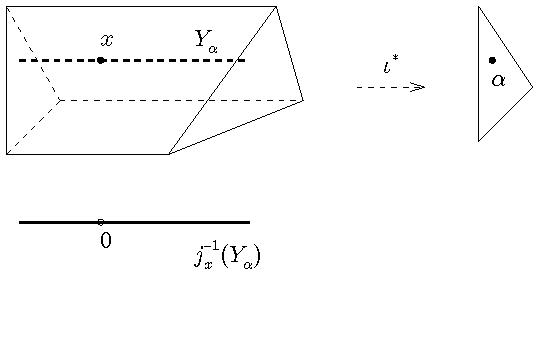} 
   \caption{Slicing  $\De$ by the \lq\lq sections" $Y_{\alpha}$}
   \label{fig:5}
\end{figure}
\begin{proof}
Fix any $H \in \ft$ and let $\Hat H = \pi(H)$.
Given $\alpha \in \iota^*(\Delta) \subset \Tilde\ft^*$,  
consider the slice $Y_\alpha$ of $\Delta$ defined by
$$Y_\alpha := \{ y \in \Delta \mid \iota^*(y) = \alpha\}.
$$
The $Y_\alpha$'s form a  family of parallel polytopes that are analogous
to the base polytope;
see Figure~\ref{fig:5}.
More precisely,  
fix $x \in (\iota^*)^{-1}(\alpha)  \subset \ft^*$ and define an isomorphism
$j_x \colon  \Hat{\ft}^* \to (\iota^*)^{-1}(\alpha)$ by $j_x(y) = \pi^*(y) + x$.
Since $Y_\alpha = 
\bigcap_{i = 1}^{\Hat{N}} \left\{ z \in \ft^* 
\left|\  \langle  \Hat{\eta}_i\,\!', z
\rangle \leq \Hat{\kappa}_i\,\!' \right. \right\} 
\cap (\iota^*)^{-1}(\alpha),$
$$
\Hat\De_x: = j_x^{-1}( Y_\alpha) = 
\bigcap_{i = 1}^{\Hat{N}} 
\left\{ y \in \Hat{\ft}^* \left| \ \langle \Hat{\eta}_i, y
\rangle \leq \Hat{\kappa}_i^x \right. \right\}, \quad \mbox{where} \quad 
 \Hat{\kappa}_i^x = \Hat{\kappa}_i\,\!' - \langle \Hat{\eta}_i\,\!', x \rangle.
$$
By definition, this polytope $\Hat\De_x$ is analogous to $\Hat\Delta$,
that is, its support numbers 
$\Hat\ka^x: =(\Hat\ka_i^x)$ lie
in the chamber $\Cc_{\Hat\De}$. 
Hence, $\langle \Hat H, c_{\Hat\Delta_x} \rangle = 
\langle \Hat H, c_{\Hat\Delta}(\Hat\ka^x) \rangle$.
Therefore,
\begin{equation}\label{relate}
\langle H, c_{Y_\alpha}(\Hat\kappa\,\!') \rangle =
\langle \Hat{H}, c_{\Hat\Delta}(\Hat\kappa^x) \rangle + \langle H,x \rangle,
 \quad \mbox{where} \quad 
\Hat{\kappa}_i^x = \Hat{\kappa}_i\,\!' - \langle \Hat{\eta}_i\,\!', x \rangle.
\end{equation}

First, 
assume that $H \in \ft$ is mass linear on $\Delta$
and that the fiber facets are symmetric; 
write $\langle  H, c_{\Delta} \rangle = \sum \beta_i 
\Hat{\kappa}_i\,\!'$.
Then  $H = \sum \beta_i \Hat{\eta}_i\,\!'$ by Lemma~\ref{le:Hsum}, and so 
$ \langle H,x \rangle = \sum\be_i \langle 
\Hat{\eta}_i\,\!',x \rangle$.
Choose
 $\alpha$ to 
 be a vertex of the polytope $\iota^*(\Delta)$, which is analogous
to  
the fiber
$\Tilde\Delta$.  Then $Y_\alpha$ is the intersection of
the corresponding fiber facets, and hence is a symmetric face of $\Delta$.
Thus, by Proposition~\ref{prop:symcent},  
$\langle H, c_{Y_\alpha}(\Hat\kappa\,\!') \rangle = 
\sum \beta_i \Hat\kappa_i\,\!'$.
Hence, substituting in Equation~(\ref{relate}) we find that 
$$
\langle \Hat H, c_{\Hat \Delta}(\Hat \kappa^x) \rangle = 
\sum \beta_i \Hat{\kappa}^x_i.
$$
This proves (i).

Conversely, assume that $\Hat H \in \Hat \ft$ is mass linear on  
$\Hat \Delta$;  write 
$\langle \Hat H, c_{\Hat \Delta} \rangle = \sum \beta_i \Hat{\kappa}_i$.  
Let $H =  \sum \beta_i \Hat\eta_i\,\!'$ be the lift of $\Hat H$.
By Lemma~\ref{le:Hsum}, $\Hat H = \sum \beta_i \Hat\eta_i$, and so $\pi(H) = \Hat H$.
Therefore, Equation~(\ref{relate}) implies that 
for all $\alpha \in \iota^*(\Delta)$,
$\langle  H, c_{Y_\alpha}(\Hat\kappa\,\!') \rangle = \sum \beta_i \Hat{\kappa}_i\,\!'$.
Since $\Delta$ is the union of such $Y_\alpha$, this immediately implies
that 
$\langle  H, c_{\Delta} \rangle = 
\sum \beta_i \Hat{\kappa}_i\,\!'$.
This proves (ii). .

Finally, (iii) 
follows immediately from 
Proposition~\ref{prop:inessential} and
Lemma \ref{le:bund}.
\end{proof}

In particular, every bundle over a simplex has many inessential functions
with symmetric fiber facets.
In [I, \S3.3]
we used this fact to
prove the following result\footnote{Proposition 3.22 in [I] has a
slightly different statement, but
its proof clearly establishes this stronger claim.}.

\begin{prop}[I, Proposition 3.22] 
\labell{prop:bund}  
Let $H \in \ft$ be a mass linear function on a 
polytope
$\Delta \subset \ft^*$ which is
a bundle over the 
simplex $\De_k$. 
Then we can write $H = H' + \Tilde H$, where
\begin{itemize}
\item $H'$ is inessential and the fiber facets are $H'$-symmetric, and
\item $\Tilde H$ is mass linear and the base facets are $\Tilde{H}$-symmetric.
\end{itemize}
\end{prop}

Part (i) of the next proposition is Corollary 3.24 from Part I.
The second part then follows easily from the proposition above,
just as in the proof of Proposition 3.25 in Part I.

\begin{prop}\labell{prop:flat}
Let $H \in \ft$ be a mass linear function
on 
a  polytope 
$\Delta \subset \ft^*$.
\begin{itemize}\item[(i)] 
If $F$ is an asymmetric facet that is not pervasive, then $\De$ is an $F$ bundle over $\De_1$.

\item[(ii)] 
We can write $H = H' + \Tilde H$, where
\begin{itemize}
\item[$\bullet$] $H'$ is inessential and the pervasive facets are $H'$-symmetric, and
\item[$\bullet$] $\Tilde H$ is mass linear and the nonpervasive facets are $\Tilde H$-symmetric.
\end{itemize}
\end{itemize}
\end{prop}

The next lemma explores what happens when we assume that the base facets are symmetric; 
cf.\ Proposition~\ref{prop:lift}.
We will not need it in this paper.

\begin{lemma}\labell{le:fibbun} 
Let $\De \subset \ft^*$ be a  $\Tilde \De \subset \Tilde \ft^*$ 
bundle
over $\Hat \Delta \subset \Hat \ft^*$.
Then the following hold:
\begin{itemize}
\item[(i)]
If  $H \in \ft$ is a  mass linear function on
 $\Delta$ and
the base facets are symmetric, then $H \in \iota(\Tilde\ft)$
and $\Tilde{H} = \iota^{-1}(H)$ is mass linear on $\Tilde{\Delta}$.
More specifically, 
if $\langle H, c_\Delta \rangle = \sum \beta_i \Tilde\kappa_i\,\!'$,
then 
$\langle \Tilde H, c_{\Tilde\Delta} \rangle = 
\sum \beta_i \Tilde\kappa_i.$

\item[(ii)]  In contrast, even if $\Tilde H \in \Tilde \ft$ is  mass linear (and inessential) on $\Tilde \De$, $H = \io(\Tilde H)$ may
not be mass linear on $\De$.

\item [(iii)]
In case (i) above, if $H$ is inessential on $\De$ then  $\Tilde H$
is inessential on $\Tilde \De$.
\end{itemize}
\end{lemma}

\begin{proof}
To prove (i), first note  that
by Lemma~\ref{le:Hsum},  $H$ lies in the span of the fiber facets,
that is, 
$H = \iota(\Tilde H)$ for some $\Tilde H \in 
\Tilde \ft$.
Moreover, let $f$ be  the face
formed by intersecting any $k = \dim \Hat\Delta$ base facets.
Then, under the natural identification (as affine spaces) of $P(f)$ with $\Tilde \ft^*$,
$f$ is analogous to $\Tilde \De$ and $H$ restricts to $\Tilde H$.
Since $f$ is symmetric, the first claim 
 now follows
from   Proposition~\ref{prop:symcent}.

To prove (ii)
let $\De$ be a nontrivial $\De_1$-bundle over 
some base polytope $\Hat\De$.
Every nonzero element $\Tilde H \in \Tilde \ft$ is mass linear (and inessential) on $\Delta_1$.
So assume that $H = \iota(\Tilde H) \in \ft$ is mass linear on $\De$.
By Proposition~\ref{prop:2asym}, $\De$ has at least two $H$-asymmetric facets.
On the other hand, let $F$ be a fiber 
facet.
Since the bundle is not trivial, $F$ is not flat,
and so Proposition~\ref{prop:asym} implies that $F$ is symmetric.
Therefore, by Proposition~\ref{prop:symcent}, the restriction of $H$ to $F$
is a mass linear function with 
at least two 
asymmetric facets; in particular, 
the restriction of $H$ to $F$ is not constant.
But this is impossible because $H = \iota(\Tilde H)$ is constant on $F$ by construction.

Note finally that if 
$H = \iota(\Tilde H)$ is inessential and the base facets are symmetric, then
Proposition~\ref{prop:inessential} and 
Lemma \ref{le:bund} 
imply that $\Tilde H$ is inessential on $\Tilde \Delta$; this proves (iii).
\end{proof}

\subsection{Expansions}\labell{ss:expan}

We now describe  
a class of polytopes -- $k$-fold expansions\footnote
{
In the combinatorial literature  this construction is known as a wedge; 
cf.\ Haase and Melnikov \cite{HM}. }
 -- which have 
inessential mass linear functions.   
Overall, these polytopes are
very similar to bundles over the simplex $\De_k$, 
except that in this case the ``base" facets
all intersect.   
As we proved in Part I,
these two classes of polytopes are
the only ones which admit nonzero inessential functions.

\begin{defn}\labell{def:expand}
Let $\Tilde{\Delta} =
\bigcap_{j = 1}^{\Tilde{N}} \{ x \in \Tilde{\ft}^* \mid \langle  
\Tilde{\eta}_j, x
\rangle \leq \Tilde{\kappa}_j \}$ be a 
smooth
polytope. 
Given a natural number $k$,
a polytope $\Delta \subset \ft^*$ is the {\bf $\mathbf k$-fold expansion}
of $\Tilde{\Delta}$ along the facet $\Tilde{F}_1$ if there
is an identification
$\ft = \Tilde{\ft} \oplus \R^k$ so that
\begin{gather*}
\Delta =
\bigcap_{j = 2}^{\Tilde{N}} \{ x \in {\ft}^* \mid 
\langle (\Tilde{\eta}_j,0),x \rangle \leq \Tilde{\kappa}_j\} 
\cap\bigcap_{i = 1}^{k+1} \{ x \in {\ft}^* \mid \langle  \Hat\eta_i,x
\rangle \leq \Hat\kappa_i \},
\quad \mbox{where} \\
\Hat\eta_i = (0, - e_i)  \mbox{  and  }\Hat\kappa_i = 0 \mbox{  for all  }   1 \leq i \leq k,
\quad \Hat\eta_{k+1} = (\Tilde{\eta}_1, \sum e_i) \quad \mbox{and} \quad
\Hat\kappa_{k+1} = \Tilde{\kappa}_1.
\end{gather*}
We shall call the facet $\Tilde{F}_j\,\!'$ of $\Delta$ with  outward conormal 
$(\Tilde{\eta}_j,0)$ the  {\bf fiber-type facet (associated to 
$\mathbf{ \Tilde{F}_j} $)  }
for all $j > 1$
and the facets 
$\Hat{F}_i$ with outward conormals 
$\Hat\eta_i$ 
the  {\bf base-type facets}.
\end{defn}

\begin{figure}[htbp] 
   \centering
 \includegraphics[width=4in]{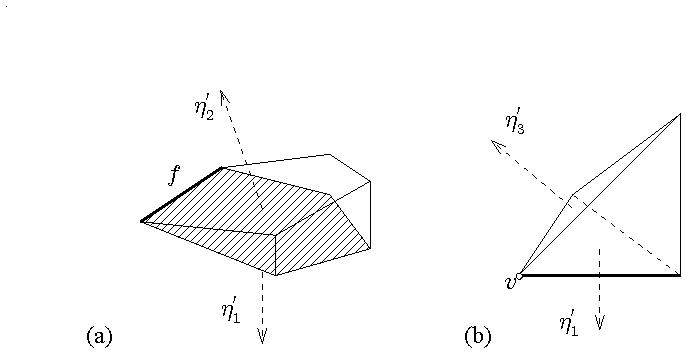} 
   \caption{(a) is the $1$-fold expansion of the shaded polygon along $f$; (b) is the $2$-fold expansion of the heavy line at the vertex $v$}
   \labell{fig:example}
\end{figure}

It is easy to check that $\De$ is 
smooth.

\begin{rmk}\labell{rmk:expint}\rm
(i)
The base-type facets are pervasive; in fact,
the face $\bigcap_{i \neq n} \Hat{F}_i$
can be identified with $\Tilde\Delta$
for all $n \in \{1,\dots,k+1\}$.  
Similarly,  
the face
$ \bigcap_{i=1}^{k+1} \Hat{F_i}$ 
can be identified with $\Tilde F_1$. 
In particular, for any $J \subset \{2,\ldots,\Tilde N\}$ and 
$n \in \{1,\ldots,k+1\}$,
the face
$ \Tilde{F}_J\,\!'  \cap \big( \bigcap_{i \neq n} \Hat{F}_i \big)  \subset \Delta$
is empty exactly if 
$\Tilde{F}_J \subset \Tilde{\Delta}$ is empty,
and
the face $\Tilde{F}_J\,\!'  \cap
\big(\bigcap_{i=1}^{k+1} \Hat{F}_i \big)  \subset \Delta$
is empty exactly if   
$ \Tilde{F}_J  \cap \Tilde{F}_1 \subset \Tilde{\Delta} $ is empty.
\MS

\NI (ii)
The base-type facets $\Hat{F}_1,\ldots,\Hat{F}_{k+1}$ 
are clearly equivalent.  
By Lemma~\ref{le:equiv}, two fiber-type facets $\Tilde{F}_i\,\!'$ and $\Tilde{F}_j\,\!'$ 
of $\Delta$ are equivalent
exactly if  the corresponding facets $\Tilde{F}_i$ and $\Tilde{F}_j$ of $\Tilde \Delta$ are
equivalent.  
Similarly, a fiber-type facet $\Tilde{F}_i\,\!'$ is equivalent to  the base-type
facets exactly if $\Tilde{F}_i \sim \Tilde{F}_1$.
\end{rmk}

Conversely, if a polytope has equivalent facets, it  is either
an expansion or a bundle over a simplex.

\begin{prop}[I, Proposition 3.17]\labell{prop:Iexpan} 
Let $\Delta \subset \ft^*$ be a smooth
 polytope.
Let $I \in \Ii$ be an equivalence class of facets and 
define $I': = I \ssminus \{n\}$ for some $n \in I$.

\begin{itemize}\item[(i)] If $F_I  = \emptyset$, then $\Delta$ is a $F_{I'}$ bundle
over $\Delta_{|I|-1}$ 
with base facets  $\{F_i\}_{i \in I}$.

\item[(ii)] If $F_I \neq \emptyset$, then $\Delta$ is the
$(|I|-1)$-fold expansion of $F_{I'}$ along $F_I = F_n \cap F_{I'}$   
with base-type facets $\{F_i\}_{i \in I}$.
\end{itemize}
\end{prop}

\begin{rmk}\rm \labell{rmk:fibexp}
In most ways, mass linear functions on $k$-fold expansions behave like
mass linear functions on bundles over  the simplex $\Delta_k$.
For example, since the base-type facets are equivalent there is a one-to-one correspondence between  mass linear functions
on $\De$ with symmetric fiber facets and mass linear 
functions
on $\Delta_k$,
and all such functions are inessential; cf.\ Proposition~\ref{prop:lift}. 
Similarly,  
as in Lemma~\ref{le:fibbun},
if $H \in \ft$ is a  mass linear function on $\Delta$ with symmetric base-type facets,
then there exists a mass linear $\Tilde H \in \Tilde \ft$ such
that $\iota(\Tilde H) = H$, where $\iota \colon \Tilde \ft \to \ft$ is the natural inclusion.
In contrast, just as for bundles, even 
if $\Tilde H$ is mass linear
on $\Tilde \De$,  $\iota( \Tilde H)$ may not be
mass linear on $\De$;
see Example~\ref{ex:expanii}.

However, there are some significant differences between these two cases.
Most notably,
 Remark~\ref{rmk:expint} (ii) implies that
  $\Tilde{H} \in \Tilde{\ft}$ is inessential
on $\Tilde{\Delta}$ exactly if $H = \iota(\Tilde{H})$ is
inessential on $\Delta$ and the base-type facets  are symmetric. 
By Lemma \ref{le:fibbun} (ii) the corresponding statement is not true for bundles.
(Contrast Remark~\ref{rmk:expint} (ii) with  Lemma~\ref{le:bund} (ii).)
These differences arise because expansions correspond to very special bundles.  
In fact, Example \ref{ex:expanbun} shows how to convert a $k$-fold expansion into  a bundle over $\De_k$ by 
blowing 
up; but the converse operation is not usually possible.
\end{rmk}

Let $\Tilde{\Delta} \subset \Tilde\ft$ be a
smooth  polytope.
If we first take the $1$-fold expansion of $\Tilde\Delta$
along a facet $\Tilde{F}_1$, and then take the $1$-fold
expansion of the resulting polytope along one of the 
base-type facets, we simply obtain the $2$-fold
expansion of $\Tilde{\Delta}$ along $\Tilde{F}_1$.
(By repeating this process, we can obtain the $k$-fold expansion.)
However, if instead we 
expand  the $1$-fold
expansion of $\Tilde\Delta$ along the fiber-type facet
associated to $\Tilde{F}_2$,  we get something new: a double expansion.

\begin{defn}\labell{def:dexpand}
Let $\Tilde{\Delta} =
\bigcap_{j = 1}^{\Tilde{N}} \{ x \in \Tilde{\ft}^* \mid \langle  
\Tilde{\eta}_j, x
\rangle \leq \Tilde{\kappa}_j \}$ be a 
smooth polytope. 
A polytope $\Delta \subset \ft^*$ is the {\bf double expansion}
of $\Tilde{\Delta}$ along the facets $\Tilde{F}_1$ and $\Tilde{F}_2$ if there
is an identification
$\ft = \Tilde{\ft} \oplus \R^2$ so that
\begin{gather*}
\Delta =
\bigcap_{j = 3}^{\Tilde{N}} \{ x \in {\ft}^* \mid 
\langle (\Tilde{\eta}_j,0),x \rangle \leq \Tilde{\kappa}_j\} 
\cap\bigcap_{i = 1}^{4} \{ x \in {\ft}^* \mid \langle  \Hat\eta_i,x
\rangle \leq \Hat\kappa_i \},
\quad \mbox{where} \\
\Hat\eta_1 = (0,-e_1), \ \ 
\Hat\eta_2 = (\Tilde{\eta}_1,  e_1),  \ \ 
\Hat\eta_3 = (0,-e_2), \ \ 
\Hat\eta_4 = (\Tilde{\eta}_2,  e_2),  \\
\Hat\kappa_{1} = \Hat\kappa_3 = 0, \ \ 
\Hat\kappa_2 = \Tilde{\kappa}_1, \quad \mbox{and} \quad
\Hat\kappa_4 = \Tilde{\kappa}_2.
\end{gather*}
We shall call the facet $\Tilde{F}_j'$ of $\Delta$ with  outward conormal 
$(\Tilde{\eta}_j,0)$ the {\bf fiber-type facet (associated  to 
$\mathbf{ \Tilde{F}_j} $) }
for all $j > 2$,
the facets $\Hat{F}_1$ and $\Hat{F}_2$ the {\bf base-type facets
(associated to $\mathbf {\Tilde{F}_1}$)}, and
the facets $\Hat{F}_3$ and $\Hat{F}_4$ the {\bf base-type facets
(associated to $\mathbf {\Tilde{F}_2} $)}.
\end{defn}

Note  that the order of the expansions does not matter; if we expand
first along $\Tilde{F}_2$ and then along the fiber-type facet  
associated to $\Tilde{F}_1$, the
resulting polytope is isomorphic to $\De$ under 
the transformation that 
interchanges the last two coordinates of $\Tilde\ft \oplus \R^2$. 
Here are a few properties which will be useful later.

\begin{rmk}  \labell{rmk:dint} \rm (i)
Fix $k \subset \{3,\dots,\Tilde{N}\}$.
If the 
facet
 $\Tilde F_k$ of $\Tilde \De$ 
intersects both $\Tilde F_1$ and $\Tilde F_2$,
then -- applying  Remark~\ref{rmk:expint} (i) twice  --
the face 
$\Hat{F}_{ij} \cap \Tilde F_k\,\!': = \Hat{F}_i \cap \Hat{F}_j \cap \Tilde F_k\,\!'$ 
 of $\Delta$
intersects all the  base-type facets
for  any $i \in \{1,2\}$ and $j \in \{3,4\}$.
Conversely, if 
$\Tilde F_k$
does not intersect  both $\Tilde F_1$ and $\Tilde F_2$,
then $\Tilde F_k\,\!'$ will not intersect both $\Hat{F}_{12}$ and $\Hat{F}_{34}$.
A fortiori, 
the face 
$\Hat{F}_{ij} \cap \Tilde F_k\,\!'$ 
will not intersect all the base-type facets for any 
$i \in \{1,2\}$ and $j \in \{3,4\}$.
\MS

\NI (ii)   
Similarly, applying Remark~\ref{rmk:expint} (ii) twice,  the
base-type facets $\Hat{F}_1$ and $\Hat{F}_2$ are equivalent,  
as are  the base-type facets $\Hat{F}_3$ and $\Hat{F}_4$.
Moreover, these facets are all equivalent to each other
exactly if the facets $\Tilde{F}_1$ and $\Tilde{F}_2$ are equivalent. 
\end{rmk}

Finally we show how to recognize double expansions.

\begin{lemma}\labell{lemma:dexpan}
Let $F_1,\dots,F_4$ be facets of a smooth polytope $\Delta \subset \ft^*$.
If $F_1 \sim F_2$, $F_3 \sim F_4$, 
$F_{12} \neq \emptyset$, and  $F_{34} \neq \emptyset$, then 
$\Delta$ is the double expansion of $F_{13}$ along $F_2 \cap F_{13}$ 
and $F_4 \cap F_{13}$ with base-type facets 
$F_1,\dots,F_4$.
\end{lemma}

\begin{proof}
By Proposition~\ref{prop:Iexpan}, $\Delta$ is the $1$-fold expansion
of $F_1$ along 
$F_2 \cap F_1$.
Clearly, the fact that $F_3$ is equivalent to $F_4$ implies
that $F_{13}$ is equivalent to $F_{14}$.
By Remark~\ref{rmk:expint} (i),
the fact that  $F_{34} \neq \emptyset$ implies
that 
$F_{13}\cap F_{14} = F_{34} \cap F_1 \neq \emptyset$.
Therefore, Proposition~\ref{prop:Iexpan} also implies
that $F_1$ is the $1$-fold expansion of $F_{13}$ along $F_4 \cap F_{13}$.
The claim follows immediately.
\end{proof}

\subsection{Blowing up}\labell{ss:blow}

In this section, we show how to construct new polytopes
by blowing up faces 
of  polytopes.
We also consider how this operation affects  
mass linear functions.
We begin with the definition of  
blowup.

\begin{defn}\labell{def:blowup}
Let $\Delta =
\bigcap_{i = 1}^N \{ x \in \ft^* \mid \langle \eta_i, x\rangle \leq \kappa_i  \}$
be a 
smooth polytope.
Given a face $f  = F_I$
of codimension at least $2$ and $\eps > 0$,
let $\eta_0\,\!': = \sum_{i \in I} \eta_i$ and 
$\kappa_0\,\!': = \sum_{i \in I} \kappa_i - \epsilon$.
The polytope
$$
\Delta' =  \Delta \cap  
\{  x \in \ft^* \mid
\langle  \eta_0\,\!', x\rangle \leq  \kappa_0\,\!'  \}
$$
is the {\bf blowup of $\mathbf \Delta$ along $\mathbf f$} 
provided that
$\eps$ is sufficiently small that
$ \langle  \eta_0\,\!', v \rangle < \kappa_0\,\!'$ for all vertices $v \in \Delta$ which do not lie on $f$.
\end{defn}

It is easy to check that $\De'$ is  smooth.
The facet $F_0\,\!'$ corresponding to $\eta_0\,\!'$ is called
the {\bf exceptional divisor};
there is a natural one-to-one correspondence
between the facets $F_j$ of $\Delta$
and the remaining facets $F_j\,\!'= F_j \cap \Delta'$ of $\Delta'$.

\begin{rmk}\labell{rmk:blowint} \rm
The exceptional divisor is a 
$\Delta_{|I|-1}$ bundle over 
$f = F_I$ 
with fiber facets $F_i\,\!' \cap F_0\,\!'$
for $i \in I$;
the base facets are the facets of $f$.
Moreover,  
$f$
is the only face of $\Delta$ 
which is ``lost''. Hence 
if $K \subset \{1,\dots,N\}$ 
does not contain $I$,  
then  
$\bigcap_{k \in K \cup \{0\} } F_k\,\!'  \neq \emptyset$
exactly if 
$\bigcap_{k \in K \cup I} F_k \neq \emptyset$; similarly,
$\bigcap_{k \in K} F_k\,\!' 
\neq \emptyset$ 
exactly if 
$\bigcap_{k \in K} F_k 
\neq 
\emptyset.$  
\end{rmk}

The following example demonstrates the very close connection 
between bundles and expansions.

\begin{example}\labell{ex:expanbun}\rm
Suppose that $\De$ is the $k$-fold expansion 
of $\Tilde \De$.
Let $\Delta'$ be the blowup of $\Delta$ along the face
$f = \bigcap_{i=1}^{k+1} \Hat{F}_i$. 
It is straightforward to check directly that $\Delta'$ is a
$\Tilde \Delta$ bundle over $\Delta_k$ and that the base facets
are $\Hat{F}_1 \cap \Delta',\ldots,\Hat{F}_{k+1} \cap \Delta'$; this justifies
our terminology.  
\end{example}

In the next remark we show that 
the blowup of a 
polytope $\Delta$ along a face $f$
corresponds to the usual geometric blowup of
the toric manifold $M_\De$ along 
a submanifold $M_f$, 
and give a geometric interpretation of the preceding example.

\begin{rmk}\labell{rmk:geoblow}\rm (i)
Let  $\Delta \subset \ft^*$ be a  smooth polytope,
and  let $\Delta'$ be the blowup of $\Delta$ along
a face $f = \bigcap_{i \in I} F_i$; assume that
$I = \{1,\ldots,k\}$;
we will use the notation of the definition above.
Construct  the associated toric manifolds
$M_\Delta = \Uu/K_\C$  and 
$M_{\Delta'} = \Uu'/K_\C\,\!'$  
as in 
[I, Remark 5.1], 
where $\Uu \subset \C^N$ and   $\Uu' \subset \C^{N+1}$;
identify $\C^{N+1}$ with $\C \times \C^N$.
Since $\eta_0\,\!' = \sum_{i \in I} \eta_i$,
$\lambda = e_0 - \sum_{i \in I} e_i$ lies in $\fk$;
define $\Lambda \colon S^1 \to K$ by $\Lambda(\exp(t)) = \exp(t \lambda)$.
Moreover, the intersection of $K'$ with the inclusion
$(S^1)^N \to (S^1)^{N+1}$ is $K$; 
hence we can write $K' = K \times \Lambda$.
It is easy to check that the map
$f \colon \Uu' \to  \Uu$ defined by  $f(z_0,\ldots,z_N) = 
(z_0^{-1} z_1,\ldots,z_0^{-1} z_k, z_{k+1}, \ldots,z_N)$
is surjective and 
induces a well defined map of toric manifolds.
If $z_i = 0$ for all $i \in I$,
the preimage $f^{-1}(z)$ is isomorphic to $\C^k \ssminus \{0\}$;
otherwise, the preimage is a single $\Lambda$ orbit.
Therefore, $f$ induces a surjective holomorphic map
$\overline{f} \colon M_{\Delta'} \to M_\Delta$ which collapses
$\Phi_{\Delta'}^{-1}(F_0\,\!')$ to $\Phi_\Delta^{-1}(f)$, but
is otherwise a homeomorphism. 
\MS

\NI (ii)  By [I, Remark 5.4] the toric manifold $M_{\De}$ corresponding to the $1$-fold expansion  $\De$ of $\Tilde\De$ along $\Tilde F_1$ can be thought of as a  nonsingular symplectic pencil with fibers 
$M_{\Tilde\De}$ and axis $M_{\Tilde F_1}$.  Thus Example \ref{ex:expanbun} shows
that when we blow up this axis we get a toric bundle.
\end{rmk}

The next lemma explains how blowing up affects
the
facet equivalence relation.

\begin{lemma} \labell{le:blowequiv}
Let $\Delta'$ be the blowup of  
a  polytope
 $\Delta$
along a face $F_I$.
\begin{itemize}\item[(i)]
Given facets $F_i$ and $F_j$ of $\Delta$, 
the corresponding facets 
$F_i\,\!'$ and $F_j\,\!'$  
of $\Delta'$ are equivalent exactly if
$F_i$ is equivalent to $F_j$ and the pair $\{i,j\}$ is
either contained in  $I$ or disjoint from $I$.
\item[(ii)]
The exceptional divisor  $F_0\,\!'$ is not equivalent to any other facet.
\end{itemize}
\end{lemma}
\begin{proof}
If $F_i$ and $F_j$ are not equivalent, claim (i)  
is clear.
So assume that $F_i$ and $F_j$ are equivalent.
By Lemma~\ref{le:equiv},
the subspace $V \subset \ft$ spanned 
by the conormals $\eta_k$ for $k \not\in \{i,j\}$
has codimension $1$; moreover, the sum $\eta_i + \eta_j$ lies in $V$.
Hence, if both $i$ and $j$ are in $I$,
then $\eta_0 = \eta_i + \eta_j \in V$,
and if neither are in $I$ then again $\eta_0\in V$.
Hence in these cases
 $F_i\,\!'$ and $F_j\,\!'$ are equivalent. 
In contrast, if  only one is in $I$ then
$\eta_0 \not\in V$,
and so  
$F_i\,\!'$ and $F_j\,\!'$ are not equivalent.

Now consider (ii). If $F_0\,\!'$  were equivalent to $F_k\,\!'$ then by Lemma~\ref{le:equiv} the subspace of 
$\ft$ spanned
by the outward conormals to 
all
the facets of $\Delta$ 
 except $F_k$ would have codimension
$1$.  But  this is impossible when 
$\De$ is compact.
\end{proof}

We are now ready to analyze the behavior of mass linear functions
under blowups. 
Our arguments 
use the elementary fact that the
volume and
$H$-moment  
\begin{equation}\labell{eq:Hmom}
V(\De):= \int_{\De} 1 \quad \mbox{and} \quad
\mu_H(\De): = \int_{\De} H(x) 
\end{equation}
of the polytope $\De$ with respect to the affine volume form are    
additive when $\De$ is decomposed as a sum $\De'\cup W$.  In other words
$V(\De) = V(\De') + V(W)$ and
$\mu_H(\De) = \mu_H(\De') + \mu_H(W)$.
Note also that 
$\mu_H(\De) =
\langle H, c_\Delta \rangle \,\,V(\De)$.

Since the facets of $\Delta$ are a subset of the facets of $\Delta'$,
we may think of $\langle H, c_\Delta \rangle$ as a function on 
an open subset $U$ of $\R^{N+1}$ ---
a function which does
not depend on the support number $\kappa_0'$ of the exceptional divisor.
We will say that $\langle H, c_\Delta \rangle$ and
$\langle H, c_{\Delta'} \rangle$ are {\bf equal} if they agree
on 
a nonempty open subset of the form
$U \cap \Cc_{\Delta'}$. In this case, the exceptional divisor is
symmetric and $H$ is mass linear on $\Delta$ exactly if it is mass linear
on $\Delta'$; moreover, if they are mass linear   
the coefficient of the support number of a facet $F_i$ in $\langle H,
c_\Delta \rangle$ is the coefficient of the support number of 
$F_i \cap \Delta'$ in $\langle H,c_{\Delta'} \rangle$.
Similarly, we may think of $\langle H, c_W \rangle$ as a function on 
an open subset of $\R^{N+1}$
which does not
depend on the support numbers of any of the facets
of $\De$
which 
do not intersect $f$.

\begin{lemma}\labell{le:newblow} Fix $H \in \ft$.  
Let $\Delta'$ be the blowup of 
a  polytope 
$\Delta \subset \ft^*$
along a face $f$ and write $\De= \De'\cup W$.   
Assume that
two of the three functions
$\langle H, c_{\Delta}\rangle, \langle H, c_{\Delta'}\rangle$, and 
$ \langle H, c_{W} \rangle$   
are equal.
Then
all three functions 
are equal;
in 
particular,
$H$ is mass linear on $\De$ exactly if 
$H$ is mass linear on $\De'$.
\end{lemma}
\begin{proof} 
Since the $H$-moment is additive,
$$
\langle H, c_\Delta \rangle \,\,V(\De) = 
\langle H, c_{\Delta'} \rangle\,\, V(\De') +
\langle H, c_W \rangle \,\,V(W).
$$
Since $V(\De) = V(\De') + V(W)$ the three 
functions 
$\langle H, c_{\Delta}\rangle, \langle H, c_{\Delta'}\rangle$, and 
$ \langle H, c_{W} \rangle$ 
must agree on some 
nonempty open
set, and hence, as explained at the beginning of \S\ref{ss:review}, on $\Cc_{\De'}$.
\end{proof}

We first 
describe 
what happens when $H$ is  mass linear on a polytope $\De'$ that is a blowup.

\begin{lemma}\labell{le:blowcent}
Let $H \in \ft$ be a mass linear function on a polytope
$\Delta'$ that is  the blowup of 
a polytope
 $\Delta$
along a face $f$.
The 
following hold.

\begin{itemize} 
\item[(i)]   The exceptional divisor $F_0\,\!'$ is symmetric.

\item[(ii)]  $H$ is mass linear on $\Delta$ and 
$\langle H, c_{\De} \rangle=\langle H, c_{\Delta'} \rangle$.

\item [(iii)]  The face $f$ meets every asymmetric facet.

\item[(iv)]
If $H$ is inessential on $\Delta'$, then it is inessential on $\Delta$.
\end{itemize}
\end{lemma}
 \begin{proof}
Decompose $\De$ as $\De'\cup W$ where $W$ is a $\De_{|I|}$-bundle over $f = F_I$ with a fiber of size 
$\eps=\sum_{i\in I}\ka_{i} -\ka_0'>0$.
Let $V_W$ and $V'$ denote the volume of $W$ and $\Delta'$,  respectively;
similarly, let 
$\mu_W $ and $\mu'$ 
denote the $H$-moment 
of $W$ and $\Delta'$ as in  Equation \eqref{eq:Hmom}.
At $\epsilon = 0$, the partial derivatives $\frac{\p V_W}{\p \ka_0'}$ and $\frac{ \p \mu_W}{\p \ka_0'}$
both vanish since $V_W$ and $\mu_W$ are polynomial functions with
a factor 
$\eps^k$, where $k > 1$.
By the additivity of the volume and moment,
 this implies that 
$\frac{\p V'}{\p \ka_0'}$ and 
$\frac{ \p \mu'}{\p \ka_0'}$ also both vanish at $\epsilon = 0$.
Finally, since $\mu' = \langle H, c_{\De'} \rangle\,\, V'$,
 this implies that 
$\frac{\p \langle H, c_{\De'} \rangle}{\p \ka_0'}$
vanishes
at $\epsilon = 0$.
Therefore, since  $H$ is mass linear on $\Delta'$, 
$F_0'$ is symmetric\footnote{
Here we use the fact that locally mass linear 
functions are globally mass linear: cf.\ [I, Lemma 2.3].
Thus 
$\langle H, c_{\De'} \rangle$ is a linear function of $(\ka_0',\ka)$ throughout the chamber $\Cc_{\De'}$.
}.
This proves (i).

Now fix 
$\kappa \in \Cc_\Delta$. 
Since $F_0'$ is symmetric,
$\langle H, c_{\Delta'}(\kappa'_0 , \kappa) \rangle$
does not depend on the
support number $\kappa'_0$  
as long as $(\kappa'_0,\kappa)$  lies  in 
$\Cc_{\Delta'}$.
In fact, since  the center of mass is a continuous function
of the support numbers, the same 
claim 
holds as long as $(\kappa'_0,\kappa)$ lies 
in the closure of $\Cc_{\Delta'}$.
Moreover,  if $\kappa'_0 = \sum_{i \in I} \kappa_i$
then $P(F_0')$ intersects $\Delta$ at exactly $f$,
and the polytopes $\Delta'(\kappa_0',\kappa) $ and $\Delta(\kappa)$ coincide. Therefore
$\langle H, c_{\Delta}(\kappa) \rangle =
 \langle H, c_{\Delta'}(\kappa'_0 , \kappa) \rangle$ is a linear function of $\ka$.
 The claims in (ii) follow immediately.

Since
the symmetric facet $F_0'$ meets all asymmetric facets of $\De'$
by Proposition \ref{prop:symcent}, the face $f$ does as well by 
Remark \ref{rmk:blowint}. This proves 
(iii).

Claim (iv)
follows immediately from 
Lemma \ref{le:blowequiv}. 
\end{proof}

We are now ready to consider the question of which 
blowups preserve 
mass linearity.
The simplest case is {\bf symmetric blowup}, that is,
blowing up along a symmetric face.

\begin{lemma}\labell{le:symblow}   
Let $H \in \ft$ be a mass linear function on a polytope $\Delta \subset \ft^*$.
Let $\Delta'$ be the blowup of 
$\Delta$
along a symmetric face $f$.
Then the following hold.

\begin{itemize}
\item[(i)]   $H$ is mass linear on $\De'$ 
and 
$\langle H, c_\De \rangle = \langle H, c_{\De'} \rangle$.
\item[(ii)] 
$H$ is essential on $\De$ exactly if it is essential on $\De'$.
\end{itemize}
\end{lemma}

\begin{proof}{}
Let
$f=F_I$,
and let $\eta_i$ be the outward conormal to $F_i$ for all $i$.
Then $\Delta = \Delta' \cup W$, where $W$ is a $\Delta_{|I|}$ bundle
over $f$. The outward conormals to the fiber facets of $W$ are $\{\eta_i\}_{i \in I}$ and $- \sum_{i \in I} \eta_i$.
The outward conormals to 
its  base facets are the outward conormals to the
facets 
of $\De$
that 
restrict to facets of $f$.
Since $f$ is symmetric, the restriction of $H$ to $f$ is mass linear with the
same coefficients.  
Hence, by Lemma~\ref{le:Hsum}, the restriction of $H$ to $W$ is the lift  
of the restriction of $H$ to $f$.
Therefore, 
Proposition~\ref{prop:lift}  
implies that 
$H|_W$ 
is mass linear with the same coefficients
on $W$,
that is,
$\langle H, c_{\Delta}\rangle = \langle H, c_{W} \rangle$.  
Thus (i) follows from Lemma \ref{le:newblow}.

Since $F_i$ is symmetric for all $i \in I$, 
Proposition~\ref{prop:inessential}
implies that 
every inessential $H$  has the form
$H = \sum_{j \not\in I} \be_j \eta_j$.
Therefore to prove (ii)  
it suffices 
to recall that, by Lemma \ref{le:blowequiv},   
$F_k$ and $F_\ell$ 
are equivalent facets of $\De$ for 
some $k$ and $\ell$ 
not in $I$ exactly if
$F_k\,\!'$ and $F_\ell\,\!'$
are equivalent facets of  $\De'$; moreover, the exceptional divisor $F_0'$ is not equivalent to any other facet.
\end{proof}

We can also blow up faces which are  not symmetric,
but then the situation is more complicated.
We first 
describe the scenario 
that is most relevant to 
the $4$-dimensional classification.  It
 turns out 
to be an important mechanism for creating 
new essential
mass linear functions, since 
$H$  may 
be essential
on $\Delta'$
even if 
it
is inessential on $\Delta$.  See Example \ref{ex:7} for an easy instance of this
process,
 and Propositions 
\ref{prop:essblow2} and \ref{prop:essblow} 
for a more extended discussion.

\begin{defn}\labell{def:edge}  Let $H$ be a mass linear function on a polytope $ \De$ with asymmetric facets 
$F_1,\dots,F_k$.
We say that a blowup of $\De$ is  
 {\bf of type} 
 $\mathbf{ ({F}_{ij},  g)}$ 
if it is the blowup of $ \De$ along the edge $ 
 {F}_{ij}\cap  g$, where
 $g$ is a  symmetric $3$-face, 
${F}_{ij} \cap {g}$ intersects every asymmetric facet,
and
$\gamma_i + \gamma_j = 0$.
Here 
$\gamma_k$ is the coefficient of the support number of ${F}_k$
in the linear function $\langle H, c_{\Delta} \rangle.$
\end{defn}

\begin{prop}\labell{prop:blow4}
Let $H \in \ft$ be a mass linear function on a smooth polytope
$\Delta \subset \ft^*$.
Let $\Delta'$ be 
 a blowup of $\De$ of type $(F_{ij},g)$. The following hold.
  \begin{itemize}\item [(i)]
 $\De$ has 
zero, two, 
or four asymmetric facets.
 \item [(ii)]
 $H$ is mass linear on $\Delta'$
and $\langle H, c_\Delta \rangle = \langle H, c_{\Delta'} \rangle$.
\item [(iii)] 
If $H$ is essential on $\De$, then it is essential on $\Delta'$;
otherwise, $H$ is essential on $\Delta'$ exactly if $F_i \not\sim F_j$
and 
there are four asymmetric facets.
\end{itemize}
\end{prop}

\begin{proof}  Label the facets of $\De$ so that 
the two facets that intersect the edge $F_{12}\cap g$ are $F_3$ and $F_4$,
and 
so that
$g=\cap_{j=1}^{n-3} G_j$.
Let $\eta_i$ denote the outward conormal to $F_i$ 
and let $\alpha_j$
denote the outward conormal to $G_j$. 
Since the edge $F_{12}\cap g$ intersects every asymmetric facet, 
each 
facet except possibly $F_1,\dots,F_4$ is symmetric.

By Proposition~\ref{prop:symcent}, $\langle H, c_g(\kappa) \rangle
= \langle H, c_\Delta(\kappa) \rangle$ for all $\kappa \in \Cc_\Delta$.
In particular, the restriction of $H$ to $g$ is mass linear.
Thus,  Proposition~\ref{prop:3d}  
implies that 
$\sum_{i=1}^4  \gamma_i = 0$.
Since $\gamma_1 + \gamma_2 = 0$,
this implies that $\gamma_3 + \gamma_4 = 0$.  
Therefore, $\ga_1$ and $\ga_2$ 
(respectively $\ga_3$ and $\ga_4$)
are either both zero or both nonzero.  
This proves (i).

If $F_1$ and $F_2$ are symmetric,  
claim (ii) follows
from Lemma~\ref{le:symblow}.
Hence, we may assume that $F_1$ and $F_2$ are asymmetric facets.
By Proposition~\ref{prop:symcent},
intersection induces a one-to-one correspondence
between the asymmetric facets of $\Delta$ 
and the asymmetric facets of $g$.
Therefore, $F_1 \cap g$ and $F_2 \cap g$ are asymmetric facets of $g$
and   
$F_{12}\cap g$ intersects every asymmetric facet of $g$. 
Hence, Lemma~\ref{le:nonsb} below implies that
$F_1\cap g$ and $F_2\cap g$ 
are equivalent facets of $g$.  

We claim that 
 $\eta_3$, $\eta_4$, $\alpha_1,\ldots,\alpha_{n-3}$, and
$\eta_1 + \eta_2$ all lie in a hyperplane of $\ft$.
 To see this, observe that
 the smallest affine plane $P(g)\subset \ft^*$ containing the face $g$ is 
  $$
 P(g) = 
\bigcap_{j = 1}^{n-3}
\{x\in \ft^*\,|\, \langle\al_j,x\rangle = \ka_j\},
 $$
 and hence may be identified with 
the dual to the quotient 
of $\ft$ by the span $V_\al$ of the $\al_j$.  
(This  is explained in more detail at the beginning of [I, \S 2].)   
 Let $\pi: \ft\to \ft/V_\al$ denote the projection. 
  Then the claim will follow if we can show that the vectors $\pi(\eta_1) + \pi(\eta_2), 
 \pi(\eta_3),$ 
and $ \pi(\eta_4)$ span a hyperplane in 
 $\ft/V_\al$. But by Lemma~\ref{le:equiv}, this 
follows from the fact that
 $F_1\cap g$ and $F_2\cap g$ are equivalent facets of $g$.

Now note that $\Delta = \Delta' \cup W$, where $W$ is a $\Delta_{n-1}$ bundle over
$\Delta_1$. 
The outward conormals to the fiber facets of $W$
are  $\eta_1, \eta_2, \alpha_1,\ldots,\alpha_{n-3},$ and 
$-\eta_1 - \eta_2 -\sum_i \alpha_i$;
the outward conormals to the base facets are $\eta_3$ and $\eta_4$.
Therefore, the facets of $W$ with conormals $\eta_3$ and $\eta_4$ are equivalent.
Moreover, since $\eta_3$, $\eta_4$, $\alpha_1,\ldots,\alpha_{n-3}$, and
$\eta_1 + \eta_2$ lie in a hyperplane,
the facets of $W$ with conormals $\eta_1$ and $\eta_2$ are equivalent.
By Lemma~\ref{le:Hsum}, 
$H = \sum_{i=1}^4 \gamma_i \eta_i$.
Since $\gamma_1 + \gamma_2 = \gamma_3 + \gamma_4 = 0$,
$H$ is inessential on $W$; hence by Proposition~\ref{prop:inessential}
$$\langle H, c_W \rangle = \sum \gamma_i \kappa_i = \langle H, c_\De \rangle.
$$
Claim (ii) now follows from Lemma~\ref{le:newblow}.

Since the first part of claim (iii) is a special case of
Lemma~\ref{le:blowcent} (iv),  we may assume
that $H$ is inessential on $\De$. 
By Proposition~\ref{prop:inessential}, this implies that every asymmetric
facet must be equivalent to at least one other asymmetric facet.
Moreover, recall that $\gamma_1 + \gamma_2 = \gamma_3 + \gamma_4 = 0$.
Hence, if $F_1 \sim F_2$ or if any of the facets $F_1,\dots,F_4$
are symmetric, then the following statements are both true.
\begin{itemize}
\item  $F_1 \sim F_2$ or $\gamma_1 = \gamma_2 = 0$, and
\item $F_3 \sim F_4$ or $\gamma_3 = \gamma_4 = 0$.
\end{itemize}
Hence,
$H$ is inessential on  
$\De'$ 
by Lemma~\ref{le:blowequiv} (i).
In contrast, if 
$\gamma_1 \neq 0$ and
$F_1 \not\sim F_2$ the same lemma implies that $F_1$ is
not equivalent to any other asymmetric facet. Claim (iii) follows immediately.
\end{proof}

Here is the auxiliary lemma used above.

\begin{lemma}\labell{le:nonsb}
Let $H \in \ft$ be a mass linear function on a smooth $3$-dimensional
polytope $\Delta \subset \ft^*$.
If $F_1$ and $F_2$ are asymmetric facets and
the edge  $F_{12}$  
meets every asymmetric facet,
then $F_1$ and $F_2$ are equivalent facets of  $\Delta$.
\end{lemma}

\begin{proof}
By Proposition~\ref{prop:3d}
we see that there 
are only three possibilities:
\begin{itemize}\item
 $\Delta$ has exactly two asymmetric facets;
 \item
$\Delta$ is the simplex $\Delta_3$; or
\item
 $\Delta$ is a $\Delta_2$ bundle over $\Delta_1$, 
$F_1$ and $F_2$ are fiber facets, and the third fiber facet is symmetric. 
\end{itemize}

If $\Delta$ has exactly two asymmetric facets, $F_1$ and $F_2$, then  
$F_1 \sim F_2$ by Proposition~\ref{prop:2asym}.
If $\Delta$ is the simplex $\Delta_3$, 
then all the facets are equivalent. Therefore 
it remains to consider the 
third case.
By  Proposition~\ref{prop:bund} there is
an inessential function $H'$ so that 
the  $\Tilde H$-asymmetric facets are exactly $F_1$ and $F_2$, where $\Tilde H := H-H'$.
By Proposition~\ref{prop:2asym}, this implies that $F_1$
and $F_2$ are equivalent.
\end{proof}

As we mentioned 
above, blowups of the form considered in  Proposition \ref{prop:blow4} may convert an inessential function on $\De$ to an essential function on the blowup $\De'$.
The next result shows that this is not possible if we blow up along a 
face that is contained in every asymmetric facet.

\begin{lemma}\labell{le:inessblow}
Fix $H\in \ft^*$ and let
$\Delta'$ be the blowup of 
a polytope
$\Delta$ along a face $F_I$ that is contained in every $H$-asymmetric facet. Then $H$ is inessential on $\De$ exactly if it is inessential on $\De'$.
\end{lemma}

\begin{proof} 
If $H$ is inessential on $\De$, this follows easily from
Proposition~\ref{prop:inessential} and
Lemma \ref{le:blowequiv} (i).
The converse is a special case of 
Lemma \ref{le:blowcent} (iv).  
\end{proof}

Further examples of blowups that preserve mass linearity are given
in \S\ref{ss:blml}.
For example, we show that blowing up a polytope 
at
a vertex that meets every asymmetric facet preserves mass linearity;
see Corollary~\ref{cor:blowv}.

\subsection{Blowing down}\labell{ss:blowd}

Although one can always blow up a smooth polytope along a facet of codimension at least $2$ to obtain a new smooth polytope,
it is not so easy to decide when this process can be reversed.
This subsection explores 
general
conditions under which this is possible.
Explicit
$4$-dimensional examples may be found in \S\ref{ss:min}.

We say that a smooth polytope $\De$ can be {\bf blown down} along
a facet $F_0$
if $\Delta$ is the blowup of a smooth polytope $\ov\Delta$ along some face 
$\ov f$, and $F_0$ is the exceptional divisor.
In this case,
the polytope $\ov \De$ is obtained 
from $\De$ by moving the hyperplane $P(F_0)$ outwards 
(i.e. increasing its support number $\ka_0$) until it 
no longer intersects the intersection of the remaining half spaces. 
The facet $F_0$ must be a bundle whose fiber is a simplex $\De_k$. 
As  $\ka_0$ increases, 
the sizes and relative positions of the fiber and the base facets of $F_0$  changes.
If the outward conormal to $F_0$  is a positive multiple of
the sum of the fiber facets, the size of the fiber facet will
decrease as we move $P(F_0)$ outwards.
The transition from $\De$ to $\ov \De$ 
is  a blowdown if $\ov \De$ is smooth and if during this movement of 
$P(F_0)$ there is precisely one value of $\ka_0$ for 
which $P(F_0)$ intersects 
a vertex of $\De$. 
What is crucial is that the size of the fiber
shrinks to zero before any new intersections of the base facets
of $F_0$
are created.

It is easy to check that an edge of a smooth $2$-dimensional polygon can
be blown down exactly if the outward conormal to that edge
is the sum of the outward conormals to the 
two
adjacent edges.
In higher dimensions, the situation is somewhat  more complicated.

\begin{prop}\labell{prop:blowdown} 
Let $\Delta =
\bigcap_{i = 0}^N \{ x \in \ft^* \mid \langle \eta_i, x\rangle \leq \kappa_i  \}$
be a smooth polytope; denote the facets by $F_0,\ldots,F_N$.
Fix $I \subset \{1,\ldots,N\}$.
Then the polytope $\ov\Delta = 
\bigcap_{i = 1}^N \{x \in \ft^* \mid \langle \eta_i, x\rangle \leq \kappa_i \}$
is 
smooth and $\Delta$ is the blowup of
$\ov \Delta$ along $\ov{F}_I := \bigcap_{i \in I} P(F_i) \cap \ov\Delta$ with 
exceptional divisor $F_0$ exactly if
\begin{enumerate}
\item [(i)] The facet $F_0$ is a $\Delta_{|I|-1}$ bundle  
with fiber facets $\{F_i \cap F_0\}_{i \in I}$
and base facets $\{F_j \cap F_0\}_{j \in J}$  for
some $J  \subset \{1,\ldots,N\}.$
\item [(ii)] $ \eta_0 = \sum_{i \in I} \eta_i .$
\item [(iii)]
Given $K \subset J$, if  $F_K = \emptyset$
then $\ov{F}_K  := \bigcap_{k \in K} P(F_k) \cap \ov\Delta =   \emptyset$. 
\end{enumerate}
\end{prop}

\begin{proof}
We have already seen that if $\ov\Delta$ is smooth and 
 $\Delta$ is
the blowup of $\ov\Delta$ along $\ov{F}_I$
with exceptional divisor $F_0$,
then (i),  (ii) and (iii) hold;
see Remark~\ref{rmk:blowint}.

To prove the converse, first  note that
since $\Delta$ is compact, the positive span of the $\eta_i$
is all of $\ft$.   By assumption (ii), this  implies that 
the positive span of the $\eta_i$ for $i \geq 1$ is also all of $\ft$, 
and so $\ov\Delta$ is compact.

Next, consider a ``new'' vertex $\ov{v}$ of $\ov \Delta$, that is,
a vertex which satisfies $\langle  \eta_0,\ov v \rangle  > \kappa_0$
and hence does  not lie in $\Delta$.
Write $\ov{v} = \ov{F}_{I'} \cap \ov{F}_{K}$, where $I' \subset I$ 
and $K \cap I = \emptyset$.
Since the facet $F_k$ is not empty
and  $\ov{F}_k = P(F_k) \cap \ov\Delta$ is connected, 
the intersection $\ov{F}_k \cap F_0 = F_k \cap F_0$
is not empty for any $k \in K$.  Hence, $K \subset J$, and
so by assumption  (iii)  the face $F_K$  is
also nonempty.  Since $\ov{F}_K$ is connected, this implies that
$\ov{F}_{K} \cap F_0 = F_K \cap F_0$ is not empty.
By assumption (i), this implies that 
$F_{I \ssminus \{i\}} \cap F_{K} \cap F_0$ 
is a vertex of $\Delta$ for all $i \in I$.
Since $\Delta$ is simple, this implies that $|I| + |K| = n$.
Since there must be at least $n$ facets through $\ov{v}$, it also
implies that $|I'| = |I|$.  Hence $I'=I$ and also 
$\ov{v}$ is a simple vertex.
Since 
$\De$ is smooth, the vectors 
$\{ \eta_j\}_{j \in I \ssminus \{i\}}$, 
 $\{\eta_k\}_{k \in K}$
and $\eta_0$ span the lattice  
$\lat$
for all $i \in I$.
By part
(ii), this means that the vectors 
$\{\eta_j\}_{j\in I} $ and  $\{\eta_k\}_{k \in K}$ also span the lattice,
that is, that $v$ is a smooth vertex.

Since  $\Delta$ is smooth,
and every ``new'' vertex is  smooth, $\ov\Delta$ is also smooth.
Finally, since $\Delta$ is compact and every new vertex lies on $F_I$,
$F_I$ is not empty.
\end{proof}

\begin{rmk}\rm 
In some cases, the  polytope $\De$ can be blown down along
the facet $F_0$ for some values of $\ka \in \Cc_{\De}$, but not for
other values $\ka' \in \Cc_{\De}$;
see Figure~\ref{fig:2}. This is because condition (iii) may depend
on $\ka$.\smallskip

 \begin{figure}[htbp] 
     \centering
 \includegraphics[width=4in]{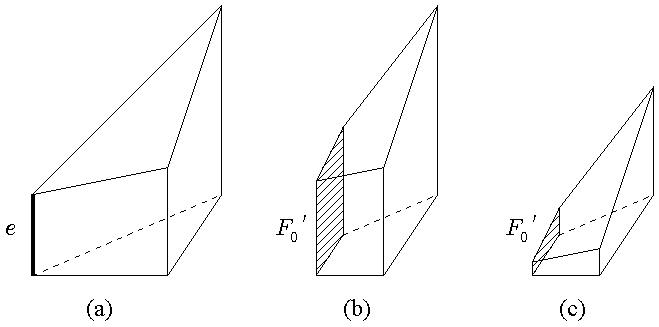} 
    \caption{(b) is the blowup of (a) along $e$.  When the top facet is
     moved down as in (c), the facet $F_0'$ no longer blows down.}
    \labell{fig:2}
 \end{figure}
 \begin{figure}[htbp] 
     \centering
 \includegraphics[width=5in]{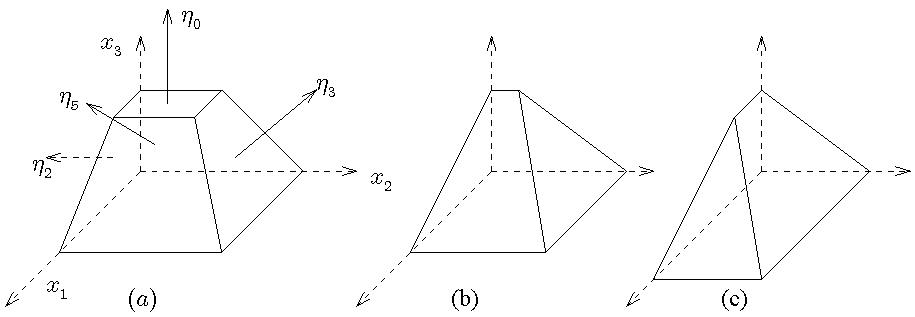} 
    \caption{(b) is the blowdown of (a) along $F_0$ with $I = \{4,5\}$;
    (c) is the blowdown with $I = \{2,3\}$.}
    \labell{fig:4}
 \end{figure}

\NI
Another possibility is that
the blowdown of $\Delta$ along $F_0$ 
depends
on the choice of $\kappa$.
For example, suppose that 
$\Delta = \bigcap_{i = 0}^5 \{x \in \R^3 
\mid \langle \eta_i, x\rangle \leq \kappa_i \}$
where
$$
\eta_0 = e_3, \ \eta_1 = -e_3,\  \eta_2 = -e_2,\  \eta_3 =  e_2 + e_3,  \ 
\eta_4 = -e_1, \ \eta_5 =  e_1+e_3,
$$
$\kappa = (1,0,0,\la,0, 2)$, and $\lambda > 1$; 
see 
Figure~\ref{fig:4}. 
In this case, $F_0$ can be viewed as a $\Delta_1$ bundle over $\Delta_1$
in two ways -- either the fiber facets are $F_{20}$ and $F_{30}$,
or the fiber facets are $F_{40}$ and $F_{50}$; so a priori
we can take $I = \{2,3\}$ or $I = \{4,5\}$. 
Either way, condition (ii) is also  satisfied.
If $\lambda > 2$, 
then condition (iii) also holds if we take $I = \{4,5\}$;
so $\Delta$ is the blowup of $\ov \Delta$ along the (non empty) face $F_{45}$.
Conversely, 
if $\lambda < 2$, 
then $\Delta$ is the blowup of $\ov \Delta$
along the (non empty) face $F_{23}$.  Finally, 
if $\lambda = 2$,
then condition (iii) is not satisfied in either case.  In fact, it is
easy to see that $\ov\Delta$ is not a simple polytope.
\end{rmk}

In practice, we will not directly prove that 
condition (ii) of Proposition~\ref{prop:blowdown}
holds;
instead, we will use the following technical lemma which
allows us to reduce to the simpler case of 
lower dimensional polytopes.

\begin{lemma}\labell{le:blowtech}
Let $\Delta =
\bigcap_{i = 0}^N \{ x \in \ft^* \mid \langle \eta_i,x\rangle \leq \kappa_i  \}$
be a smooth
polytope.
Assume that $F_0$ is a $\Delta_{|I| -1}$ bundle with
fiber facets 
$\{ F_i \cap F_0 \}_{i \in I}$.
Also assume that there exists $L \subset \{1,\ldots,N\}$
so that the face $F_L$ is  the blowup of a smooth polytope 
$\ov F_L$ along the face 
$\bigcap_{i \in I} P(F_i) \cap \ov F_L$
with exceptional divisor $F_0 \cap F_L$.
Then condition (ii) of Proposition~\ref{prop:blowdown}  is satisfied.
\end{lemma}

\begin{proof}
By 
the  definition of  a $\Delta_{|I|-1}$ bundle,  
$\sum_{i \in I} \eta_i$ 
is constant when restricted
to $P(F_0)$, that is,
$\sum_{i \in I} \eta_i = c \eta_0$ for
some real number $c$. 
Since $F_0 \cap F_L$ is not empty 
and $\De$ is simple, 
$\eta_0$ 
is nonconstant 
when restricted to $P(F_L)$. 
Since $F_L$ is the blowup of a smooth polytope $\ov F_L$ along
the face 
 $\bigcap_{i \in I} P(F_i) 
 \cap \ov F_L$ with exceptional
divisor $F_0 \cap F_L$,
$\eta_0 - \sum_{i \in I} 
\eta_i$ 
is constant on $P(F_L)$.
Therefore, $(1 - c)\eta_0$  
 is also constant on
$P(F_L)$.
Since $\eta_0$ 
is nonconstant on 
$P(F_L)$, this implies that $c=1$. 
\end{proof}

\begin{rmk}\labell{rmk:blowtech} \rm
Conversely, let $\Delta'$ be the blowup of 
a  polytope 
$\Delta$
along a face $F_I$ with exceptional divisor $F_0\,\!'$.
Let  $F_L\,\!'$ be a face of $\Delta'$ that meets $F_0\,\!'$.
If $|I \setminus (I \cap L)| \geq 2$, then $F_L\,\!' := F_L \cap \Delta'$ is the blowup of
$F_L$ along the face $F_I \cap F_L$ with exceptional divisor $F_0\,\!' \cap F_L$.
\end{rmk}

In most of the cases we consider, condition (iii)
of Proposition~\ref{prop:blowdown}
is extremely straightforward to check.
However, for the third case of Lemma~\ref{le:blowblow}, 
we will need the following lemma.

\begin{lemma} \labell{le:blowdown2} 
Let $\Delta =
\bigcap_{i = 0}^N \{ x \in \ft^* \mid \langle \eta_i, x\rangle \leq \kappa_i  \}$
be a smooth 
$4$-dimensional
 polytope.
Assume that $F_0$ is a $\Delta_1$ bundle over $\De_1\times \De_1$
with base facets $F_1 \cap F_0, \ldots, F_4 \cap F_0$, and that  
$F_{ij}: = F_i\cap F_j$ 
is not empty for any pair $1 \leq i < j \leq 4$.
Then condition (iii) of Proposition~\ref{prop:blowdown} is satisfied.
\end{lemma}

\begin{proof}
It follows immediately from the assumptions that (iii) holds for all 
$K \subset \{1,2,3,4\}$ with at most two elements.
By renumbering, we may assume that
$F_{12} \cap F_0$ and $F_{34} \cap F_0$ are empty.
Since $F_{12}$ is not empty and $\ov{F}_{12}: = P(F_{1}) \cap P(F_2) \cap \ov{\Delta}$ 
is connected, this implies that
$\ov{F}_{12}  \subset \Delta$.  Hence, condition (iii) is satisfied
for all $K \subset \{1,2,3,4\}$  which contain $\{1,2\}$.
A similar argument shows that condition (iii) is satisfied 
for all $K \subset \{1,2,3,4\}$  which contain $\{3,4\}$.
Since every subset of $\{1,2,3,4\}$ with more than $2$ elements contains
one of these pairs, 
this completes the proof.
\end{proof}

We end this section by  considering 
blowdowns of 
polygons.  This process is well understood; every 
smooth  
polygon with
more than four edges can be blown down to a trapezoid; see \cite{Fu}.
We shall need the following more precise version of this result.

\begin{lemma}\labell{le:2blowdown}
Let $\Delta$ be a smooth  convex 
$2$-dimensional polygon with
more than four edges. 

\begin{itemize}\item[(i)] 
If $e$ and $e'$ are parallel edges, then there exists 
an edge which is not equal to $e$ or $e'$,
which is not  adjacent to $e$, and
which can be blown down. 

\item[(ii)] 
Let $e, e', $ and  $e''$ be  adjacent edges
with outward conormals $\alpha$, $\alpha'$, and $\alpha''$,
respectively.
If $\alpha'$ is  not a positive 
linear combination of $\alpha$ and $\alpha''$,
then there is an edge which is not equal to $e, e'$ or $e''$
which can be blown down.
\end{itemize}
\end{lemma}

\begin{proof}
We begin with the first claim, following the
proof in \cite{Fu}.
Let $e = e_1 , e_2, \ldots, e_k = e'$ be a 
sequence of
edges in $\Delta$ with outward conormals $\alpha_1,\ldots,\alpha_k$, 
respectively.
Since $\Delta$ has more than four 
edges,
we may assume that $k > 3$.
Since $\alpha_1$ and $\alpha_2$ form an integral basis, we
may write $\alpha_j = -a_j \alpha_1 + b_j \alpha_2$ for each $j$,
and set $c_j = a_j + b_j$.
Since for each $j$ there is an integer $d_j$ such that $\alpha_j = \frac{1}{d_j}( \alpha_{j-1} + \alpha_{j+1})$,
 we see that
$c_j = \frac{1}{d_j}( c_{j-1} + c_{j+1})$. 
Note that $d>0$ since $e_1$ and $e_k$ are parallel. Hence
$c_3 = 1+d_2\geq 2 $ and $c_k = 1$. It follows that there exists
$\ell \geq 3$ with $c_\ell > c_{\ell+1}$ and $c_\ell \geq c_{\ell-1}$.
In this case, $d_\ell$ must be $1$,
and so $e_\ell$ can be blown down.

To prove (ii), 
note first that
every smooth convex 
polygon with more that three edges is the blowup of a trapezoid
and so must have (at least)
two edges which are parallel.  Our assumptions imply that 
$e'$ together with extensions of the two edges 
$e,  e''$ form a triangle.  It follows that
$e'$ must be parallel to another edge.
We can now apply the first part.
\end{proof}

\section{Examples of essential mass linear functions}
\labell{sec:ex}

In this section we give examples of 
(essential) 
mass linear functions on
polytopes.  We
consider two basic types of 
examples: bundles and 
blowups of double expansions.
The examples  
that we consider 
include $\Delta_2$ bundles over $\Delta_1$. 
By   
Proposition~\ref{prop:3d},
this implies
that in this section we construct every essential mass linear function
on a smooth polytope of dimension at most $3$.
More 
importantly,  the examples we consider include all the types of polytopes
described in Theorem~\ref{thm:4d}.  Therefore,  
we also construct
every essential mass linear function on a smooth $4$-dimensional polytope.

The results in this section are not needed for the proof
of the main theorem
since that  gives necessary rather than sufficient conditions
for mass linearity. 
 In fact, except for Corollary~\ref{cor:5plus}
(and several remarks),
this section  and \S\ref{sec:4dim} are completely independent.

\subsection{Essential mass linear functions on bundles}\labell{ss:calcul}

In this subsection, we find all essential mass linear functions
on each of the bundles described in part (a) of Theorem~\ref{thm:4d}.

To begin, 
consider
a polytope $Y$ that is  $\Delta_k$ bundle over $\Delta_1$. 
By definition there exists $a = (a_1, a_2,\ldots,a_k) \in \R^k$,
an identification of $\ft$ with $\R^{k+1}$, and $\kappa \in \R^{k+3}$ so
that
\begin{gather}\labell{eq:Yb}
Y  =  \bigcap_{i=1}^{k+3} 
\big\{ x \in \big(\R^{k+1}\big)^* \; \big| \;
\langle x,\eta_i \rangle \leq \kappa_i  \big\}, \quad \mbox{where}  \\
\notag
\eta_i =  -e_i \;\; \mbox{for all }   1 \leq i \leq k,  \;\;
\eta_{k+1} = \sum_{i=1}^k e_i,  \;\;
\eta_{k+2} = - e_{k+1},\;\;
 \mbox{and} \;\; \eta_{k+3} = e_{k+1} +   \sum_{i=1}^k a_i e_i.
\end{gather}
Here, $e_1,\ldots,e_{k+1}$ is the standard basis for $\R^{k+1}$. 
The fiber facets are $F_1,\dots,F_{k+1}$.
Conversely,  \eqref{eq:Yb}
describes a $\Delta_k$ bundle over $\Delta_1$ exactly if $\kappa \in \Cc_a$, where
\begin{gather*}
\Cc_a = \Bigg\{ \kappa \in \R^{k+3}  \; \Bigg| \; \sum_{i=1}^{k+1} \kappa_i > 0 \mbox{ and } 
\sum_{i=1}^{k} a_i \kappa_i 
+ \kappa_{k+2} + \kappa_{k+3} 
> 
\max (0,a_1,\dots,a_k) \sum_{i=1}^{k+1} \kappa_i 
\Bigg\}.
\end{gather*}

\begin{proposition}\labell{DkoverD1}
Let $Y$ be the $\Delta_k$ bundle over $\Delta_1$ associated to
$a \in \R^k$ as in \eqref{eq:Yb} above; set $a_{k+1} = 0$.
Then $H \in \ft$ is a mass linear function on $Y$ exactly if 
\begin{gather*}
H = \sum_{i=1}^{k+3} \gamma_i \eta_i, \quad \mbox{where} \quad 
\gamma_{k+2} + \gamma_{k+3} =  
\sum_{i=1}^{k+1} \gamma_i = 
\sum_{i=1}^k a_i \gamma_i = 0.
\end{gather*}
In this case, $\langle H, c_Y \rangle = \sum_{i=1}^{k+3} \gamma_i 
\kappa_i$.   Moreover,
$H$ is inessential exactly if
$$
\sum_{a_i = \alpha} \gamma_i = 0 \quad \forall \ \alpha \in \R,
$$
where 
the sum is over $i \in \{1,\dots,k+1\} $ such that $a_i = \alpha$.
\end{proposition}

\begin{rmk}\labell{rmk:Dk}\rm (i)  Because $\sum_{i=1}^{k+1} \eta_i =0$ 
and $\eta_{k+2}+\eta_{k+3} = \sum_{i\le k} a_i\eta_i$, each $H \in \ft$ can be written as 
$\sum_{i=1}^{k+3} \gamma_i \eta_i$ where  
$\gamma_{k+2} + \gamma_{k+3} =  
\sum_{i=1}^{k+1} \gamma_i = 0$.  Therefore the most significant condition 
on $H$ above 
is that 
$\sum_{i\le k} a_i\ga_i = 0$.  Note that this holds
 for all $H$ exactly if 
$a_1 = \dots = a_k = 0$, that is, exactly if 
$Y$ is the product $\Delta_1 \times \Delta_k$.
Moreover, in this case every $H \in \ft$ is
inessential.
(More generally,  by [I, Theorem~1.20], the only 
polytopes for which all vectors $H \in \ft$ are mass linear are products of 
simplices.)
Otherwise, $Y$ admits a $k$-dimensional family of mass linear functions
and the inessential mass linear functions 
form a subspace of dimension $k + 2 - \left|\{a_1,\dots,a_k,0 \}\right|$.
\MS

\NI (ii) The polytope $Y$ is smooth exactly if $a \in \Z^k$.
In this case,
the corresponding toric manifold $M_Y$ is the $\C P^k$ bundle over $\C P^1$
associated to the action 
$$
e^{i\theta} \cdot [z_1:\dots:z_{k+1}]
= [e^{-a_1i\theta} z_1: \dots :  e^{-a_ki\theta} z_k: z_{k+1}].
$$
The polytope $Y$ is determined up to translation by two constants, namely 
$\la : = \sum_{i=1}^{k+1} \kappa_i$ and $h: =  \sum_{i=1}^k a_i \kappa_i + \kappa_{k+2} 
+ \kappa_{k+3}$;  
cf.\ the proof of Lemma \ref{le:calcul} below.
Note that
$\la$ determines the \lq\lq size" of the fiber, 
while $h$ determines that of (one section of) the base.\MS

\NI (iii)  At first glance, the restrictions 
on $H$ in Proposition~\ref{DkoverD1} may seem mysterious; we will now
give a geometric motivation.
Suppose that $H$ is mass linear and write
$\langle H, c_Y \rangle = \sum_{i=1}^{k+3} \gamma_i
\kappa_i$. Then  Lemma~\ref{le:Hsum} implies that
$H = \sum_{i=1}^{k+3} \gamma_i \eta_i$.
Because
$\De$ is a bundle over $\De_1$ with base facets
$F_{k+2}$ and $F_{k+3}$, Proposition~\ref{prop:bund} implies
that we must have
$\ga_{k+2} + \ga_{k+3} = 0$, the first condition in Proposition~\ref{DkoverD1}.
The remaining two conditions can
be interpreted in
terms of the existence of a vector
$\xi_H\in \ft^*$ such that  $\ga_i = \langle \eta_i, \xi_H\rangle$
for all $i$.
If we assume only that  $\gamma_1,\dots,\gamma_{k+3}$
satisfy
the first condition $\ga_{k+2}+ \ga_{k+3}= 0$ then,
because of the linear relations between  the conormals $\eta_i$,
the other two conditions
are satisfied exactly if  there exists a vector $\xi_H$ such that
$\ga_i= \langle\eta_i,\xi_H\rangle$ for all $i$.
 Therefore,
we do not need Proposition \ref{DkoverD1} to see that -- if we express mass 
linear functions as a linear combination of the $\eta_i$ using the natural coefficients provided by Lemma~\ref{le:Hsum} -- any mass linear $H \in \ft$
that is generated by some $\xi_H$ in the sense of Definition \ref{def:full} must satisfy the conditions of Proposition~\ref{DkoverD1}.  Conversely,  it follows from
Proposition \ref{DkoverD1} that
every 
mass linear $H$ is generated by some $\xi_H$; cf. Lemma~\ref{le:full1}.

\end{rmk}

The proof of Proposition~\ref{DkoverD1} rests mainly on the following direct calculation.

\begin{lemma}\labell{le:calcul}  
Let $Y$ be the $\Delta_k$ bundle over $\Delta_1$ associated to
$a \in \R^k$ as in \eqref{eq:Yb} above.
Let $H = \sum_{i=1}^{k+1}\ga_i\eta_i$, where $\sum_{i=1}^{k+1}\ga_i=0$. 
Then $H$ is mass linear on $Y$ 
if and only if 
$$
\sum_{i=1}^k\ga_i a_i = 0; \quad \mbox{in this case,} \quad \langle H,c_Y \rangle = 
\sum_{i=1}^{k+1} \gamma_i \kappa_i.
$$
\end{lemma}

\begin{proof}  
As a first step, fix $\kappa_1 = \dots = \kappa_k = 0$ and $\kappa_{k+2} = 0$,
and let $\kappa_{k+1} = \lambda$ and $\kappa_{k+3} = h$.
Let 
$\Delta_k^\la \subset \R^k$ 
denote the  $k$-simplex described by the inequalities
$$
x_i \geq 0 \quad \mbox{for all } 1 \leq i \leq k \qquad
\mbox{and} \qquad \sum_{i=1}^k x_i \leq  \lambda.
$$
An elementary calculation shows that for any non-negative 
integers $i_1,\ldots,i_k$, 
\begin{equation}\labell{integral}
\int_{\Delta_k^\la} x_1^{i_1} x_2^{i_2} \cdots x_k^{i_k} = 
\frac{ i_1!\, i_2! \cdots i_k!\, \lambda^{I + k}}{  (I + k)!},
\quad \mbox{where }
I  = \sum_{j=1}^k i_j. 
\end{equation}
Here by convention $0! = 1$.
Furthermore, both here and elsewhere 
we integrate with respect to
the standard measure $dx_1 \cdots dx_k$ on $\R^k$.
Since $Y$ is a $\Delta^\la_k$ bundle over $\Delta_1$,
 $Y$ has volume 
$$
V = \int_{\Delta_k^\la} \Big( h - \sum_{i = 1}^k a_i x_i  \Big)=
 \frac{(k+1) h \lambda^k -  \big(\sum_{i = 1}^k a_i\big) \lambda^{k+1} }{(k+1)!}.
$$
For $j \neq k+1$,
the moment $\mu_j$ 
of $Y'$ along the $x_j$ axis is
$$
\mu_j 
= \int_{\Delta_k^\la} \Big(h x_j - \sum_{i=1}^k a_i x_i x_j\Big)
= \frac{(k+2) h \lambda^{k+1} - 
 \big(a_j+ \sum_{i =1}^k a_i\big) \lambda^{k+2} }{(k+2)!}.
$$
Let $c_j := \mu_j/V$ 
denote the $j$'th component of the center of mass.
For $j \neq k+1$, 
$$
c_j = 
\frac{\lambda}{k+2}\  \frac{h (k+2)  - 
\lambda \big(a_j + \sum_{i=1}^k a_i\big)}
{h(k+1) - \lambda \sum_{i=1}^k a_i}.
$$
Since $\sum_{i=1}^{k+1} \gamma_i = 0$,  a straightforward calculation shows
that 
\begin{equation}\labell{Yform}
\langle H, c_{Y} \rangle = \sum_{i=1}^k (\gamma_{k+1} - \gamma_i) c_i
= \lambda \left(\gamma_{k+1}  + \frac{\lambda \sum_{i=1}^k \gamma_i a_i}
{(k+2)\big(h(k+1)  - \lambda \sum_{i=1}^k a_i\big)} \right).
\end{equation}
This is linear function of $h$ and $\lambda$ exactly
if $ \sum_{i=1}^k \gamma_i a_i = 0$.
Hence, if $H$ is mass linear, this sum must be zero.

To prove the converse, assume that 
$\sum_{i=1}^k \gamma_i a_i = 0$.
Given $\kappa \in \Cc_a$, note that by Remark~\ref{translate}
\begin{gather*}
Y(\kappa) = Y(0,\dots,0,\lambda,0,h) - (\kappa_1,\dots,\kappa_k,\kappa_{k+2}), \quad \mbox{where} \\
\lambda = \sum_{i=1}^{k+1} \kappa_i \quad \mbox{and} \quad h = 
\sum_{i=1}^k a_i \kappa_i
+ \kappa_{k+2} + \kappa_{k+3}. 
\end{gather*}
Hence, 
\eqref{Yform} implies that
$$
\langle H, c_{Y}(\kappa) \rangle  =
\langle H, c_Y(0,\dots,0,\lambda,0,h) \rangle - 
\langle H, (\kappa_1,\dots,\kappa_k,\kappa_{k+2}) \rangle
= \sum_{i=1}^{k+1} \kappa_i \gamma_i.
$$
This completes the proof. \end{proof}

We are now ready to prove our first main proposition.

\begin{proof}[Proof of Proposition~\ref{DkoverD1}.]
As explained in Remark \ref{rmk:Dk} (i) above,
every $H \in \ft$ 
can be written uniquely as
$H = \sum_{i=1}^{k+3} \gamma_i \eta_i$, where 
$\gamma_{k+2} + \gamma_{k+3}  = \sum_{i=1}^{k+1} \gamma_i = 0$.
By Lemma~\ref{le:equiv}, $F_{k+2}$ and $F_{k+3}$ are equivalent. 
Hence, $\gamma_{k+2} \eta_{k+2} + \gamma_{k+3} \eta_{k+3}$ is inessential,
and so
by Proposition~\ref{prop:inessential}
$$\langle \gamma_{k+2} \eta_{k+2} + \gamma_{k+3} \eta_{k+3} , c_Y(\kappa) \rangle = \gamma_{k+2} \kappa_{k+2}
+ \gamma_{k+3} \kappa_{k+3}..$$
On the other hand, by Lemma~\ref{le:calcul},
$ \Tilde{H} = \sum_{i=1}^{k+1} \gamma_i \eta_i$ is
mass linear
exactly if $\sum_{i=1}^k a_i \gamma_i = 0$, in which case
$$
\langle \Tilde{H}, c_Y \rangle = \sum_{i=1}^{k+1} \gamma_i \kappa_i.$$
The first two claims follow immediately.

To establish the conditions under which $H$ is inessential, note first
that
since $\gamma_{k+2} \eta_{k+2} + \gamma_{k+3} \eta_{k+3}$ is inessential,
$H$ is inessential exactly if $\Tilde{H}$ is inessential.
Further,
Lemma~\ref{le:equiv} implies that
for each pair $\{i,j\} \subset \{1,\dots,k+1\}$, we have
$F_i\sim F_j$ exactly if $a_i = a_j$.
Therefore  
the equivalence classes  of the relation $\sim$ on $\{F_1,\dots, F_{k+1}\}$
are precisely the sets   $\{F_i: a_i=\al\}$. 
Since $\Tilde{H}$ can be written uniquely as $\Tilde{H} = \sum_{i=1}^{k+1} \gamma_i
\eta_i$,
where $\sum_{i=1}^{k+1} \gamma_i = 0$, 
 it is inessential
exactly if $\sum_{a_i=\al}\ga_i = 0$ for each $\al$.
\end{proof}

Proposition~\ref{DkoverD1} 
immediately gives  all essential mass linear functions for $\Delta_3$ bundles over $\Delta_1$,
the polytopes $\ov \De$
 in case (a1) of Theorem~\ref{thm:4d}.

\begin{corollary}\labell{cor:Mabc}
Let $Y$ be the $\Delta_3$ bundle over $\Delta_1$ associated to
$a \in \R^3$ as in \eqref{eq:Yb} above; set $a_4 = 0$.
Then $H \in \ft$ is a mass linear function on $Y$ exactly if 
\begin{gather*}
H = \sum_{i=1}^{6} \gamma_i \eta_i, \quad \mbox{where} \quad 
\gamma_{5} + \gamma_{6} =  
\sum_{i=1}^{4} \gamma_i = 
\sum_{i=1}^3 a_i \gamma_i = 0.
\end{gather*}
In this case, 
$\langle H, c_Y \rangle = \sum_{i=1}^{6} \gamma_i \kappa_i$. 
Moreover, if $a_i \neq a_j$ for all $1 \leq i < j \leq 4$,
then $H$ is inessential exactly if 
$\gamma_1 = \gamma_2 = \gamma_3 = \gamma_4 = 0$,
and so there is a $1$-dimensional subspace of inessential functions
in the $3$-dimensional space of mass linear functions.
If $a_i=a_j$ for only one such pair $\{i,j\}$, there is
a $2$-dimensional family of inessential functions -- for example, if
$a_1 = a_2$ but $a_i \neq a_j$ for all $2 \leq i < j \leq 4$ 
then $H$ is inessential exactly if $\gamma_3 = \gamma_4 = 0$.
If more than one such equality holds, then every $H$ is inessential.
\end{corollary}

Next we consider  the 
second class of polytopes  in
Theorem~\ref{thm:4d}.

\begin{defn}\labell{def:121}    
 A {\bf  $\mathbf{121}$-bundle} 
is a smooth polytope
$Z$ that is a $\Delta_1$ bundle over a polytope $Y$ which
is itself a  $\Delta_2$ bundle over $\Delta_1$.  
\end{defn}

Given a $121$-bundle $Z$, it is easy to see that we may
identify $\ft$ with $\R^4$ so that its outward conormals  are
\begin{gather}\labell{eq:de'}
\Tilde \eta_0 = (1,0,0,0), \; \Tilde \eta_1= (-1,0,0,0),
\; \eta'_2 = (0,-1, 0,0), \;
\eta'_3 = (0,0,-1,0), \\ \notag\eta'_4 = (d,1,1,0) , \;
\eta'_5 = (0,0,0,-1), \; \mbox{and} \; \eta'_6 = (a_1,a_2,a_3,1)
\end{gather}
for some $a \in \Z^3$ and some integer $d \ge 0$. 
Here the fiber $\De_1$ of $Z$
lives in the first coordinate direction, and $Y$ is the 
$\De_2$ bundle over $\De_1$  associated to $(a_2,a_3)$ and with conormals equal to the projections 
of the $\eta'_j$ for $2 \leq j \leq 6$ onto the last three coordinates.

\begin{proposition}
\labell{prop:doublebundle}
Let $Z$ be a  $121$-bundle as in \eqref{eq:de'} above.
Then $H \in \ft$ is a mass linear function on  $Z$ exactly if
\begin{gather*}
H =  \gamma_0 \Tilde \eta_0 + \gamma_1 \Tilde \eta_1 +
\sum_{i=2}^6 \gamma_i \eta'_i, \quad \mbox{where}  \\
\gamma_0 + \gamma_1 = d \gamma_0 = a_1 \gamma_0 = 0\ \mbox{ and }\
\gamma_2  + \gamma_3 + \gamma_4 = 
a_2 \gamma_2  + a_3 \gamma_3 = 
\gamma_5 + \gamma_6 
 = 0.
\end{gather*}
In this case 
$ \langle H, c_Z \rangle = \sum_{i=0}^6 \gamma_i \kappa_i.$
If $a_2 a_3 (a_2 - a_3) \neq 0$, then $H$ is inessential exactly if 
$\gamma_2
= \gamma_3 = \gamma_4 = 0$; otherwise, every $H$ is inessential.
\end{proposition}

\begin{proof}
Since $Z$ is a bundle with fiber $\Delta_1$, 
the fiber facets $\Tilde F_0$ and $\Tilde F_1$ are not pervasive.
If $d \neq 0$ or if $a_1 \neq 0$, then
these facets also are not flat.
Therefore, if $H \in \ft$ is mass linear
then Proposition~\ref{prop:asym} implies that $\Tilde F_0$ and $\Tilde F_1$
are symmetric.
On the other hand, if $d = a_1 = 0$  then $Z = \Delta_1 \times Y$, and so we
can also view $Z$ as a $Y$ bundle over $\Delta_1$ with {\em base facets}
$\Tilde F_0$ and $\Tilde F_1$.
Therefore, Proposition~\ref{prop:bund} implies that
$H$ is mass linear exactly if 
$H = \gamma_0 \Tilde \eta_0 + \gamma_1 \Tilde \eta_1 + \Tilde H$, 
where $\gamma_0 + \gamma_1 = 0$ and $\Tilde H$ is a mass linear function
on $Z$ so that $\Tilde F_0$ and $\Tilde F_1$ are symmetric.
Therefore, the
result is an immediate consequence of
Proposition~\ref{prop:lift} and Proposition~\ref{DkoverD1}.
\end{proof}

The following corollary is immediate.

\begin{cor}\labell{cor:db}
Let $Z$ be a $121$-bundle that admits a mass linear function so
that every facet is asymmetric. Then $Z = \De_1 \times Y$, where
$Y$ is a $\De_2$ bundle over $\De_1$.
\end{cor}

We now consider the third class of polytopes in Theorem \ref{thm:4d}.
Let $\De \subset \ft^*$ be a smooth $\De_2$ bundle over a $2$-dimensional 
polygon $\Hat\Delta \subset \R^2$. 
We aim 
to find all essential mass linear functions on $\Delta$.

First, we describe $\De$ in more detail.
Let $\Hat\eta_1,\ldots,\Hat\eta_k$ be the outward conormals to $\Hat\De$.  
Assume that the edges of $\Hat\Delta$ are labelled 
in {\bf order of 
adjacency}, that is,
so that  $e_i$ is adjacent to $e_{i-1}$ and $e_{i+1}$ for all $i$
(where we interpret the $i$ in cyclic order, 
that is,
moduli $k$.)
Then there is an identification of  $\ft^*$ with $\R^4$ 
and a pair of integers $(b_1^i,b_2^i)$ for all $1 \leq i \leq k$
so that 
the outward conormals to the fiber facets are
\begin{equation}\labell{eqD1}
\eta_{1} = (-1,0,0,0),\quad \eta_{2} = (0,-1,0,0),\quad \mbox{and} \quad
\eta_{3}=(1,1,0,0), 
\end{equation}
and the outward conormals to the base facets are
\begin{equation}\labell{eqD2}
\Hat\eta_i'  = (b_1^i,b_2^i,0,0) + \Hat \eta_i  \quad \mbox{for all } 1 \leq i \leq k,
\end{equation}
where we identify $\Hat\eta_i\in \R^2$ with its image in the plane $x_1=x_2=0$ of $\R^4$.
Moreover, we may assume that $(b_1^1,b_2^1) = (b_1^2,b_2^2) = (0,0)$.

\begin{proposition}\labell{prop:polybundle}
Let $\Delta$ be the $\Delta_2$ bundle over $\Hat\Delta$ defined above.
Let $P(\Hat\ka_1,\dots,\Hat\ka_k)$ be the 
polynomial
which gives the area of $\Hat\De$ 
for all $\Hat \kappa \in 
\Cc_{\Hat \Delta}$.  
Then  $H \in \ft$ 
is mass linear on $\Delta$ exactly if
$H = \Tilde H + \Hat H$,
where 
$\Hat H = \sum_{i=1}^k \Hat\gamma_i\,\!' \Hat\eta_i\,\!' \in \ft$ 
is the lift  
of an inessential
function on $\Hat\De$ and 
$\Tilde H$ 
is a  mass linear function on $\Delta$ of the form
$$
\Tilde H 
= \gamma_1 \eta_1 + \gamma_2 \eta_2 + \gamma_3 \eta_3
\quad \mbox{ with }\gamma_1 + \gamma_2 + \gamma_3 = 0.
$$
Moreover, $\Tilde H$ 
is mass linear if it is zero or if
there are 
real numbers  $r_3,\ldots, r_k$ so that 
\begin{itemize}
\item [(i)] $(b^i_1,b^i_2) = r_i (\gamma_2,-\gamma_1)$  for all 
$i \in \{3,\dots,k\}$, and 
\item [(ii)] either $P(0,0,r_3,\ldots,r_k)  = 0$  or $\gamma _1 \gamma_2 \gamma_3 = 0$.
\end{itemize}
In this case, 
$\langle  H, c_\Delta \rangle
= \sum_{i=1}^3 \gamma_i \kappa_i + \sum_{i=1}^k \Hat\gamma_i\,\!' \Hat\kappa_i\,\!'$.  
Finally,  $H$ is inessential exactly if
the bundle is trivial or if $\gamma_1 \gamma_2 \gamma_3 = 0$.
\end{proposition}

Before proving this proposition, we 
consider the case that
$\Hat\De$ is a triangle or quadrilateral.

\begin{cor}\labell{cor:22bundle}
Let $\Delta$ be a $\De_2$ bundle over $\De_2$ defined as above.
\begin{itemize}
\item [(i)] Every mass linear function on $\De$
is inessential.
\item[(ii)] 
If $b_1^3 b_2^3 (b_1^3-b_2^3)\ne 0$, then $\De$ has a $2$-dimensional
family of mass linear functions,
and the fiber facets are symmetric for every mass linear
function.
\end{itemize}
\end{cor}

\begin{proof}
If $\Hat\Delta$ is a simplex,
then $P(0,0,r_3) \neq 0$ unless $r_3 = 0$. 
Therefore, 
both claims follow
immediately from the  proposition above.  
\end{proof}

In contrast, many $\De_2$ bundles over quadrilaterals 
admit essential mass
linear functions. For example, let $\De$ be a generic 
$\Delta_1 \times \Delta_2$ bundle over $\Delta_1$.  As we mentioned
in the introduction, this implies that $\Delta$ can either be
viewed as a
 $121$-bundle
or as a $\Delta_2$ bundle over a trapezoid. 
Hence, we can use either Proposition~\ref{prop:doublebundle} or
Proposition~\ref{prop:polybundle} to show that $\De$ admits
essential mass linear functions.
In fact, if $W$ is a generic $\Delta_2 $ bundle over $\Delta_1$,
then $\Delta_1 \times W$ admits essential mass linear functions
with $7$ asymmetric facets. 
We spell out the details here because this example is rather special; 
as explained in 
the remarks (2) and (3) after Theorem~\ref{thm:4d},
no other  $4$-dimensional
polytopes  admit an essential mass linear function with more than $6$
asymmetric facets.

\begin{cor}\labell{cor:7}
Let $H \in \ft$ be an essential mass linear function on a
polytope $\Delta \subset \ft^*$ which is a $\Delta_2$
bundle over a polygon $\Hat \De$.   If $\Delta$ has 
more than six
asymmetric facets, then $\Delta =  
\Delta_1 \times Y$,
where $Y$ is a $\Delta_2$ bundle over $\Delta_1$.
\end{cor}

\begin{proof}
By Proposition~\ref{prop:polybundle},  this is impossible unless
$\Hat \De$ admits an inessential function with four asymmetric
facets.  Hence, $\Hat \De = \De_1 \times \De_1$.
We may assume that $\Hat\eta_1 = (-1,0), \Hat\eta_2 = (0,-1), 
\Hat\eta_3 = (1,0)$ and $\Hat\eta_4 = (1,0)$.
Then  $P(0,0,r_3,r_4) = r_3 r_4$,
so that $P(0,0,r_3,r_4)=0$ exactly if $r_3=0$ or $r_4 = 0$,
that is, exactly if 
$(b^3_1,b^3_2) = (0,0)$ or
$(b^4_1,b^4_2) = (0,0)$.
\end{proof}

The proof of Proposition \ref{prop:polybundle} is based on the following lemma.

\begin{lemma}\labell{calcul2}
Let $\Delta$ be the $\Delta_2$ bundle over $\Hat\Delta$ defined above.
Let $P(\Hat\ka_1,\dots,\Hat\ka_k)$ be the 
polynomial which gives the area of $\Hat\De$ for all 
$\Hat \kappa \in \Cc_{\Hat{\Delta}}$.  
Let $H = \sum_{i=1}^3 \gamma_i \eta_i$, where $\sum_{i=1}^3 \gamma_i = 0$.
If $H$ is not zero, then  $H$ is mass linear on $\Delta$  exactly if
there are real numbers 
$r_3,\ldots, r_k$ so that 
\begin{itemize}
\item [(i)] $(b^i_1,b^i_2) = r_i (\gamma_2,-\gamma_1)$  for all 
$i \in \{3,\dots,k\}$, and 
\item [(ii)] either $P(0,0,r_3,\ldots,r_k)  = 0$  or $\gamma _1 \gamma_2 \gamma_3 = 0$.
\end{itemize}
In this case, $\langle H, c_\Delta \rangle = \sum_{i=1}^3 \gamma_i \kappa_i$,
where $\kappa_i$ is the support number of the fiber facet $F_i$.
\end{lemma}

\begin{proof}
Let $G_i$ be a base facet.
Since $\Hat\Delta$ is smooth, there is an integer $m_i$ so
that 
\begin{equation}\labell{eq:m}
m_i \Hat\eta_i = \Hat\eta_{i-1} + \Hat\eta_{i+1},
\end{equation}
where as usual we interpret the $i$ in cyclic order. 
The facet $G_i$ is a $\Delta_2$ bundle over $\Delta_1$ with fiber
facets 
$F_1 \cap G_i,F_2\cap G_i,$ and $F_3 \cap G_i$,  
and base facets $G_{i-1} \cap G_i$ and $G_{i+1} \cap G_i$.
As such, it 
is determined by a pair of integers $(a^i_1,a^i_2)$.
One can check that
\begin{equation}\labell{sides}
(a^i_1,a^i_2) = (b^{i-1}_1, b^{i-1}_2) - m_i (b^{i}_1,b^i_2) + (b^{i+1}_1,
b^{i+1}_2).
\end{equation}

If $H  \in \ft $  
is mass linear on $\Delta$, then 
by Proposition~\ref{prop:symcent},
it must also be mass linear on $G_i$. 
Hence, Lemma~\ref{le:calcul} implies that $\gamma_1 a^i_1 + 
\gamma_2 a^i_2 = 0$ for all $i$.
Since $(b^1_1,b^1_2)
= (b^2_1,b^2_2) = (0,0)$ by assumption, this and Equation~(\ref{sides})
together imply that $\gamma_1 b^i_1 + \gamma_2 b^i_2 = 0$ for all $i$.
Therefore,  since $H \neq 0$,
there is a constant $r_i$ for each $1 \leq i \leq k$  such that
\begin{equation}\labell{multiple}
(b^i_1,b^i_2) = r_i (\gamma_2,-\gamma_1);
\end{equation}
note that $r_1=r_2=0$.

So now assume that \eqref{multiple} holds for all $i$.
As in Lemma~\ref{le:calcul}
it is convenient first to 
consider the case that 
the support numbers of the fiber 
facets $F_1$ and $F_2$ are $0$.
Let $\lambda$ denote the support number of $F_3$
and $\Hat\kappa_i$ denote the support number of the base facet $G_i$.
To get the volume $V$  of $\Delta$ we  integrate over the  
simplex $\Delta_2^\lambda$
the function  which gives  the area of the intersection of $\De$ with the 
$2$-plane where $x_1$ and $x_2$ are constant. 
Each such section is 
affine equivalent to the base polygon $\Hat\De$ with structural
constants 
$\Hat\kappa_i' = \Hat\kappa_i - r_i
(\ga_2 x_1 -\ga_1 x_2)$.
Since $P$ is a 
homogeneous
quadratic function  of the 
support numbers, the area of this 
section is 
$$
P\big(\Hat\kappa_i - r_i(\gamma_2 x_1 - \gamma_1 x_2)\big) = 
Q_0 + Q_1 (\gamma_2 x_1 - \gamma_1 x_2) +
Q_2(\gamma_2 x_1 - \gamma_1 x_2)^2,
$$  where  each 
$Q_d$ is a polynomial of degree $d$
in the $r_i$ and of degree $2-d$ in the $\Hat\kappa_i$; moreover
$Q_2$ is $P(r_1,\ldots, r_k)$. 
Therefore,
by \eqref{integral}
\begin{eqnarray*}
 V &=&   
 \int_{\Delta_2^\la}  Q_0 + Q_1 (\gamma_2 x_1 -\gamma_1 x_2) + 
Q_2 (\gamma_2 x_1 - \gamma_1 x_2)^2 \\
 & = & \frac {\la^2}{12}\Bigl(6 Q_0 +  2 (\gamma_2 - \gamma_1) Q_1 \la + 
 \big(\gamma_2^2  - \gamma_2 \gamma_1  + \gamma_1^2\big) Q_2 \la^2\Bigr).
\end{eqnarray*}
Similarly, the
moment $\mu_1$ along the $x_1$ axis is
\begin{eqnarray*}
 \mu_1&=&  
 \int_{\Delta_2^\la}  Q_0 x_1 + Q_1 (\gamma_2  x_1  - \gamma_1 x_2)x_1 + 
Q_2 (\gamma_2  x_1 - \gamma_1 x_2)^2 x_1 \\   
  & = & \frac {\la^3}{120}\Bigl(20  Q_0 +   
(10 \gamma_2 -  5 \gamma_1) Q_1 \la + 
  \big(6\gamma_2^2 - 4\gamma_1 \gamma_2 + 2 \gamma_1^2\big)  Q_2 \la^2\Bigr).
\end{eqnarray*}
By symmetry the moment $\mu_2$ along the $x_2$ is given by
interchanging $\gamma_2$ and $-\gamma_1$.  
As before, a straightforward (though  tedious) calculation
shows that
\begin{eqnarray*}
\langle H, c_\Delta \rangle & = &
(\gamma_3 - \gamma_1) \frac{\mu_1}{V} + (\gamma_3 - \gamma_2) \frac{\mu_2}{V}\\
 &=& \lambda \left( \gamma_3 + \frac{ \gamma_1 \gamma_2 \gamma_3 Q_2 \la^2}
{5\Bigl(6 Q_0 + 2(\gamma_2 - \gamma_1) Q_1 \lambda  + 
\big(\gamma_2^2 - \gamma_1 \gamma_2 + \gamma_1^2\big) Q_2 \lambda^2\Bigr)} \right).
\end{eqnarray*}
This is a linear function exactly if $\gamma_1 \gamma_2 \gamma_3 = 0$
or $Q_2 = 0$.

Together, these two paragraphs imply that if $H$ is mass linear on
$\Delta$, then $(b_1^i,b_2^i) = r_i(\gamma_2,-\gamma_1)$ for all $i$ and
either $P(0,0,r_3,\dots,r_k) = 0$ or $\gamma_1 \gamma_2 \gamma_3 = 0$.
It remains to show that if $H$ is mass linear, then
we must have $\langle H, c_\Delta \rangle =
\sum_{i=1}^3 \gamma_i \kappa_i$.  We calculated 
$\langle H, c_\Delta \rangle = \la \ga_3$ above in the special case when $\ka_1=\ka_2=0$ and $\ka_3=\la$.    Just as at the end of the proof of 
Lemma~\ref{le:calcul}, the general case follows by using Remark \ref{translate}.
\end{proof}

We are now ready to complete the proof.

\begin{proof}[Proof of Proposition~\ref{prop:polybundle}]
If $\Hat \Delta = \Delta_2$, then by Proposition~\ref{prop:bund}
we can write 
$H ' =\Hat H + \Tilde H$, where
\begin{itemize}
\item $\Hat H$ is inessential and the fiber facets are $\Hat H$-symmetric, and
\item $\Tilde H$ is mass linear and the base facets are $\Tilde H$-symmetric.
\end{itemize} 
On the other hand,
if $\Hat \Delta$ contains more than three edges, then the 
base facets are the nonpervasive facets.
Hence, in this case the
same claim follows from Proposition~\ref{prop:flat}.
It then follows from
Proposition \ref{prop:lift} that 
$\Hat H$ is the lift of
an inessential  function on $\Hat\De$.

Now consider  
$\Tilde H$.   Since all its asymmetric facets are fiber facets,
Lemma~\ref{le:Hsum}  implies that  
$\Tilde H$ lies in the span of
the conormals to the
fiber facets. Since $\sum_{i=1}^3 \eta_i = 0$, this means that
there are constants $\gamma_i$ so that  
$\Tilde H = \sum_{i=1}^3 \gamma_i \eta_i$,
where $\sum_{i=1}^3 \gamma_i = 0$.
Therefore, by Lemma~\ref{calcul2}, (i) and (ii) hold.  
Moreover $H$ is inessential exactly if  
$\Tilde H$ is.   If $\ga_1\ga_2\ga_3\ne 0$ then $H'$ is inessential 
exactly if the three fiber facets $F_1,F_2, F_3$ are equivalent.
But by \eqref{eqD2} this happens only if all $b_j^i=0$, that is, if 
 the bundle $\De\to \Hat\De$ is trivial.
\end{proof}

We now consider the topological  
implication of
Proposition~\ref{prop:polybundle}.

\begin{prop}\labell{prop:geo3dep}  Let $\De$ be a $\De_2$-bundle  over a polygon 
$\Hat\De \subset \R^2$,
and let $M_\De$ and $M_{\Hat \De}$ denote the associated toric manifolds.
Then $\De$ admits an essential mass linear function exactly if
there
exist integers $\gamma_1$ and $\gamma_2$,
and  
a $T_{\Hat\De}$-equivariant principal 
$S^1$-bundle $L$ over $M_{\Hat \Delta}$
such that
\begin{itemize}
\item [(i)] the (ordinary) Euler  class $\chi \in 
H^2(M_{\Hat \Delta};\Z)$ of $L$ is not trivial but
has vanishing square,
\item [(ii)] $\gamma_1 \gamma_2 (\gamma_2 - \gamma_1) \neq 0$, and
\item [(iii)]  $M_\De$ is  $T_{\De}$-equivariantly diffeomorphic to  $L \times_{S^1} \CP^2$,
where $S^1$ acts on $\CP^2$ by 
$\lambda \cdot [z_1:z_2:z_3]
= [\lambda^{\gamma_2} z_1: \lambda^{-\gamma_1} z_2: z_3].$
\end{itemize}
\end{prop}

\begin{proof}
We may assume that $\De$ 
is described by \eqref{eqD1} and \eqref{eqD2},
where $(b_1^1,b_2^1) = (b_1^2,b_2^2) = 0$
and where $\Hat \eta_1,\dots,\Hat \eta_k$ are the outward conormals to $\Hat \De$.
Let $P$ be the polynomial which gives the 
area of $\Hat \De$.
By Proposition~\ref{prop:polybundle}, $\De$ admits an essential mass linear function
exactly if
there exist real numbers $\gamma_1, \gamma_2$ and $r_3,\dots,r_k$
such that 
\begin{itemize}
\item[(a)]
the $r_i$'s are not all zero but
$P(0,0,r_3,\dots,r_k) = 0$, 
\item [(b)]
$\gamma_1 \gamma_2 (\gamma_2 - \gamma_1) \neq 0$, and
\item [(c)] 
$(b_1^i,b_2^i) = r_i (\gamma_2,-\gamma_1)$ for all $i$.
\end{itemize}
Further, by multiplying $\gamma_1$ and $\gamma_2$ by a suitable constant,
we may assume that $\gamma_1$ and $\gamma_2$ are mutually prime integers,
so that each $r_i$ is also in $\Z$.

By  [I, Remark 5.2], 
$M_\De$ is a $\C P^2$ bundle over $M_{\Hat\De}$.
More specifically, 
identify $M_{\Hat\De}$ with the symplectic 
quotient $\C^k/\!/\Hat K$ for a suitable subtorus $\Hat K\subset (S^1)^k$,
and let $(S^1)^2$ act on $\C P^2$ by 
$\lambda \cdot [z_1:z_2:z_3]=
[\lambda_1 z_1: \lambda_2 z_2: z_3]$.
Then 
$M_\Delta$
is the $\CP^2$ bundle associated to the
homomorphism  $\rho \colon \Hat K \to (S^1)^2$
given by
$$
\rho(\exp(x)) = \exp\Big( \sum x_i b_1^i, \sum x_i b_2^i \Big)
\quad\mbox{ for all }x= (x_1,\ldots,x_k) \in \Hat\fk \subset \R^k.
$$
Next observe that the torus $(S^1)^k$ acts on $M_{\Hat\De}$ via its quotient
$T_{\Hat\De} = (S^1)^k/\Hat K$.    
Moreover, there is a one-to-one correspondence between $(S^1)^k$ equivariant
principal
$S^1$-bundles over $M_{\Hat \De}$,
representations of $(S^1)^k$, and $k$-tuples $r \in \Z^k$.
Hence, $M_\De$ is the $\CP^2$ bundle associated to an equivariant 
principal
$S^1$ bundle 
over $M_{\Hat \De}$ exactly if there 
exist 
integers $\gamma_1,\gamma_2$, and $r_1,\dots,r_k$ such
that $(b_1^i,b_2^i) = r_i (\gamma_2,-\gamma_1)$ for all $i$.
In this case, $S^1$ acts on $\CP^2$ by $\lambda \cdot 
[z_1:z_2:z_3]
= [\lambda^{\gamma_2} z_1: \lambda^{-\gamma_1} z_2: z_3]$

Thus conditions (ii) and (iii) in the proposition are equivalent to conditions 
(b) 
and (c).  
To complete the proof we must show that condition (i) is equivalent to 
(a). 
 This is accomplished in
Lemma~\ref{le:geo3dep} below.
\end{proof}

\begin{lemma} \labell{le:geo3dep} 
Let $\Hat \De \subset \Hat \ft^*$ be an $n$-dimensional 
 polytope with facets $F_1,\dots,F_k$;
let $M_{\Hat \De} = \C^k /\!/ \Hat K$ be the associated toric manifold.
Given $r \in \Z^k$,  let $\chi$ be the Euler class of the 
principal $S^1$-bundle associated to the induced homomorphism from 
$\Hat K \subset (S^1)^k$ to $S^1$.
Let
$P(\Hat \kappa)$
be 
the polynomial which 
gives the volume of $\Hat \De(\Hat \kappa)$ for all 
$\Hat \kappa \in \Cc_{\Hat \De}$. Then  
\begin{equation}\labell{eq:chi}
P(r_1,\ldots, r_k) = 0 \;\;\Longleftrightarrow\;\;
\int_{M_{\Hat\Delta}} \chi^n = 0.
\end{equation}
Moreover, if 
$\bigcap_{i=1}^n F_i \neq  \emptyset$  and
$r_1 = \dots = r_n=0$, then
$\chi=0$ exactly if $r_i=0$ for all $n<i\le k$.
\end{lemma}

\begin{proof}
Fix $\Hat \kappa \in \Cc_{\Hat \De}.$
There exists a symplectic form
$\omega$ on $M_{\Hat \De}$  with moment map
$\Hat \Phi \colon M_{\Hat \De} \to \Hat \ft^*$ 
such that $\Hat \Phi(M_{\Hat \De}) = \Hat \De(\Hat \kappa)$.
On the one hand, the symplectic form $\omega$ represents the cohomology
class $\sum_i \kappa_i X_i$, where  $X_i \in H^2(M_{\Hat \De})$
represents the Poincar\'e dual to the compact submanifold
$\Hat \Phi^{-1}(F_i)$ for all $1 \leq i \leq k$.
On the other hand,
since $M_{\Hat \De}$ is a toric manifold the Duistermaat-Heckman measure 
on $\Hat \ft^*$ is given by Lebesgue measure 
on $\Hat \De(\Hat \kappa)$ and vanishes outside $\Hat \De (\Hat \kappa)$.
(Recall that the Duistermaat-Heckman measure is the 
pushforward of the Liouville 
measure $\frac{1}{n!} \omega^n$ on $M$ under the moment map.)
Therefore, 
$$P(\Hat \kappa) = \frac{1}{n!}\int_{M_{\Hat \De}} \omega^n = \frac{1}{n!}\int_{M_{\Hat \De}}
\Big( \sum  \kappa_i X_i \Big)^n \quad \forall\ \Hat \kappa \in \Cc_{\Hat \De}. $$
Since both sides are polynomials, and since $\chi = \sum r_i X_i$, the
first claim follows.

The second 
claim
holds because
the outward conormals to $F_1,\dots,F_n$ form a basis for the lattice 
in $\Hat \ft$. Hence,
by the standard Stanley--Reisner presentation 
for the cohomology ring of a toric manifold,  $H^2(M_{\Hat\De};\Z)$ is freely generated by
$X_{n+1},\dots, X_k$.
\end{proof}

This completes our discussion of bundles that support 
essential mass linear functions.
We end this section with some supplementary results.
First, we 
determine the number of
asymmetric facets for each of the polytopes $\ov \De$ described
in case (a) of
Theorem~\ref{thm:4d}.

\begin{rmk}\labell{rmk:asym}\rm
Let $H \in \ft$ be an essential mass linear function on a polytope $\ov \De$.
\begin{itemize}
\item If $\ov \De$ is a $\Delta_3$ bundle over $\Delta_1$, as in case (a1),
then
at least  $3$  of the $4$
 fiber facets are asymmetric; 
if one base facet is asymmetric then both are.  
Thus the number of asymmetric facets 
can be
anywhere between $3$ and $6$.
\item If $\ov \Delta$ is a $121$-bundle and the conormals
to the three pervasive facets are linearly independent then the three
pervasive facets are asymmetric;
the two fiber facets are symmetric unless $\ov \Delta$ is the product
$\Delta_1 \times Y$, where $Y$ is a $\Delta_2$ bundle over $\Delta_1$, in which
case they may both be asymmetric; finally,
if one of the remaining two 
facets is  asymmetric  then both are.
Thus there are $3$, $5$, or $7$ asymmetric facets,
with $7$ 
impossible unless $\ov \Delta = \Delta_1 \times Y$.

\item If $\ov \De$ is a $\Delta_2$ bundle over a polygon $\Hat \De$  
as in case (a3),
the three fiber facts are asymmetric. The base facets
are symmetric unless $\Hat \De$ is a $\Delta_1$
bundle over $\Delta_1$. In that case,
two base facets that correspond to equivalent facets of $\Hat \De$
may both be asymmetric.
As in the previous case, there can be $3$, $5$, or $7$ asymmetric
facets, with $7$ impossible unless $\ov \Delta = \Delta_1 \times Y$,
where $Y$ is a $\Delta_2$ bundle over $\Delta_1$.
\end{itemize}
The first two claims follow trivially from Corollary~\ref{cor:Mabc} and
Proposition~\ref{prop:doublebundle}. In order to see that the
last claim follows
from Propositions~\ref{prop:inessential}, \ref{prop:2dim}, and
\ref{prop:polybundle}, 
and Corollary~\ref{cor:7},
remember that there is no essential mass linear
function on any $\Delta_2$ bundle over 
$\Delta_2$ by 
Corollary~\ref{cor:22bundle}.
\end{rmk}

By Lemma~\ref{le:symblow} and Proposition~\ref{prop:blow4},
an essential  mass linear function on a polytope will still be  
essential and
mass linear if the polytope  is blown up by either of the two types
of blowups described in Theorem~\ref{thm:4d} -- blowups 
along  symmetric $2$-faces  and blowups of type $(F_{ij},G)$.
Moreover, the
conclusions above allow us to  analyze the ways that the
bundles $\ov \De$ listed in case (a) of Theorem~\ref{thm:4d}
can be blown up 
in 
these ways.
Therefore, we can now find
all essential mass linear functions of the type described
in case (a)  of Theorem~\ref{thm:4d}.

\begin{rmk}\rm\labell{rmk:blowable}
Let $H \in \ft$ be an essential mass linear function on a polytope
$\ov \De$,
where $\ov \De$ is one of the polytopes described in case (a) of
Theorem~\ref{thm:4d}.   
\begin{itemize}
\item  [(i)]
If $\ov \De$ has exactly three asymmetric 
facets,
then it must have symmetric $2$-faces (which can be blown up).
Otherwise,  $\ov \De$ 
 does not have any symmetric $2$-faces.
\item [(ii)]
By Proposition~\ref{prop:blow4}, 
blowups of type $(F_{ij},G)$ are not possible unless  $\ov \De$ has
four asymmetric facets.
However, the bundles in case (a) do not have four asymmetric facets
unless $\ov \De$ is a $\De_3$ bundle over
$\De_1$,
all four fiber facets are asymmetric, and the base facets are symmetric.
In this case,
blowups of this type are possible exactly if there exist fiber
facets $F_i$ and $F_j$ so that $\gamma_i + \gamma_j = 0$,
where $\gamma_k$ is the support number of $F_k$ in 
$\langle c_{\ov \De} , H \rangle.$
\item  [(iii)] By Lemma~~\ref{le:symblow} and Proposition~\ref{prop:blow4},
any polytope $\Delta'$  obtained from $\ov \De$ by a sequence
of blowups of these types will itself have symmetric $2$-faces.
Type $(F_{ij},G)$ blowups of $\De'$ are possible exactly
if we are in the situation described in part (ii) above.
\end{itemize}
\end{rmk}

Finally, we can now
construct an example in which mass linearity is destroyed by
an expansion, as promised in Remark~\ref{rmk:fibexp}.

\begin{example}\labell{ex:expanii}\rm
Let $Y \subset (\R^3)^*$ be  the  
$\Delta_2$ bundle over $\Delta_1$ associated to $a \in \R^2$
as in \eqref{eq:Yb}.
Assume that $a_1 a_2 (a_2 - a_1) \neq 0$, that is,
that none of the fiber facets are equivalent.
By Proposition \ref{DkoverD1}
$$
\Tilde H = a_2 \eta_1 - a_1 \eta_2 + (a_2-a_1) \eta_3 = 
(a_1 - 2 a_2, 2a_1 -a_2, 0) \in \R^3.
$$
is an essential mass linear function on $Y$. 

Let $\Delta \subset \ft^*$ be the $1$-fold expansion of $Y$ along 
the base facet with conormal $(0,0,-1)$,
 and let $H 
= (a_1 - 2a_2, 2 a_1 - a_2, 0 ,0)
\in \ft$ be the image of $\Tilde H$ under the natural 
inclusion.  Then $\Delta$ is a $\Delta_2$ bundle over $\Delta_2$;
the outward conormals to the base facets are
$$
 \Hat \eta_1 = (a_1,a_2,1,0),\;\; \Hat \eta_2 = (0,0,0,-1),\;\;
\mbox{and} \; \; \Hat \eta_3=  (0,0,-1,1).
$$
By  Corollary~\ref{cor:22bundle} every
mass linear function on $\De$ is inessential.
Hence, since no  fiber facet of $\De$ is equivalent to
any other facet, every mass linear function on $\De$ has the form
$H' = \sum_{i=1}^3 \ga_i\Hat\eta_i$,
where  $\sum\ga_i = 0$.
Therefore, $H$ is not mass linear.
\end{example}

\subsection{Blowups of double expansions}\labell{ss:blowex}

In this subsection, 
we find all essential mass linear functions of the type described
in case (b) of Theorem~\ref{thm:4d}
To begin, we classify mass linear functions on double 
expansions with symmetric fiber-type facets,
showing that they are all inessential.

\begin{lemma}\labell{le:dexpan0}  Let $\Delta \subset \ft^*$ be the double 
expansion of a polytope $\Tilde\Delta$ 
along facets $\Tilde F_1$ and $\Tilde F_2$.
Let $\Hat F_1$ and $\Hat F_2$ ($\Hat F_3$ and $\Hat F_4$) be the 
base-type facets associated to $\Tilde F_1$ (
respectively, $\Tilde F_2$).
Let
$H \in \ft$ 
be a
mass linear function on $\De$
with symmetric fiber-type facets.  Then $H$ is inessential, and
\begin{equation*}
H = \sum_{i=1}^4 \gamma_i \Hat\eta_i, \quad \mbox{where}  
\  
\begin{cases}
\gamma_1 + \gamma_2 + \gamma_3 + \gamma_4 = 0 & \mbox{if } \Tilde F_1 \sim \Tilde F_2, \  \mbox{and} \\
\gamma_1 + \gamma_2 = \gamma_3 + \gamma_4 = 0 & \mbox{if } \Tilde F_1 \not\sim \Tilde F_2.
\end{cases}
\end{equation*}
Conversely, any function of this form  is inessential.
\end{lemma}
\begin{proof}  
By Remark \ref{rmk:dint} (ii), 
$\Hat F_1 \sim \Hat F_2 \sim \Hat F_3 \sim \Hat F_4$  
if $\Tilde F_1 \sim \Tilde F_2$,  while  $\Hat F_1 \sim \Hat F_2 \not\sim \Hat F_3 \sim \Hat F_4$
if $\Tilde F_1 \not\sim \Tilde F_2$.
Thus, the last statement is clear.  Moreover, by
Lemma~\ref{le:ea} 
applied twice, 
there exists an inessential
function $H'$ so that at most two facets
are 
$(H - H')$-asymmetric.
Therefore $H$ is inessential by  Proposition \ref{prop:2asym}.
Since, by hypothesis, the only asymmetric facets are $\Hat F_1,\dots,\Hat F_4$,
this implies that it has the given form.
\end{proof}
 
The following proposition clarifies exactly which of
the  blowup operations allowed in Theorem \ref{thm:4d}
are needed in order for $H$ to become essential. 
 We restrict to the $4$-dimensional case, though the result can be extended to higher dimensions without too much difficulty. 
Blowups of type $(F_{ij},g)$ are defined in Definition \ref{def:edge}. 
Note that in
the $4$-dimensional case a  
symmetric $3$-face $g$
is just a symmetric facet $G$.

\begin{proposition}\labell{prop:essblow2}
Let $\Delta \subset \ft^*$ be the double expansion of a smooth
polygon $\Tilde \Delta$ 
along edges $\Tilde F$ and $\Tilde F'$, 
and let  $H \in \ft$ be  a mass linear function on $\De$
such that  
the fiber-type facets are the symmetric facets.
Let $F_1$ and $F_2$ ($F_3$ and $F_4$) be
the base-type facets associated to $\Tilde F$ (respectively, $\Tilde F'$).
Consider  a polytope $\Delta'$  that is obtained from $\Delta$ by a 
sequence of blowups, where
each blowup is either along a symmetric face or 
of type $(F_{ij}, G)$.
Then $H$ is essential on $\Delta'$ exactly if
one of the following occurs.
\begin{itemize}
\item  
$\Tilde F \not\sim \Tilde F'$ and 
one of the blowups is
of type $(F_{ij}, {G})$,
where $i \in \{1,2\}$ and $j \in \{3,4\}$.

\item  $\Tilde F \sim \Tilde F'$ and 
there exists $\{i,j,k\} \subset \{1,2,3,4\}$ such that
one of the blowups is of type $(F_{ij}, {G})$
and another is of type $(F_{ik}, {G'})$.
\end{itemize}
\NI
Moreover,
in either 
case
there exists a blowup $\Delta'$ of this type 
so that $H$ is essential 
exactly if $|\gamma_1| = |\gamma_2| = |\gamma_3| = |\gamma_4|$,
the polygon $\Tilde \De$ is not a 
triangle,
and it
contains an edge $\Tilde e$ with endpoints  $\Tilde e\cap \Tilde F$ and $\Tilde e\cap \Tilde F'$.
\end{proposition}

\begin{proof}
As before, by 
Lemma~\ref{le:symblow} and Proposition~\ref{prop:blow4}, 
$H$ is mass linear on each intermediate blowup,  
the exceptional divisors are all symmetric,
and the coefficients $\gamma_k$ remain constant under blowup.
Moreover, $H$ is inessential on $\De$ by Lemma \ref{le:dexpan0}.

Assume first 
that $\Tilde F \not\sim \Tilde F'$.
(In particular, 
$\Tilde \Delta$ cannot be a triangle.)
By Remark~\ref{rmk:dint} (ii), this implies 
that
$F_1 \sim F_2 \not\sim F_3 \sim F_4$.
If all blowups are along symmetric faces, then 
$H$ is inessential by 
Lemma \ref{le:symblow}.  
Similarly, Proposition~\ref{prop:blow4} and
Lemma~\ref{le:blowequiv} imply that
that $H$ remains
inessential  under any blowup  of type $(F_{ij}, {G})$ if $\{i,j\}$ 
is $\{1,2\}$ or $\{3,4\}$.
but is essential on $\De'$ 
after a blowup of type $(F_{ij}, {G})$
with $i \in \{1,2\}$ and $j \in \{3,4\}$.

Since 
$H$ is inessential
$\gamma_1 + \gamma_2 = \gamma_3 + \gamma_4 = 0$, 
and so
blowups of the latter type
are not allowed unless 
$|\gamma_1| = |\gamma_2| = |\gamma_3| = |\gamma_4|$. So 
assume that this equation 
holds.
Given an edge
$\Tilde e$ of
$\Tilde\De$
consider the  corresponding symmetric fiber-type facet
$G$ of $\Delta$.
By Remark~\ref{rmk:dint}  (i),
if 
$\Tilde e$ has endpoints $\Tilde e \cap \Tilde F$ and $\Tilde e \cap \Tilde F'$, then 
$F_{ij} \cap G$ intersects every base-type facet for every 
$i \in \{1,2\}$ and $j \in \{3,4\}$.
Thus, we are allowed to blow up along the edge $F_{ij} \cap G$ 
for some such $\{i,j\}$,
and this blowup makes $H$ essential.
Conversely,  
assume that
$\Tilde \Delta$ does not contain any edge $\Tilde e$ 
with endpoints $\Tilde e \cap \Tilde F$ and $\Tilde e \cap \Tilde F'$. Then  
Remark~\ref{rmk:dint} (i) implies that 
$F_{ij} \cap G$ does not  meet every asymmetric facet for any
$i \in \{1,2\}$,  $j \in \{3,4\}$, and symmetric facet $G$.
Moreover, the allowed blowups  cannot create a new  symmetric facet $G_0$ 
so that $F_{ij} \cap G_0$ 
intersects every asymmetric facet.
To see this, 
let $\Delta'$ be the blowup of $\Delta$ along a face  $f$
with exceptional divisor $G_0$. 
If $F_{ij} \cap G_0$  intersects every asymmetric facet,
then Remark~\ref{rmk:blowint} implies that 
$F_{ij} \cap f$ intersects every asymmetric facet.
Since $f$ must lie on at least one symmetric facet, this is impossible.
Therefore, 
the function
$H$ remains inessential on all allowed blowups of $\ov\De$.

Now assume that $\Tilde F \sim \Tilde F'$.
By Remark~\ref{rmk:dint} (ii), 
$F_1 \sim F_2 \sim F_3 \sim F_4$. 
Lemma~\ref{le:blowequiv}, Lemma~\ref{le:symblow}, and 
Proposition~\ref{prop:blow4} together  imply
that $H$ is essential on
$\Delta'$ 
exactly if there exists $\{i,j,k\} \subset \{1,2,3,4\}$ such that
one of the blowups is 
of type $(F_{ij}, {G})$
and another is
of type $(F_{ik}, {G'})$.

Since $\gamma_1 + \gamma_2 + \gamma_3 + \gamma_4 = 0$,
blowups of these types are not allowed unless $|\gamma_1|
=|\gamma_2| = |\gamma_3| = |\gamma_4|$; so assume that this
equation holds.
If we do perform these two blowups, then  one of them is along an
edge $F_{ij} \cap {G}$ with $i\in \{1,2\}$ and $j\in \{3,4\}$.
Just as in the previous case, this implies that $\Tilde \De$  
contains an edge $\Tilde e$ with endpoints 
$\Tilde e  \cap \Tilde F$ and $\Tilde e \cap \Tilde F'$. 
If  $\Tilde \De \neq \De_2$, then it is clear that there 
exists another edge $\Tilde e' \neq \Tilde e$
with endpoint $\Tilde e' \cap \Tilde F$, so that these two blowups are possible.
Therefore, to finish the proof 
it remains to check that if $\Tilde \De = \De_2$ suitable blowups are not possible.
Since $\De$ is then a $4$-simplex it has one symmetric facet, which 
we  call $G$.
By renumbering the $F_i$ if necessary, we may assume that the first blowup is along
$F_{12} \cap G$; let us call this blowup $\Delta''$ and the exceptional divisor $G'$.
By Remark~\ref{rmk:blowint}, $F_{12} \cap G$ and $F_{34} \cap G'$ are both  empty in $\Delta''$.
Therefore, neither $F_{ij} \cap G$ nor $F_{ij} \cap G'$ intersect
every asymmetric facet for any $i \in \{1,2\}$ and $j \in \{3,4\}$.
As we have already seen, this remains true under all further  allowed blowups.
Hence we cannot blow up in such a way to make $H$ essential.
 \end{proof}

\begin{rmk}\rm 
Let $\Delta$ and $H$ be as in Proposition~\ref{prop:essblow2} above,
and assume that 
the conditions described in its last sentence hold. 

(i) 
If $\Tilde \De$ has  four edges (the fewest possible),
then $\Tilde F$ and $\Tilde F'$ are either parallel
or equivalent (or both). If they are parallel,
then $\De$ is a $\De_3$-bundle over $\De_1$  with fiber facets $F_1, \dots,
F_4$; if we write $\De$ as in  \eqref{eq:Yb}, then $a_1 = a_2$ and $a_3 = 0$.
See, for instance,  Example \ref{ex:7}.    
In this case,
the mass linear functions that arise when we think of $\De$ as a double 
expansion are special cases of those considered in Corollary \ref{cor:Mabc}.
In contrast, if $\Tilde F$ and $\Tilde F'$ are equivalent,
the polytope $\De$ is a $\De_1$-bundle over $\De_3$ 
with base facets $F_1,\dots, F_4$.

(ii) $\Delta$ has a symmetric $2$-face (which can be blown up)
exactly if $\Tilde \De$ has more than four edges.
Blowups of type $(F_{ij},G)$ are always possible; indeed, they
are required.
\end{rmk}

The final remark in this subsection will be relevant 
to our discussion in \S\ref{ss:min}
  of the minimality of the polytope $\ov\De$; 
 cf. Remark~\ref{rmk:intmin}.

\begin{rmk}\rm
\labell{rmk:vertexblow}
Again let
$\Delta$ and $H$ be as in Proposition~\ref{prop:essblow2},  
but now assume that $\Tilde F$ and $\Tilde F'$ are adjacent edges.
Then the base-type facets intersect in a vertex $F_{1234}$
and the blowup $\De''$ of $\De$ at
$F_{1234}$ is the double expansion of  $\Tilde \De''$
along $\Tilde F \cap \Tilde \De''$ and $\Tilde F' \cap \Tilde \De''$,
where $\Tilde \De''$ itself is the blowup of $\Tilde \De$ at
the vertex
$\Tilde F \cap \Tilde F'$.
In particular, $\Tilde \De''$ is not a triangle and the exceptional divisor
meets the edges $\Tilde F \cap \Tilde \De''$ and $\Tilde F' \cap \Tilde \De''$.
Moreover, by Lemma~\ref{le:inessblow}, $H$ is an inessential function on
$\De''$. Hence, there exist a blowup $\De'$  of $\De''$
of the type described in Proposition~\ref{prop:essblow2} such that $H$
is essential on $\De'$ exactly if $|\gamma_1| = |\gamma_2| = |\gamma_3| =
|\gamma_4|$; cf. Proposition~\ref{prop:4indepblow}.
\end{rmk}

\section{$4$-dimensional polytopes}\labell{sec:4dim}

In this section, 
we establish the propositions used in \S\ref{ss:out}
to prove Theorem~\ref{thm:4d}.    It follows that the examples
constructed in the previous section, together with their blowups,
are the
only essential mass linear functions on 
smooth $4$-dimensional polytopes.
The first two subsections
analyze polytopes with three or four pervasive
asymmetric facets, while the third considers the remaining cases.

\subsection{Three asymmetric facets}\labell{ss:3perv}

This subsection analyzes
mass linear functions on 
$4$-dimensional
polytopes 
with exactly three asymmetric facets.
Our first main result, Proposition~\ref{prop:3dep}, addresses the
case that the conormals to these asymmetric facets are linearly dependent;
the case that the conormals are linearly independent is considered in 
Proposition~\ref{prop:3indep}.
Many of the results in this subsection are valid in all dimensions.
In particular, our first lemma implies that whenever 
their are exactly three asymmetric facets, each one is pervasive.

\begin{lemma}\labell{le:3comb}
Fix $H \in \ft$.
Let $\Delta \subset \ft^*$  
have
 exactly three asymmetric facets $F_1, F_2,$ and $F_3$.
Then every symmetric face 
intersects  $F_{12}$, $F_{13}$, and $F_{23}$
(and hence $F_1$, $F_2$, and $F_3$).
Moreover, every
symmetric face contains a $2$-dimensional 
triangular symmetric subface.
\end{lemma}

\begin{proof}{}  
Every symmetric face contains 
a symmetric face $g$
which is minimal in the sense that it 
does not properly contain another  symmetric face.
By Proposition~\ref{prop:symcent}, $g$ is a polytope
with exactly three facets, 
$F_1 \cap g, F_2 \cap g$, and $F_3 \cap g$.
This is only possible if $g$ is a $2$-dimensional triangle, 
and so 
it
intersects $F_{12}$, $F_{13}$, and $F_{23}$.
\end{proof}

We first assume that 
the conormals to
the three asymmetric faces are linearly dependent.

\begin{prop}\labell{prop:3dep}
Let $H \in \ft$ be a mass linear function on 
$\Delta \subset \ft^*$
with exactly three asymmetric facets $F_1$,
$F_2$, and $F_3$ with linearly dependent conormals.  
Then $\Delta$ is a $\Delta_2$ bundle over the face $F_{12}$,  
and the base facets 
are the symmetric facets.
\end{prop}

\begin{proof}{}
Since the outward conormals to the $F_i$ are linearly dependent 
the triple intersection $F_{123}$ must be empty.
Thus, Lemma~\ref{le:3comb} implies that $\De$ is combinatorially equivalent
to the product $\Delta_2 \times F_{12}$.
Hence, by Lemma~\ref{le:recog}, $\De$ is
a $\Delta_2$ bundle over the face
$F_{12}$ with fiber facets $F_1$, $F_2$, and $F_3$.  
\end{proof}

We next  
assume that the conormals
to the three asymmetric facets $F_1, F_2$, and $F_3$
are linearly independent. 
Then the three affine planes $P(F_i)$ which contain the asymmetric facets
intersect in an 
affine subspace  $\ell_{123}$ that contains the 
(possibly empty) face $F_{123}$.
Define a graph $\Gamma$ as follows: its  vertices $V$
are the   vertices in $F_{12} \ssminus F_{123}$ and its edges $E$
are the  
edges of
$F_{12}$  
that have both endpoints in $V$
and are not parallel to $\ell_{123}$.

\begin{lemma}\labell{etaindep}
Fix $H \in \ft$.
Let $\Delta \subset \ft^*$ be a smooth 
polytope with exactly three asymmetric facets $F_1$,
$F_2$, and $F_3$ with linearly independent 
conormals.

\begin{itemize}\item[(i)]
Let $Y$ be a 
symmetric $3$-face 
 of $\De$ that 
contains two 
symmetric $2$-faces,  
and assume that $Y \cap F_{12}$ is not parallel to $\ell_{123}$. 
Then $Y$ is a $\Delta_1$ bundle over $\Delta_2$, and the
symmetric facets are the fiber facets.
\item[(ii)]
If the associated graph $\Gamma$ is connected then $F_1 \sim F_2 \sim F_3$.
\end{itemize}
\end{lemma}

\begin{proof}{} 
Let $Y$ be a  
symmetric $3$-face 
that  
contains two 
symmetric $2$-faces,
 and
assume that $Y \cap F_{12}$ is not parallel to $\ell_{123}$.  
We now apply Lemma~\ref{le:3comb}. 
Since the edge $F_{12} \cap Y$  
meets both symmetric faces of $Y$, 
$F_{123}\cap Y = \emptyset$. Hence,
$Y$ is combinatorially equivalent
to $\Delta_1 \times \Delta_2,$ where the symmetric faces are triangular.
Since $Y \cap F_{12}$ is not parallel to $\ell_{123}$,
the conormals to the $F_i$
remain linearly independent when restricted to $P(Y)$.
Hence,  Lemma~\ref{prodsimp}
implies that  $Y$ is a $\De_1$ bundle over $\De_2$;
the symmetric faces are the fiber facets.
This proves (i).

Since $\Delta$ is simple,
Lemma~\ref{le:3comb} implies that intersection with $F_{12}$ 
induces a one-to-one correspondence
between the set of  
symmetric $2$-faces and 
the vertex set $V$ of $\Ga$.
It also induces
a one-to-one correspondence between 
the set of  
symmetric $3$-faces $Y$ that contain
two symmetric $2$-faces  
so that $Y \cap F_{12}$ is not parallel to $\ell_{123}$,
and the edge set $E$ of $\Ga$.  Moreover, in this case 
claim (i) 
implies
that the two symmetric  
$2$-faces are parallel.
Hence, 
two symmetric $2$-faces $X$ and $X'$ 
are parallel if the vertices $X \cap F_{12}$ and $X' \cap F_{12}$
lie in the same component of $\Ga$.

If $\Ga$ is connected, this implies that all 
symmetric $2$-faces are parallel.
By Lemma~\ref{le:3comb}, every symmetric facet must contain a 
symmetric $2$-face, so this implies that the 
conormals to all the symmetric facets
lie in a codimension  $2$ subspace.  Hence, 
by Lemma~\ref{le:equiv},
the  three asymmetric facets are equivalent.
\end{proof}

We now specialize to the $4$-dimensional case. 
The definition of a $121$-bundle may be found in Definition \ref{def:121}.

\begin{prop}\labell{prop:3indep} Let $H \in \ft$
be an essential  mass linear function on a smooth
$4$-dimensional polytope $\De \subset \ft^*$
with exactly three asymmetric facets 
$F_1,  F_2,$ and $F_3$
with linearly independent conormals.
Then there exists a smooth polytope $\ov\Delta \subset \ft^*$ so that:
\begin{itemize}
\item $H$ is an essential mass linear function on $\ov\Delta$.
\item One of the following statements is true:
\begin{itemize}
\item 
$\ov\Delta$ is a $\Delta_3$ bundle over $\Delta_1$,
and the base facets and one fiber facet of $\ov\Delta$ are 
the symmetric facets.
\item 
$\ov\Delta$  is a 
$121$-bundle
and the nonpervasive facets of $\ov\Delta$ are 
the symmetric facets.
\end{itemize}
\item 
$\Delta$ can be obtained from $\ov\Delta$ by a series of blowups along
symmetric $2$-faces.
\end{itemize}
\end{prop}

\begin{proof}
The $2$-dimensional polygon $F_{12}$ has at most two edges parallel to $\ell_{123}$.
If it has at most one such edge,
then the associated graph
$\Gamma$ is connected, and so $F_1 \sim F_2 \sim F_3$ by part (ii) of Lemma~\ref{etaindep}.
By Lemma~\ref{le:ea}, this implies that there exists an inessential function
$H' \in \ft$ so that the function $\Tilde H = H - H'$ has at most one asymmetric facet.
By Proposition~\ref{prop:2asym}, this implies that $H$ is inessential.
Therefore
$F_{12}$ has two edges parallel to 
$\ell_{123}$.
We now consider the following cases.
\MS

\NI
{\bf Case (a)}:
{\it  $F_{12}$ has exactly four edges and  $F_{123}\ne \emptyset$.} 
\smallskip

The edge $F_{123}$ is parallel to $\ell_{123}$.
Let $G_1 \cap F_{12}$, $G_2  \cap F_{12}$ and 
$G_3 \cap F_{12}$ be the remaining edges of $F_{12}$, where each $G_i$ is
a symmetric facet and $G_1 \cap F_{12}$ is parallel to $\ell_{123}$.
Then the  
conormals to  $F_1, F_2, F_3$, and $G_1$ are linearly dependent and
the  intersections 
$G_1 \cap F_3 \cap F_{12}$
and 
$ G_2 \cap G_3 \cap F_{12}$ are empty,
but the remaining  edges of $F_{12}$ do intersect.
Hence, Lemma~\ref{le:3comb} implies that $\Delta$ is combinatorially
equivalent to $\Delta_3 \times \Delta_1$, where $G_2 \cap G_3$ and $G_1 \cap
F_{123}$ are both empty.
Hence, by Lemma~\ref{le:recog},
$\Delta$ is $\De_3$ a bundle over $\De_1$;  
the fiber facets are $F_1, F_2, F_3$, and $G_1$.
\MS

\NI
{\bf Case (b)}:
{\it  $F_{12}$ has exactly four edges and  $F_{123}= \emptyset$.} 
\smallskip

Let $G_1 \cap F_{12}$,  $G_2 \cap F_{12}$,
$G_3 \cap F_{12}$, and $G_4 \cap F_{12}$
be the edges of $F_{12}$, 
where each $G_i$ is
a symmetric facet and $G_1 \cap F_{12}$ and $G_2 \cap F_{12}$
are parallel to $\ell_{123}$.
Then
the intersections $G_1\cap G_2 \cap F_{12}$ and $G_3 \cap G_4 \cap F_{12}$
are empty,
but the remaining edges of $F_{12}$ do intersect.
Lemma~\ref{le:3comb} implies that 
$\Delta$ is combinatorially equivalent to 
$\Delta_1 \times \Delta_1 \times \Delta_2$, where
$G_{12}, G_{34}$ and $F_{123}$ are empty.

Let $\eta_i$ and $\alpha_j$ denote the outward conormals to
$F_i$ and $G_j$, respectively.
Since $G_1 \cap F_{12}$ and $G_2 \cap F_{12}$ are parallel to $\ell_{123}$,
$\alpha_1$ and $\alpha_2$ both lie in the subspace spanned by
$\eta_1,\eta_2,$ and $\eta_3$.
Moreover, applying part (i) of Lemma~\ref{etaindep} to $G_3$,
we find that its two symmetric faces
$G_{13}$ 
and $G_{23}$ are parallel. 
If $G_1$ and  $G_2$ are not parallel,
this implies that $\alpha_3$ lies in the plane spanned by 
$\alpha_1$ and $\alpha_2$.
Hence, by the claim above, $\alpha_3$ lies in the subspace spanned
by $\eta_1, \eta_2,$ and $\eta_3$.   But $G_3$ is not parallel
to $\ell_{123}$, so this is impossible.
Hence, $G_1$ is parallel to $G_2$.
 
Therefore, by Lemma~\ref{le:recog},
$\Delta$ is a $\Delta_1$ bundle over the polytope $G_1$.
Moreover, since $G_1 \cap F_{12}$ is parallel to $\ell_{123}$,
$G_1$ itself is a $\Delta_2$ bundle over $\Delta_1$ 
with fiber facets $F_1 \cap G_1$, $F_2 \cap G_1$, and $F_3 \cap G_1$.
\MS

\NI
{\bf Case (c)}:
{\it  The general case.} 
\smallskip

If $F_{12}$ has four edges, the  result follows from (a) or (b).
So assume that $F_{12}$ has more than four edges.
The face $F_{12}$ is a $2$-dimensional smooth polygon with two  
edges which are parallel to $\ell_{123}$.
With the possible exception of one of these parallel edges,
every edge has 
the form $G \cap F_{12}$, where $G$ is a symmetric facet.
Therefore, 
by part (i) of Lemma~\ref{le:2blowdown}
there exist a symmetric facet $G$
so that the edge $G \cap F_{12}$ can be blown down in $F_{12}$,
is not parallel to $\ell_{123}$, and
is adjacent to $G' \cap F_{12}$ and $G'' \cap F_{12}$,
where $G'$ and $G''$ are also symmetric facets.

We claim that $\Delta$ is the blow-up of a smooth polytope $\ov\De$
along the face $\ov G' \cap \ov G''$ with exceptional divisor $G$.
To prove this, we check the three conditions of 
Proposition~\ref{prop:blowdown}. 
First, 
by part (i) of Lemma~\ref{etaindep},
 $G$ is a $\Delta_1$ bundle over $\Delta_2$ with fiber facets
$G' \cap G$ and $G'' \cap G$  and base facets $F_1 \cap G $,
$F_2 \cap G$, and $F_3 \cap G$, so condition (i) holds.
Second, the previous paragraph implies that
$F_{12}$ is the blowup of a smooth polytope 
$ F_{12}'$
along the vertex 
$P(G') \cap P(G'') \cap F_{12}'$
with exceptional divisor $G \cap F_{12}$,
and so by Lemma~\ref{le:blowtech}, condition (ii) also holds.
Finally, if $F_{123} \neq \emptyset$, then condition (iii)  
holds trivially. 
On the other hand, 
if $F_{123} = \emptyset$, then $\ell_{123}$ is a line in $P(F_{12})$
which does not intersect $F_{12}$ and is parallel to two of its edges.
Since $G \cap F_{12}$ is not parallel to $\ell_{123}$,
the polygon obtained from blowing down $G \cap F_{12}$ in $F_{12}$
will not intersect $\ell_{123}$.  Hence, $\ov{\Delta}$ also
will not intersect $\ell_{123}$, that
is, $\ov{F}_{123}$ is empty.

By Lemmas~\ref{le:blowcent} and \ref{le:symblow}, $H$ is an essential mass 
linear function on $\ov \De$
and  $\langle H, c_{\ov\De} \rangle
=\langle H, c_\Delta \rangle$. 
The result now follows by induction.
\end{proof}

\subsection{Four asymmetric facets}\labell{ss:4asym}

We now  
analyze mass linear functions on 
$4$-dimensional
polytopes with
exactly
four asymmetric facets, each of which is pervasive.
As before, we first consider the case that the conormals to these asymmetric facets
are linearly dependent, and then the case that they are linearly independent; see
Proposition~\ref{prop:4dep} and Proposition~\ref{prop:4indep}.
We
begin by considering the combinatorics.

Let $H \in \ft$ be a mass linear function on a 
 polytope $\Delta$ with
exactly four asymmetric facets $F_1, F_2, F_3,$ and $F_4$.
Proposition~\ref{prop:symcent} implies that
each 
symmetric $2$-face $g$ has exactly
four asymmetric edges: $F_1 \cap g, \ldots, F_4 \cap g$. 
Moreover $H$ is mass linear on $g$. Therefore,
Proposition~\ref{prop:2dim} implies
that 
$g$ has no symmetric edges.
Hence, for exactly two of the six pairs $\{i,j\} \subset \{1,2,3,4\}$, 
the intersection 
$F_{ij} \cap g$ is empty;
we will refer to this 
set of two pairs
as the {\bf rectangle order} of the face.

Suppose first
that $\Delta$ is
 $3$-dimensional.
Then Proposition~\ref{prop:3d} implies that $\Delta$ has at most two symmetric
facets; 
moreover, the following are true.
\begin{itemize}
\item[(i)]
If $\Delta$ has no symmetric facets, then
$\Delta$ is the simplex $\Delta_3$.
\item[(ii)]
If $\Delta$ has one symmetric facet 
then $\Delta$ is a $\Delta_2$ bundle over $\Delta_1$, and
the symmetric facet is a fiber facet. 
\item[(iii)]
If $\Delta$ has two symmetric facets 
then $\Delta$ is a $\Delta_1$ bundle over $\Delta_1 \times \Delta_1$,
and the symmetric facets are fiber facets.
\end{itemize}

Now assume that $\Delta$ is $4$-dimensional,
and let the {\bf symmetric subspace}  $\Ss \subset \Delta$ be the
union of the symmetric facets.
By Proposition~\ref{prop:symcent}, each $3$-dimensional symmetric 
facet $G$ is a polytope with four asymmetric facets
and $H$  is a mass linear function on $G$.
By the 
discussion
above,  this implies that $G$ intersects
at most two other symmetric facets.
Moreover, if $G$ does not intersect any symmetric facets, it is
the simplex
$\De_3$.
If $G$ intersects just one other symmetric facet $G'$,
then $G$ is  
a $\Delta_2$ bundle over $\Delta_1$, and $G' \cap G$ is a fiber facet.
Once the rectangle order
of $G' \cap G$ is specified,  
there are two possibilities for the combinatorics of $G$.
Assume, for example,  that
$F_1$ and $F_2$ are opposite in the rectangle order of $G' \cap G$.
If $F_{12} \cap G$ is empty, then $F_{34} \cap G$
is not, and  $F_1 \cap G$ and $F_2 \cap G$
are the triangular faces of $G$.   
Conversely, if $F_{34} \cap G$ is
empty then $F_{12} \cap G$ is not, and $F_3 \cap G$ and $F_4 \cap G$
are the triangular faces.
In contrast, 
if $G$ intersects two other symmetric facets,
then both 
symmetric $2$-faces
of $G$
have the same rectangle order, and $G$ is determined
combinatorially by this rectangle order.
This proves the following result.

\begin{lemma}\labell{le:symorder} 
Let $H \in \ft$ be a a mass linear function on a $4$-dimensional
polytope $\Delta \subset \ft^*$ 
with exactly four asymmetric
facets.  Let $C$ be a connected
component of the symmetric subspace $\Ss \subset \Delta$.
Then every 
 $2$-face in $C$ has the same rectangle order.
\end{lemma}

The analysis above also
implies that, after renumbering  the $F_i$ if necessary, 
each  component
 $C$ of $\Ss$ 
has one of the following four types.
\begin{enumerate}
\item [(a)] The component $C$ contains only one symmetric facet.
\item [(b)] The component $C$ contains two symmetric facets $G$ and $G'$
that  
each intersect only one other symmetric facet, 
and $F_{12} \cap G = F_{34} \cap G' = \emptyset$.
The remaining symmetric facets in $C$ 
each intersect two symmetric facets.
\item [(c)] The component $C$ contains two symmetric facets $G$ and $G'$
that 
each intersect only one other symmetric facet, 
and $F_{12} \cap G = F_{12} \cap G' = \emptyset$.
The remaining symmetric facets in $C$ each intersect two symmetric facets.
\item  [(d)] Every symmetric facet in $C$ intersects two other symmetric facets.
\end{enumerate}

\begin{lemma}\labell{le:4comb} 
Let $H \in \ft$ be a a mass linear function on a $4$-dimensional
polytope $\Delta \subset \ft^*$ 
with exactly  four asymmetric facets
$F_1,\dots, F_4$,
each of which is pervasive.

\begin{itemize}\item[(i)] If $F_{1234} = \emptyset$,
the symmetric subspace  $\Ss \subset \Delta$ 
has two components.  Otherwise, it has one component.

\item[(ii)]  Each of the four triple intersections $F_{ijk}$ is nonempty.

\item[(iii)] Each  
component of the symmetric subspace $\Ss \subset \Delta$ 
has type (a) or (b) above.
\end{itemize}
\end{lemma}

\begin{proof}\,
To begin, consider the
$2$-dimensional polygon $F_{ij}$  for any $i \neq j$;
let $\{i,j,k,\ell\} = \{1,2,3,4\}$.
By assumption, the asymmetric facets are pervasive, so 
$F_{ij}$  cannot be empty.
If the intersections $F_k \cap F_{ij}$ and $F_\ell \cap F_{ij}$
are nonempty, they are edges of this polygon. 
All other edges lie in  $\Ss \cap F_{ij}$.
Since every $2$-dimensional polygon  
has  at least three edges,  
the set $\Ss \cap F_{ij}$ cannot be empty.
Therefore, 
after possibly switching $k$ and $\ell$,
 there are four possibilities:
\begin{itemize}
\item[(1)]
$F_k \cap F_{ij}  = F_\ell \cap F_{ij} = \emptyset$; this implies that 
$\Ss \cap F_{ij}$
is homeomorphic to a circle.\SSS
\item[(2)]
$F_k \cap F_{ij} \neq \emptyset$ but $F_\ell \cap F_{ij} = \emptyset$;
this implies that $\Ss \cap F_{ij}$ is homeomorphic to a line
segment and both its ends are adjacent to $F_k \cap F_{ij}$.\SSS
\item[(3)]
$F_k \cap F_{ij} \neq \emptyset$ and $F_\ell
 \cap F_{ij} \neq  \emptyset$
but $F_k \cap F_\ell \cap F_{ij}
= F_{1234}
=\emptyset$; this implies that  $\Ss \cap F_{ij}$
has two components, 
each component is homeomorphic to a line segment, and each has one end
adjacent to $F_k \cap F_{ij}$ and the other to $F_\ell \cap F_{ij}$.\SSS
\item[(4)] $F_k \cap F_\ell \cap F_{ij} = F_{1234} \neq \emptyset$;
this implies that  $\Ss  \cap F_{ij}$ 
is homeomorphic to a line segment,
one end is adjacent to $F_k \cap F_{ij}$ and the other to 
$F_\ell \cap F_{ij}$. 
\end{itemize}

Now consider a component $C$ of 
$\Ss$; by Lemma~\ref{le:symorder} we may 
assume that $F_1$ and $F_2$ are opposite in the rectangle order
of every  
symmetric $2$-face in $C$.
We will evaluate each of the cases (a) through (d) listed 
before the lemma.
\begin{enumerate}
\item[(a)] In this case, $C \cap F_{ij}$ is a single edge 
(and hence 
is
homeomorphic to a line segment)
for all $i \neq j$.
One end of this edge is adjacent to $F_k \cap F_{ij}$ and the other to 
$F_\ell \cap F_{ij}$,
where $\{i,j,k,\ell\} = \{1,2,3,4\}$.
\item[(b)]  
In this case,
the  intersection   $C \cap F_{13}$
is homeomorphic to a line 
segment; 
one end
is  adjacent to  $F_4 \cap F_{13}$ and the other
to  $F_2   \cap F_{13}$.
Similarly,  
each of the intersections $C \cap F_{23}$, $C \cap F_{14}$, and $C \cap F_{24}$
is homeomorphic to a line segment,
the  ends of which are adjacent to different edges.
Finally, 
each of the intersections   $C \cap F_{12}$ and  $C \cap F_{34}$ 
is a single edge, 
the ends of which are adjacent to different edges.
\item[(c)] In this case, 
the intersection $C \cap F_{13}$
is homeomorphic to a line segment, and both ends of the segment
are adjacent to the edge $F_4 \cap F_{13}$. 
Similarly,  
each of the intersections $C \cap F_{23}$, $C \cap F_{14}$, and $C \cap F_{24}$ 
is homeomorphic to a line segment,
the ends of which are adjacent to the same edge.
However, the intersection $C \cap F_{12}$ 
is empty.
\item[(d)] In this case, 
each of  the intersections  
$C \cap F_{13}, C \cap F_{14}, C \cap F_{23},$ and 
$C \cap F_{24}$ 
is homeomorphic to a circle,
while $C \cap F_{12} = C \cap F_{34} = \emptyset$.
\end{enumerate}

We now show that the last two cases cannot occur.
Assume first that there is a component $C$  of type (c).
This cannot be the only component of $\Ss$, because then
$\Ss  \cap F_{12} =  \emptyset$, which we
showed to be impossible at the beginning of the proof.
So there exists another component $C'$ of $\Ss$.
Since $C' \cap F_{ij}$ is nonempty for at least four pairs 
$i \neq j$,
at least one polygon
$F_{ij}$ must intersect the space of symmetric facets
in at least two components, at least one of which is homeomorphic to
a line segment, the ends of which are adjacent to the the same edge.
As we saw above, this is
impossible. Hence there are no components of type (c).
Similarly, no polygon $F_{ij}$ can intersect the space of
symmetric facets in at least two components, one of 
which is homeomorphic to a circle.
Therefore, there are no
components of type (d). This proves (iii).

Since each component 
of $\Ss$ has type (a) or (b)
its intersection with each polygon $F_{ij}$ is homeomorphic to
a line segment, the ends of which are adjacent to different edges.
Therefore cases (1) and (2) for  $S\cap F_{ij}$
are impossible. Statement (ii) follows immediately.
Finally, if $F_{1234} = \emptyset$ we are in case (3),
and if it is not we are in case (4).
This proves (i).
\end{proof}

We will use the next lemma to identify facets that can be blown down.

\begin{lemma}\labell{le:blowblow}
Let $H \in \ft$ be a mass linear function
on a $4$-dimensional smooth polytope $\Delta \subset \ft^*$ 
with exactly four asymmetric facets $F_1,\ldots,F_4$,
each of which is pervasive.
Let $G$ be a symmetric facet.
Assume that the edge $G \cap F_{13}$ of the polygon $F_{13}$ can be blown down,
and that 
$F_1$ and $F_3$ are not opposite in the
rectangle order of any 
symmetric $2$-face of $G$.
Then $\Delta$ is the blowup of a smooth polytope
$\ov{\De} \subset \ft^*$ along a face $f$  with exceptional divisor $G$.
Moreover, one of the following holds:
\begin{enumerate}
\item[(i)] $G$ has no symmetric facets and  $f$ is the vertex $\ov{F}_{1234}$.
\item[(ii)] $G$ has exactly one symmetric facet $G' \cap G$,  
and $f$ is the edge $\ov{G}' \cap \ov{F}_{ij}$,  where $F_i$ and $F_j$
are opposite in the rectangle order of  $G' \cap G$.
Moreover, $f$ intersects each $\ov{F}_k$.
\item[(iii)] $G$ has two symmetric facets $G' \cap G$ and $G'' \cap G$,
and  $f = \ov{G}' \cap \ov{G}''$.
\end{enumerate}
\end{lemma}

\begin{proof}
By Proposition~\ref{prop:symcent}, $G$ itself is a $3$-dimensional
polytope with exactly four asymmetric facets  $F_1 \cap G, \ldots,
F_4 \cap G$, and the restriction of $H$ to $G$ is mass linear.
As before, Proposition~\ref{prop:3d} implies that there are only three possibilities:
\MS

\NI
{\bf Case (i)}:  
{\it  $G$ has no symmetric facets; it is a simplex with facets $F_1 \cap G, 
\ldots, F_4 \cap G$ .}
\smallskip

Condition (i) of Proposition~\ref{prop:blowdown} is clearly satisfied.
Since the edge $G \cap F_{13}$ of the polygon $F_{13}$
can be blown down, and $G \cap F_{13}$
is adjacent to $F_2 \cap F_{13}$
and $F_4 \cap F_{13}$, $F_{13}$ is the blowup of a smooth polygon
$F_{13}\,\!'$ along the vertex $P(F_2) \cap P(F_4) \cap F_{13}\,\!'$.
Hence, Lemma~\ref{le:blowtech} implies that condition (ii) of 
Proposition~\ref{prop:blowdown} is also satisfied.
Finally, 
condition (iii) of Proposition~\ref{prop:blowdown} is trivial in this case.
Hence, the claim follows by Proposition~\ref{prop:blowdown}.
\MS

\NI
{\bf Case (ii)}:  
{\it  $G$ has one symmetric facet $G' \cap G$; it is a $\Delta_2$
bundle over $\Delta_1$ and $G' \cap G$ is a fiber face.}
\smallskip

Since $F_1$ and $F_3$ are not opposite in the 
rectangle order of $G' \cap G$, one is a base facet and one is a fiber facet.
Hence, we may renumber so that
$G$ is a $\Delta_2$ bundle over $\Delta_1$ with fiber facets
$F_1 \cap G$, $F_2 \cap G$, and $G' \cap G$ 
and base facets 
$F_3 \cap G$ and $F_4 \cap G$. 
In particular,  condition (i) of Proposition~\ref{prop:blowdown} is satisfied.
Since the edge $G \cap F_{13}$ of the polygon $F_{13}$ can be blown down, 
and $G \cap F_{13}$ 
is adjacent to $F_2 \cap F_{13}$ and $G' \cap F_{13}$, 
$F_{13}$ is the blowup of a smooth polygon $F_{13}\,\!'$ 
along the 
vertex $P(F_2) \cap P(G') \cap  F_{13}\,\!'$.
Hence, Lemma~\ref{le:blowtech} implies that condition (ii) of 
Proposition~\ref{prop:blowdown} is also satisfied.
Since each asymmetric facet is pervasive,  
$F_{34}\neq \emptyset$.
Hence, condition (iii) of Proposition~\ref{prop:blowdown} is satisfied. 
Therefore,
Proposition~\ref{prop:blowdown} implies that $\Delta$ is the
blowup of a smooth polytope $\ov\Delta \subset \ft^*$
along the edge $\ov{G}' \cap \ov{F}_{12}$ with exceptional 
divisor $G$.
Finally, 
Remark~\ref{rmk:blowint}
implies
that 
$\ov{G}' \cap \ov{F}_{12}$  intersects both $\ov{F}_3$ and $\ov{F}_4$.
\MS

\NI
{\bf Case (iii)}:  
{\it  $G$ has two  symmetric facets $G' \cap G$ and $G'' \cap G$; 
it is a $\Delta_1$
bundle over $\Delta_1 \times \Delta_1$ with fiber facets
$G' \cap G$ and $G'' \cap G$ and base facets $F_1 \cap G, \ldots,
F_4 \cap G$.}
\smallskip

Clearly, 
condition (i) of Proposition~\ref{prop:blowdown} is satisfied.
Since the edge $G \cap F_{13}$ of the polygon $F_{13}$ can be blown down, 
and $G$ is adjacent to $G' \cap F_{13}$ and $G'' \cap F_{13}$,
$F_{13}$ is the blowup of a smooth polygon  $F_{13}\,\!'$ 
along the
vertex $P(G') \cap P(G'') \cap F_{13}\,\!'$.
Hence, Lemma~\ref{le:blowtech} implies that condition 
(ii) of Proposition~\ref{prop:blowdown} is satisfied.
Finally, since each asymmetric facet is pervasive, 
$F_{ij}\neq \emptyset$
for all pairs $1 \leq i < j \leq 4$.
Hence, Lemma~\ref{le:blowdown2} implies that
condition (iii) of Proposition~\ref{prop:blowdown}  is satisfied.
Hence, the claim follows by Proposition~\ref{prop:blowdown}.
\end{proof}

We can now prove our first main proposition in this subsection.

\begin{prop}\labell{prop:4dep}
Let $H \in \ft$ be a mass linear function on a smooth
$4$-dimensional polytope $\Delta \subset \ft^*$
with exactly four asymmetric facets $F_1,\ldots,F_4$
with linearly dependent  conormals.
Assume that each asymmetric facet is pervasive.
Then there exists a smooth polytope $\ov\Delta \subset \ft^*$ so that
\begin{itemize}
\item $H$ is a mass linear function on $\ov\Delta$.
\item $\ov\Delta$ is a $\Delta_3$ bundle over $\Delta_1$, 
and the base
facets of $\ov\Delta$ are 
the symmetric facets.
\item $\Delta$ can be obtained from $\ov\Delta$ by a series of blowups.
Each blowup is either along a  
symmetric $2$-face or along
an edge of 
type $(\ov{F}_{ij}, \ov{G})$ for some symmetric 
facet $\ov G$ of $\ov \De$. 
\end{itemize}
Moreover, if $H$ is inessential on $\ov \De$ then the polytope $\ov \De$
is the double expansion of a trapezoid along two parallel edges and 
the asymmetric facets are the base-type facets.
\end{prop}

\begin{proof}
Since the  conormals to the $F_i$ are linearly dependent, 
$F_{1234} = \emptyset$.
By part (i) of Lemma~\ref{le:4comb}, this implies that the  symmetric subspace 
$\Ss \subset \Delta$ has two components. 
Moreover, by Lemma~\ref{le:symorder}, the rectangle order is the
same on every 
symmetric $2$-face in each component of $\Ss$.
Hence, after possibly renumbering, $F_1$ and $F_3$ are not
opposite on the rectangle order of any 
symmetric $2$-face.

If both components of $\Ss$ contain a single symmetric facet, then 
each symmetric facet is a $3$-simplex.
Since the two symmetric facets don't intersect,
$F_{1234} = \emptyset$,  and
the conormals to $F_1,\dots,F_4$ lie in a hyperplane,
this implies that $\De$ is a $\Delta_3$ bundle over $\Delta_1$,
and the base facets are the symmetric facets.

So assume on the contrary that at least one component of $\Ss$
contains more than one symmetric facet.
Since $F_1$ and $F_3$ are not opposite in the rectangle order of any
symmetric $2$-face,
 every symmetric facet intersects $F_{13}$.  
Moreover, by
part (ii) of Lemma~\ref{le:4comb}, the triple intersections
$F_2 \cap F_{13}$ and $F_4 \cap F_{13}$ are not empty.
Since there are at least $3$ symmetric facets, this 
implies that the smooth convex polygon $F_{13}$
has  more than four edges. Also,
since the $\eta_i$ are linearly dependent, 
the two edges $F_2 \cap F_{13}$ and $F_4 \cap F_{13}$ 
are parallel.
Hence, by part (i) of Lemma~\ref{le:2blowdown},
there exists a  symmetric facet $G$ of $\Delta$ so that
the edge $G \cap F_{13}$ of $F_{13}$ can be blown 
down; moreover,
$G \cap F_{13}$ 
is not
adjacent to both 
$F_2 \cap F_{13}$ and $F_3 \cap F_{13}$.
Hence, $G$  has 
at least one 
symmetric $2$-face $g$.
Since $F_1$ and $F_3$ are not opposite in the rectangle order
of any symmetric face of $G$, we may renumber so that $F_1$ and $F_2$
are opposite on the rectangle order of 
$g$.
By Proposition~\ref{prop:symcent}, $g$ has four asymmetric edges $F_1 \cap g,
\ldots, F_4 \cap g$ and the restriction of $H$ to $g$ is mass linear.
Since $F_1 \cap F_2 \cap g = \emptyset$ and $F_3 \cap F_4 \cap g =
\emptyset$, 
Proposition~\ref{prop:2dim}
implies that $\gamma_1  + \gamma_2
= \gamma_3 + \gamma_4 = 0$, where $\gamma_i$ is the coefficient
of the support number of $F_i$ in the linear function $\langle H,
c_\Delta \rangle$.
Lemma~\ref{le:blowblow} implies that
$\Delta$ can be obtained  from
a smooth polytope $\ov\Delta$ by blowing up
along a face $f$, where $f$ is either a  
$2$-face  
of the form $\ov{G}\,\!' \cap \ov{G}\,\!''$ where $G'$ and $G''$ are symmetric facets, or  an 
edge of the form $f = \ov{F}_{12} \cap \ov{G}\,\!'$ or $f = \ov{F}_{34} \cap \ov{G}\,\!'$,
where $G'$ is a symmetric 
facet;  moreover,  in the latter case, $f$
intersects each $\ov{F}_i$.
By Lemma~\ref{le:blowcent}, $\langle H, c_{\ov\Delta} \rangle = 
\langle H, c_{\Delta} \rangle $. 
The first claim 
now follows by induction.

Finally, assume that $H$ is inessential on $\ov \De$.
Since the asymmetric facets are the four fiber facets,  
each fiber facet is equivalent to at least one other fiber facet.
It is straightforward 
to check that this implies that
$\ov \De$ is the double expansion of a trapezoid along
the two parallel sides. (Alternatively, $\ov \De$ is a double
expansion by Proposition~\ref{prop:essblow}.)
\end{proof}

To deal with the case when 
the
asymmetric facets are linearly 
independent, we
need one final technical lemma.

\begin{lemma}\labell{4comp}  
Let $H \in \ft$ be a mass linear function
on a $4$-dimensional smooth polytope 
$\Delta \subset \ft^*$  
with exactly four asymmetric facets $F_1,\ldots,F_4$
with linearly independent conormals. 
If $F_1$ and $F_2$ are opposite in the rectangle order of 
each   
symmetric $2$-face,
then $F_1 \sim F_2$.
\end{lemma}

\begin{proof}\,
First, note that by Proposition~\ref{prop:symcent}, every symmetric facet $G$ 
is itself a $3$-dimensional smooth polytope with exactly four
asymmetric faces 
$F_1 \cap G, \ldots,F_4 \cap G$, 
and   
the restriction of $H$ to $G$ 
is mass linear.

Let $\eta_i$ be the outward  conormal to $F_i$.
If $G$ has no symmetric facets,
then by Proposition~\ref{prop:3d} $G$ is a $3$-simplex with
with facets $F_1 \cap G, \ldots, F_4 \cap G$. Hence,
$\eta_1 + \eta_2 + \eta_3 + \eta_4$ is a multiple of $\alpha$,
the outward conormal to $G$.
Since the $\eta_i$ are linearly independent, this
implies that $\alpha$ 
is a multiple of $\eta_1 + \eta_2 + \eta_3 + \eta_4.$

Otherwise, the component of $\Ss$ which contains
$G$ contains more than one symmetric facet.
By part (iii) of Lemma~\ref{le:4comb}
we may label these symmetric facets $G_1,\ldots,G_k$ 
so that $G_i \cap G_{i+1} \neq \emptyset$ 
for all $i$  
but otherwise $G_i \cap G_j = \emptyset$,
and so that $F_{12} \cap G_1 = \emptyset = F_{34} \cap G_k$.
Let $\alpha_i$ denote the outward conormal to $G_i$.
The polytope  $G_1$ has one symmetric facet 
$G_2 \cap G_1$; moreover, $F_{12} \cap G_1 = \emptyset$.
Hence by Proposition~\ref{prop:3d}, $G$ is a $\Delta_2$ bundle over $\Delta_1$
with fiber facets $F_3 \cap G_1$, $F_4 \cap G_1$, and $G_2 \cap G_1$.
Hence
 $\alpha_2 + \eta_3 + \eta_4$ is
a multiple of $\alpha_1$, and so  
$\alpha_2$ lies
in the plane spanned by $\eta_3 + \eta_4$ and $\alpha_1$.
If $k > 2$,
the polytope $G_2$ has two symmetric facets $G_1 \cap G_2$
and $G_3 \cap G_2$.
Hence,  by Proposition~\ref{prop:3d}, $G$ is a $\Delta_1$ bundle
over $\Delta_1 \times \Delta_1$
with fiber facets $G_1 \cap G_2$ and $G_3 \cap G_2$. Hence,
$\alpha_1 + \alpha_3$ is
a multiple of $\alpha_2$, and so  
$\alpha_3$ lies in the plane
spanned by $\alpha_1$ and $\alpha_2$, and hence in the plane
spanned by $\eta_3 + \eta_4$ and $\alpha_1$. 
Continuing in this way, 
$\alpha_j$ lies in the plane
spanned by $\eta_3 + \eta_4$ and $\alpha_1$ for all $j$.
Since $F_{34} \cap G_k = \emptyset$,
a similar argument shows that 
that $\alpha_j $ lies in the plane spanned by $\eta_1 + \eta_2$ and $\alpha_k$ for
all $j$.
Since  
the $\alpha_j$ are not  all parallel and 
the $\eta_i$ are linearly independent,
this implies that $\alpha_j$ lies in the plane
spanned by $\eta_1 + \eta_2$ and $\eta_3 + \eta_4$ for all $j$.

In short, the conormal to every symmetric facet lies in the plane spanned by
$\eta_1 + \eta_2$ and $\eta_3 + \eta_4$.
Thus, the conormal to every facet except $F_1$ and $F_2$ lies in the hyperplane
spanned by $\eta_1 + \eta_2$, $\eta_3$, and $\eta_4$.
By Lemma~\ref{le:equiv}, this implies that $F_1$ is equivalent to $F_2$.
\end{proof}

\begin{proposition}\labell{prop:4indep}
Let $H  \in \ft$ be a  mass linear function on
a $4$-dimensional smooth polytope
$\Delta \subset \ft^*$  
with exactly four asymmetric facets $F_1,\ldots,F_4$ with linearly
independent conormals. 
Assume that every asymmetric facet is pervasive.
Then there exists a smooth polytope $\ov\Delta \subset \ft^*$ so that:
\begin{itemize}
\item $H$ is an inessential  mass linear function on $\ov\Delta$.
\item  $\ov\Delta$ is the double expansion of a smooth polygon $\Tilde \De$, 
and the fiber-type facets are the symmetric facets.
\item $\Delta$ can be obtained from $\ov\Delta$ by a series of blowups.
Each blowup is  along a 
symmetric $2$-face or 
is  of type $(F_{ij},G)$.
\end{itemize}
\end{proposition}

\begin{proof} 
Let $\eta_1,\dots,\eta_4$ be the outward conormals to $F_1,\dots,F_4$.
\\

\noindent{\bf Case (a):} 
{\it Every  
symmetric $2$-face has the same rectangle 
order} 

Assume, for example, that
$F_1$ and $F_2$ are opposite in every such face. 
Lemma~\ref{4comp} implies that
$F_1$ and $F_2$ are equivalent.
Similarly, $F_3$ and $F_4$ are equivalent.   
Since $F_1,\dots,F_4$ are pervasive, Lemma~\ref{lemma:dexpan} implies that
$\De$ is a double expansion, and $F_1,\dots,F_4$ are the base-type facets.
Hence,  by Lemma~\ref{le:dexpan0}, $H$ is inessential,
and the proposition holds with  $\ov\De=\De$.
\\

\noindent{\bf Case (b):}
{\it The general case}

By Lemma~\ref{le:symorder}, every
$2$-face in 
each component of the symmetric subspace  has the same rectangle order.
Hence, if $\Ss$ has one component, we are in Case (a).
Therefore, by Lemma~\ref{le:4comb}, we may assume that $F_{1234} = \emptyset$
and that $\Ss$ has two components; moreover, 
none of the triple intersections  $F_{ijk}$ are empty.
Consider the fan associated to the polytope $\De$. 
Since
$F_{ijk} \neq \emptyset$,
the set of non-negative linear combinations of $\eta_i, \eta_j$,
and $\eta_k$ is a convex cone in this fan 
for each triple $\{i,j,k\} \subset \{1,2,3,4\}$.
Deleting the
union 
${\mathcal B}$ of these cones divides $\ft$ into two open regions.
Since $\eta_1,\ldots,\eta_4$ are 
linearly
independent, one of these
regions is the open  cone 
$\Cc := \big\{ \sum_{i=1}^4 a_i \eta_i \; \big| \; a_i > 0\big\}$;
denote the other by $\Cc'$.
Each cone in the fan lies
entirely in the closure of one of these regions.
On the other hand, since $F_{1234}=\emptyset$,
the cone 
$\Cc$ 
itself does not lie in the fan,
and
so there must be another facet whose outward conormal lies in  $\Cc$.
Hence this boundary 
${\mathcal B}$
divides the rays of the fan corresponding to the 
symmetric facets into two nonempty sets, which  must
correspond to the two components of $\Ss$.  
Let $C$ be the component of $\Ss$ corresponding to the symmetric
facets whose conormal lies in $\Cc$; let $C'$ denote the other component.

If either component 
has only one symmetric facet $G$ then $G$ is a simplex, and so it 
has no 
symmetric $2$-faces.
Therefore, we are in Case (a).
So assume on the contrary that $\Cc$ contains more than one symmetric facet.
After renumbering, we may assume that
$F_1$ and $F_3$ are opposite in the rectangle order of the 
symmetric $2$-faces in  the component $C$, and $F_1$ and $F_2$ are opposite
in the rectangle order of the  
symmetric $2$-faces  in 
$C'$.
Since both components contain more than
one symmetric facet, by Lemma~\ref{le:4comb} each
component is of type (b).
Hence, the  edges of the  $2$-dimensional smooth polytope $F_{12}$
are (in order)  $F_3 \cap F_{12}, G_1 \cap F_{12}, \ldots, G_k
\cap F_{12}, F_4 \cap F_{12}, G' \cap F_{12}$,
where $G_1,\ldots,G_k$ are the symmetric facets in $C$,
and $G'$ is 
one of the end  symmetric facets in $C'$.
Moreover, restricting to the plane containing $F_{12}$,
the fact that the conormal to $G_i$ is contained in $\Cc$ 
for all $i$ implies that
the outward conormals to the edges $G_i\cap F_{12}$ are all positive
linear combinations of the outward conormals to $F_3 \cap F_{12}$
and $F_4 \cap F_{12}$.  Hence
the outward conormal to $G' \cap F_{12}$ cannot be.
Therefore, by part (ii) 
of Lemma~\ref{le:2blowdown},
there is at least one edge $G_i \cap F_{12}$ 
that can be blown down in $F_{12}$.

Note that $G_i$ has at least one  
symmetric $2$-face $g$. 
Since 
$F_{13} \cap g = \emptyset$ and $F_{24} \cap g = \emptyset$,
Proposition~\ref{prop:2dim} implies that $\gamma_1 + \gamma_3 = \gamma_2 + \gamma_4 = 0$.
Since $F_1$ and $F_2$ are not opposite in the rectangle order of the
symmetric faces of $G_i$, Lemma~\ref{le:blowblow} implies that $\Delta$ can be obtained
from a smooth polytope $\ov \De$ by blowing up along a face $f$, where
$f$ is either a  
symmetric $2$-face of the form $\ov G_{i-1} \cap \ov G_{i+1}$,
or an edge of the form $\ov F_{13} \cap \ov G'$ or $\ov F_{24} \cap \ov G'$, where $G'$ is
a symmetric face.
By Lemma~\ref{le:blowcent}, $\langle H, c_{\ov\Delta} \rangle = 
\langle H, c_{\Delta} \rangle $. 
The result now follows by induction.
\end{proof}

\subsection{More than four asymmetric facets}
\labell{s:more4}

To finish the proof of Theorem \ref{thm:4d}  
outlined in \S\ref{ss:out} we 
now analyze mass linear functions on $4$-dimensional polytopes with
more than four asymmetric facets, each of which is pervasive.
As we see in Corollary \ref{cor:5plus} this case does not 
occur for an essential $H$.
Additionally,
 we classify polytopes which admit essential mass linear
functions with nonpervasive asymmetric facets.

\begin{proposition}\labell{prop:5plus}
Let $H \in \ft$
be a mass linear function
on a $4$-dimensional smooth polytope  
$\Delta \subset \ft^*$ with more than four asymmetric facets.
Assume
that every asymmetric facet is pervasive. 
Then one of the following statements is true:
\begin{enumerate}
\item  $\Delta$ is the four-simplex $\Delta_4$, or
\item $\Delta$ is a  $\Delta_2$ bundle over $\Delta_2$.
\end{enumerate}
\end{proposition}

\begin{proof}
Assume first that every facet is asymmetric.
By [I, Corollary A.8]
this implies that
$\Delta$ is combinatorially equivalent to the product of simplices.
Since every facet is pervasive, $\Delta$ is 
either  a $4$-simplex or is combinatorially equivalent to 
$\De_2 \times \De_2$.
In the second case, by Lemma~\ref{prodsimp}, $\Delta$
is a $\Delta_2$ bundle over $\Delta_2$.

Therefore, 
we may assume that $\Delta$ has at least one symmetric facet.
Label the asymmetric facets 
$F_1,\ldots,F_k$.  
Proposition~\ref{prop:symcent} implies that
every symmetric face has $k$ asymmetric facets
and that the restriction of $H$ to this face is mass linear.
Since 
Proposition~\ref{prop:2dim} 
implies 
that 
a $2$-dimensional polygon with a mass linear function has at most four 
asymmetric
edges, 
there are no 
symmetric $2$-faces.
Hence,  no symmetric facets of $\De$ intersect, and
each symmetric facet has no symmetric faces.
Thus, Proposition~\ref{prop:3d}
implies that there
are only two possibilities
with $k > 4$.
\MS

\NI
{\bf Case (a)}:  {\it 
 $\Delta$ has five asymmetric facets, and each symmetric facet
is a $\Delta_2$ bundle over $\Delta_1$.}\smallskip

Let $S$ 
denote the set of symmetric facets.
Since every symmetric facet
is  a  $\Delta_2$ bundle over $\Delta_1$,
there are $5|S|$ 
$2$-dimensional faces, 
$9|S|$ edges, and $6|S|$ vertices
that do lie on a symmetric facet.
Since the
five
asymmetric facets are pervasive
there are ten $2$-dimensional 
faces that 
do not
lie on any symmetric facet.
Let $E$ and $V$ be the sets
of edges and vertices, respectively,
that  
do not lie on any symmetric facet.
Since  
the
Euler characteristic of $\Delta$ is $0$,
$$
 5 + |S| - 10 - 5|S| + |E| + 9|S| - |V| - 6|S| = 0,
 $$
and hence $|E| = 5 +  |S| + |V|$.

Each vertex in $V$ lies on four edges in $E$, and
each vertex that  
lies on a symmetric facet lies on
exactly one edge in $E$.
Since exactly two vertices lie on each edge,
$2|E| = 4|V| + 6|S|$.
Combined, these equations yield 
$|V| + 2|S| = 5$. 
Since by assumption $S \neq \emptyset$, this
implies that $|S| =1$ or $2$.

If $|S| = 1$, then $\De$ has $6$ facets.
It is well known that any $n$-dimensional polytope with $n+2$ facets,
such as $\Delta$, is a product of
two 
simplices;
for a proof in the current setting see ~\cite[Prop~1.1.1]{Tim}.
Since $\Delta$ has at most one facet that is not pervasive, this means that
$\De$ is combinatorially equivalent to $\De_2\times \De_2$.
By Lemma~\ref{prodsimp}, this implies that
$\De$ is a $\De_2$ bundle over $\De_2$.

So assume 
instead
that there are two symmetric facets, $G$ and $G'$.
Then $|V| =1 $ and $|E| = 8$.
Since no edge in $E$ can connect two vertices in the same symmetric facet, 
and $G$ and $G'$ each  
have $6$ vertices,
there must be four edges that join $G$ to $G'$
and two edges that join each to the vertex in $V$.
By renumbering, we may assume that
$F_{1234} \neq \emptyset$, that
the edges $F_{123}$ and $F_{124}$ intersect $G$ but not $G'$, 
and that $F_{234}$ and $F_{134}$ intersect $G'$ but not $G$.
Since both $G$ 
and $G'$ are $\Delta_2$ bundles over $\Delta_1$,
this is only possible if  the fiber  facets of $G$
are  $F_1 \cap G$, $F_2\cap G$, and $F_5 \cap G$, 
and the base facets $F_3 \cap G$ and $F_4 \cap G$.
Similarly, the fiber facets of $G'$ are
$F_3 \cap G'$, $F_4 \cap G'$, and $F_5 \cap G'$,
and the base facets are $F_1 \cap G'$ and $F_2 \cap G'$.
This
implies that the remaining four edges
in $E$ are $F_{135}, F_{145}, F_{235}$, and $F_{245}$.

Now let $\eta_i$ denote the outward conormal to the
facet $F_i$, and $\alpha$ and $\alpha'$ denote the
outward conormals to $G$ and $G'$, respectively.
Since $F_{1234} \neq \emptyset $  
and $\De$ is smooth,
there is a change of basis so that
$\eta_1 = (1,0,0,0)$, $\eta_2 = (0,1,0,0)$, $\eta_3 = (0,0,1,0)$,
and $\eta_4 = (0,0,0,1)$.  
Since the fiber facets of $G$  
are  $F_1 \cap G$, $F_2\cap G$, and $F_5 \cap G$,
we must have 
$\eta_1 + \eta_2 +  \eta_5 = A \alpha$
for some integer $A$. 
Similarly, $\eta_3 + \eta_4 + \eta_5 = A'\alpha'$,
for some integer $A'$.  
Since $\eta_1 + \eta_2  \neq \eta_3 + \eta_4$,  
we may assume without loss of  generality that 
$A' \neq 0$.
The polygon $F_{12}$ has only three edges: $F_3 \cap F_{12}$, $F_4 \cap F_{12}$,
and $G \cap F_{12}$.
Hence, it is the standard $2$-simplex, that is,  $\alpha = (x,y, -1,-1)$
for some integers $x$ and $y$.
Therefore, $\eta_5 = (A x - 1, A y - 1, -A, -A)$.
Since $A' \neq 0$, 
$$
 \alpha' = \frac{1}{A'}(A x - 1, A y -1, 1-A,1-A).
$$
Thus, the facets $F_3$ and $F_4$ are equivalent.
By Lemma~\ref{le:ea},  
this implies that
there exists an inessential function $H'$ so that 
the mass linear function $\Tilde{H} := H - H'$ has the following property:
$F_3$ is $\Tilde{H}$-symmetric, but $F_1, F_2$, and $F_5$ are 
$\Tilde{H}$-asymmetric.
Since $F_3$ is pervasive by hypothesis, and $\De$ has seven facets,
$F_3$ has six facets.  
Moreover, by Proposition~\ref{prop:symcent}, $\Tilde{H}$ is
mass linear on $F_3$ with at least three asymmetric facets.
Thus we may apply Proposition \ref{prop:3d} to $F_3$.
Both the  
polytopes on this list with six facets are combinatorially equivalent 
to $\De_1\times \De_1\times \De_1$.  
But we saw above that $F_3\cap G$ is a base facet of $G$ and so is a triangle.  This is impossible. Hence this case also does not occur.

\MS

\NI
{\bf Case (b)}:  {\it 
 $\Delta$ has six asymmetric facets, and each symmetric facet is
$\Delta_1 \times \Delta_1 \times \Delta_1$.} \smallskip

Let $G$ be a symmetric facet.
Consider the slices 
$Q^\la$ 
through $\Delta$ parallel to $G$.
More precisely, consider how the the parallel
planes $P(F_i)\cap Q^\la$ and $P(F_j)\cap Q^\la$ 
associated to opposite faces of
the box come together as we move the slice through $\Delta$.
If one (or two) of the pairs of planes come together before
the other pairs (or pair),
then the remaining facets will not intersect.
This contradicts the claim that they are pervasive.
So assume that the  three pairs of planes come together at the
same time. 
Since $\Delta$ is simple
this point cannot lie in the polytope,
that is, it must be cut off by some symmetric facet.
But then none of the opposite pairs intersect,
which again contradicts the claim that they are pervasive.
Thus, this case does not occur.\end{proof}

This has a number of corollaries.

\begin{cor}\labell{cor:5plus}
In the situation of Proposition \ref{prop:5plus} every mass linear function on $\De$ is inessential.
\end{cor}
\begin{proof}
By Corollary~\ref{cor:22bundle},
all mass linear functions on $\De_2$ bundles over $\De_2$ are inessential.
Similarly, every $H \in \ft$ is inessential on $\Delta_4$.
\end{proof}

\begin{cor}\labell{cor:asym1}
Let $H \in \ft$ be a mass linear function on a $4$-dimensional 
smooth polytope $\De \subset \ft^*$.
There is an inessential function $H' \in \ft$ such that $H-H'$ has 
at most $4$ asymmetric facets.
\end{cor}

\begin{proof}
By Proposition~\ref{prop:flat} (ii), we may assume without loss of generality
that every asymmetric facet is pervasive.
Hence, Proposition~\ref{prop:5plus} implies that if $\Delta$ has more than four asymmetric facets then
either $\Delta$ is the four-simplex $\Delta_4$, or it is a $\Delta_2$ bundle over $\Delta_2$.
In the first case, every $H \in \ft$ is inessential; in the second case, the result follows from
Proposition~\ref{prop:bund}.
(Alternately, $H$ is inessential by Corollary~\ref{cor:5plus}.)
\end{proof}

We now consider the final case.

\begin{proposition}\labell{prop:nonperv}
Let $H \in \ft$ be 
a
mass linear function on a smooth $4$-dimensional 
polytope $\Delta \subset \ft^*$.
Assume that at least one asymmetric facet is not pervasive.
If $\Delta$ has at most four asymmetric facets, $H$ is inessential.
If $\Delta$ has more than four asymmetric facets, then
one of the following statements is true:
\begin{itemize}
\item  $\De$ is a $\Delta_3$ bundle over $\Delta_1$; 
\item $\De$ is a $121$-bundle;
\item $\De$ is a $\Delta_2$ bundle over a 
polygon which is a $\Delta_1$ bundle over $\Delta_1$;  or
\item $\De$ is a $\De_1$ bundle over the product $(\De_1)^3$ and $H$
is inessential; moreover, either the base facets are the asymmetric
facets or every  facet is asymmetric and $\De = (\De_1)^4$.
\end{itemize}
\end{proposition}

\begin{proof}
By the first part of Proposition \ref{prop:flat},  
$\De$ is a bundle over $\De_1$.
Therefore, by Proposition~\ref{prop:bund}, 
we can write $H = H' + \Tilde H$, where $H'$ is an inessential
function, the fiber facets are $H'$-symmetric, and the two base facets
are $\Tilde H$-symmetric.

If at 
most
four facets are $H$-asymmetric, then at most two facets
are $\Tilde H$-asymmetric.  Hence, $\Tilde H$ (and thus $H$)
is inessential by Proposition~\ref{prop:2asym}.

On the other hand, if at least five facets are $H$-asymmetric, then
at least three facets are $\Tilde H$-asymmetric.
Therefore, by Proposition~\ref{prop:symcent}, $\Tilde H$ is mass linear on
$F$  with  at least three asymmetric facets.
Hence, we may apply Proposition~\ref{prop:3d} to $F$.
If $F = \De_3$,   $F$ 
is a $\Delta_1$ bundle over $\Delta_2$, or 
$F$ is a $\Delta_2$ bundle over $\Delta_1$ then we are clearly in  one of
the first three cases listed above.
So suppose that $F$ is a $\De_1$ bundle over $\De_1 \times \De_1$ and
the base facets of $F$ are the $\Tilde H$-asymmetric facets. 
Repeating the argument above
for each asymmetric facet, we see that $\De$ is a $\De_1$ bundle
over $(\De_1)^3$ and the base facets of $\De$ are the $H$-asymmetric facets.
Simillary, if $F$ is $(\De_1)^3$ and every facet of $F$ is $\Tilde H$-asymmetric, then $\De = (\De_1)^4$ and every facet of $\De$ is $H$-asymmetric by a similar
argument.
In either case, 
Proposition \ref{prop:flat} (ii) implies that $\De$ supports no 
essential mass linear functions.
\end{proof}

\begin{cor}\labell{cor:asym2}
Let $H \in \ft$ be a mass linear function on a $4$-dimensional smooth polytope $\De \subset \ft^*$.
The polytope $\Delta$ has at most eight asymmetric facets.
If it has eight then
$\Delta$ is the hypercube 
$\Delta_1 \times \Delta_1 \times \Delta_1 \times \Delta_1$,
and $H$ is inessential.
Moreover, if it has exactly seven asymmetric facets
then $\Delta$ is the product $\De_1 \times Y$, where
$Y$ is a $\De_2$ bundle over $\De_1$.
\end{cor}

\begin{proof}
Assume that $\De$ has at least seven asymmetric facets.
Proposition~\ref{prop:5plus} shows that there exists an
asymmetric facet which is not pervasive.
Hence 
Proposition~\ref{prop:nonperv} implies that $\De$ is
a $121$-bundle, a $\De_2$ bundle over a polygon, or $(\De_1)^4$.
Therefore, the claim follows immediately from
Corollaries~\ref{cor:db} and \ref{cor:7}. 
\end{proof}

\begin{cor}\labell{cor:asym3}
Let $H \in \ft$ be a mass linear function on a $4$-dimensional smooth 
polytope $\De \subset \ft^*$ so that every facet is asymmetric.
Then one of the following is true:
\begin{itemize}
\item $\Delta$ is the four-simplex $\De_4$,
\item $\Delta$ is a $\Delta_3$ bundle over $\Delta_1$,
\item $\Delta$ is a $\Delta_2$ bundle over $\Delta_2$
\item $\Delta$ is the product $\Delta_1 \times Y$, where
$Y$ is a $\Delta_2$ bundle over $\Delta_1$, or
\item $\Delta$ is the product $(\Delta_1)^4$.
\end{itemize}
\end{cor}

\begin{proof}
By assumption, $\De$ must have at least five asymmetric facets.
If every asymmetric facet is pervasive, the claim follows immediately
from Proposition~\ref{prop:5plus}.
On the other hand, if there exists a asymmetric facet that is
not pervasive, then  Proposition~\ref{prop:nonperv} implies that $\De$ is
a $\De_3$ bundle over $\De_1$, a $121$-bundle, a $\De_2$ bundle over
a polygon, or the product $(\De_1)^4$.
Therefore, the claim follows immediately from
Corollaries~\ref{cor:db} and \ref{cor:7}. 
\end{proof}

\section{Further results}\labell{s:furth}

This section contains several 
results that 
are not needed for the proof of the main theorem.
  \S\ref{ss:min} addresses 
the question of which of the polytopes $\ov\De$ in Theorem 
\ref{thm:4d} are minimal. Then in \S\ref{ss:fullML} we use Theorem
\ref{thm:4d} to show that in dimensions $\le 4$ every mass linear function
 is fully mass linear.  Finally, in  \S\ref{ss:blml} we 
 discuss the question of which blowup operations preserve mass linearity, 
considering both vertex and edge blowups. 
Additionally,  we
show that a vertex blowup never converts an inessential function to an essential one.
On the other hand, edge blowups of type $(F_{ij},g)$  may do this, but only if the underlying polytope is a double expansion.

\subsection{Minimality}\labell{ss:min}

We now consider which of the polytopes $\ov\De$ in Theorem 
\ref{thm:4d} are minimal.
In particular, we show that in most cases 
we can blow down $\Delta$
in the two allowed ways to obtain a minimal polytope $\ov \Delta$;
the exceptions  occur in cases (a2) and (b).

We begin with a useful technical lemma.

\begin{lemma}\labell{le:noblow}
Let $F$ and $F'$ be (distinct) equivalent facets of a polytope $\Delta$.
Then neither $F$ 
nor $F'$ can be blown down.
\end{lemma}

\begin{proof}
Let $\eta$ and $\eta'$ be the outward conormals to $F$ and $F'$, respectively.
By Lemma~\ref{le:equiv}, 
there  exists a vector $\xi \in \ft^*$ that is 
parallel to all the other facets.   Since $\Delta$ is compact,
$\langle \eta, \xi \rangle$ and $\langle \eta', \xi \rangle$ have opposite signs.
Therefore, $\eta$ 
cannot be written as the positive sum of  
outward conormals without including 
$\eta$ itself.
\end{proof}

As we show in the next proposition, 
polytopes of type (a1) 
which admit an essential mass linear function
are all minimal.

 \begin{prop}\labell{prop:a1blow} 
Let $\De$ be a
$\Delta_3$ bundle over $\Delta_1$ that admits an essential mass linear function.  Then $\De$  is minimal.
\end{prop}
\begin{proof}
Let
  $\Delta$ be the $\Delta_3$ bundle over $\Delta_1$ associated to 
$a \in \R^3$
as in \eqref{eq:Yb}.  
By Lemma~\ref{le:a1blow} below, 
some facet of  $\Delta$ can be
blown down  exactly if $\sum_{j=1}^3 a_j e_j$ is the conormal to one of the fiber facets,
that is,
exactly if
$a$ is  $(-1,0,0), (0,-1,0), (0,0,-1)$, or $(1,1,1)$.
By Corollary~\ref{cor:Mabc}, none of these bundles 
admit essential mass linear functions. 
\end{proof}

\begin{lemma}\labell{le:a1blow}
Let $\Delta$ be a $\Delta_3$ bundle over $\Delta_1$.
The base facets cannot be blown down, and a 
fiber facet $F$ with outward conormal $\eta$
can be blown down exactly if $\eta = \alpha_1 + \alpha_2$,
where $\alpha_1$ and $\alpha_2$ are the outward conormals to the base facets. 
In that case $\De$ 
is the blowup of 
 a $4$-simplex.
\end{lemma}

\begin{proof}
Since the base facets are equivalent, the first claim follows 
from Lemma~\ref{le:noblow}.
 
Let $G_1$ and $G_2$ be the base facets, and let
$F_1,\dots,F_4$ be the fiber facets with outward conormals $\eta_1,\dots,\eta_4$.
Note that, for example,  $F_1$ is a $\Delta_2$ bundle over $\Delta_1$ with 
fiber facets $F_{12}, F_{13}$, and $F_{14}$  and with
base facets $G_1 \cap F_1$ and $G_2 \cap F_1$.
Therefore, by Proposition~\ref{prop:blowdown},
$F_1$ cannot be blown down unless either $\eta_1 = \eta_2 + \eta_3 + \eta_4$
or $\eta_1 = \alpha_1 + \alpha_2$.

 Since the first equation does not hold,
let us assume that $\eta_1 = \alpha_1 + \alpha_2$.
Then 
the two facets  $G_1 \cap F_1$ and $G_2 \cap F_1$ of $F_1$ are parallel, and so, 
we can also view $F_1$ as a (trivial) $\Delta_1$ bundle over
$\Delta_2$ with {\em fiber} facets $G_1 \cap F_1$ and $G_2 \cap F_1$. 
Thus,
conditions (i) and (ii) of Proposition~\ref{prop:blowdown} are both satisfied.
Finally, since $F_{234} \subset \Delta$ is  nonempty, condition (iii) is vacuous.  
Therefore the claims in the first sentence hold.
The last statement holds because $\De_4$ is the only $4$-dimensional polytope with $5$ facets.
\end{proof}

In 
contrast,
as we show in the next lemma,
there exist polytopes of type (a2) that admit essential mass linear functions
but 
are not minimal. 
Note that,  in most cases, the
blowup described below is not one of the two types allowed
in the main theorem.

\begin{lemma}\labell{le:a2blow0}
Let $\Delta \subset \ft^*$ be a $\Delta_3$ bundle over $\Delta_1$ 
with fiber facets $F_1,\dots,F_4$.  
Let $H \in \ft$ be a mass linear function on $\Delta$ 
such
that $F_1$ is symmetric.
The blowup  $\Delta'$  of $\Delta$ along the edge
$F_{234}$  
is a $121$-bundle; 
(see Definition~\ref{def:121}).
Moreover,   
$H$ is 
a mass linear function on 
$\Delta'$.  Finally,
$H$ is inessential on $\Delta'$ exactly if it is inessential on $\Delta$.
\end{lemma}

\begin{proof}  
The first claim is easy. 
To prove the second, 
decompose
$\Delta$ as $\Delta' \cup W$, where $W$ is also a $\Delta_3$ bundle over $\Delta_1$.
Because the exceptional divisor
is parallel to 
$F_1$,
the polytope $W$
is analogous to $\De$.    Since  
$F_1$ is symmetric, 
this implies that  
$\langle H, c_{\De} \rangle = \langle H, c_{W} \rangle.$
Thus $H$ is mass linear on $c_{\De'}$ by Lemma \ref{le:newblow}.

Finally, by  
Lemma~\ref{le:bund}(iii), 
the base facets of $\Delta$
are not equivalent to any fiber facet.
Hence,  Lemma~\ref{le:blowequiv}  implies that two asymmetric facets
are equivalent in $\Delta$ exactly if the corresponding facets of $\Delta'$ are equivalent,
and the exceptional divisor $F_0'$ is not equivalent to any other  facet.
By the definition of inessential and Proposition~\ref{prop:inessential}, this proves the last claim.
\end{proof}

We next show that all other polytopes $\De'$ of type (a2) 
that admit essential mass linear functions
are minimal.
We write $\De'$ as in \eqref{eq:de'} and
denote by $F_i\,\!'$ 
the facet with outward conormal $\eta'_i$, and 
by $\Tilde F_j$ 
the
facet with outward conormal $\Tilde \eta_j$.

\begin{prop}\labell{prop:a2blow} 
Let  $H$ be an essential mass linear function
on a $121$-bundle $\De'$
that is the blowup of another polytope $\Delta$.
Then $H$ is an 
essential
mass linear function on $\De$ and
$\Delta$ is a $\Delta_3$ bundle over $\Delta_1.$ Further,
exactly three of the fiber facets 
of $\De$
are asymmetric,
and the blowup is along the intersection of those three facets.
\end{prop}

\begin{proof} 
Write $\De'$ as in Equation~\eqref{eq:de'}.
By Proposition \ref{prop:flat} we may subtract an inessential function 
from $H$ to get an essential function $\Tilde H$ such that the four 
nonpervasive
facets 
are $\Tilde H$-symmetric, but each pervasive 
facet is
$\Tilde H$-symmetric exactly
if it is $H$-symmetric.
Since the fiber facets $\Tilde F_0$ and $\Tilde F_1$
are $\Tilde H$-symmetric, Proposition~\ref{prop:lift} implies that $\Tilde H$
is the lift of an essential mass linear function on the  base of $\Delta'$.
Hence, since the base of $\Delta'$  is the
$\De_2$ bundle over $\De_1$ associated to $(a_2,a_3)$, 
Proposition~\ref{DkoverD1} 
implies that $a_2 a_3 (a_2 - a_3) \neq 0$ and the three non-pervasive
facets $F_2\,\!', F_3\,\!',$ and $ F_4\,\!'$ are $\Tilde H$-asymmetric,
and hence $H$-asymmetric.  
Therefore by Lemma \ref{le:a2blow}
below, 
$d = 1$,
$\De$ is a $\De_3$ bundle over $\De_1$, 
and the blowup is along 
the 
intersection of the three
fiber facets of $\De$ corresponding to 
$F_2\,\!',F_3\,\!',$ and $F_4\,\!'$.  
By Lemma \ref{le:blowcent}(ii), $H$ is mass linear on $\De$, 
and all the claims in the last sentence hold.  
Finally, $H$ is essential on $\De$ by Lemma~\ref{le:a2blow0}.
\end{proof}

\begin{lemma} \labell{le:a2blow} Suppose that the 
$121$-bundle 
$\De'$ of Equation \eqref{eq:de'} is the blowup of a polytope $\De$.
 If $a_2 a_3 (a_2 - a_3) \neq 0$, then
$d=1$ and $\Delta$ is the $\Delta_3$ bundle over $\Delta_1$ associated to $(a_1,a_2,a_3)$.  Moreover
 the blowup is along the intersection of 
the three fiber facets of $\De$ corresponding to the facets 
$F_2\,\!',F_3\,\!',$ and $F_4\,\!'$ of $\De'$.
\end{lemma}

\begin{proof}
Since $F_5\,\!' \sim F_6\,\!'$, neither facet can be blown down by  Lemma~\ref{le:noblow}.
Next, fix $i \in \{2,3,4\}$, and observe that
$F_i\,\!'$ is a $\Delta_1$ bundle over a 
trapezoid.
The three non-intersecting pairs of facets
of $F_i\,\!'$
are $\Tilde F_0 \cap F_i\,\!'$ and $\Tilde F_1 \cap F_i\,\!'$, 
$F_{ij}\,\!'$ and $F_{ik}\,\!'$ where $\{i,j,k\}=\{2,3,4\}$, and 
$F_{i5}\,\!'$ and $F_{i6}\,\!'$.
It is easy to check that  
$\eta_i' \neq \Tilde \eta_0 + \Tilde \eta_1$ and
$\eta_i' \neq \eta_j' + \eta'_k$.  Further, because 
$a_2 a_3 (a_2 - a_3) \neq 0$, we also have $\eta'_i \neq \eta'_5 + \eta'_6$. 
Hence, Proposition~\ref{prop:blowdown} implies that $F_i\,\!'$ cannot be blown down.

So fix $j \in \{0,1\}$, and assume that $\Tilde F_j$ can be blown down.
Note that $\Tilde F_j$ is a $\Delta_2$ bundle over $\Delta_1$
with fiber facets $F_i\,\!' \cap \Tilde F_j, i=2,3,4$
and base facets $F_5\,\!' \cap \Tilde F_j$ and $F_6\,\!' \cap \Tilde F_j$.
Moreover, since $a_2 a_3 (a_2 - a_3) \neq  0$,
$\Tilde \eta_j \neq \eta'_5 + \eta'_6$.
Therefore, Proposition~\ref{prop:blowdown} implies that 
$\Delta'$ is the blowup of $\Delta$ along the edge $F_{234}$,
which (is not empty and) meets $F_5$ and $F_6$.
(Here, $F_i$ is the facet of $\Delta$ such that 
$F_i\,\!' = F_i \cap \Delta'$.)
In particular, $\Tilde \eta_j = \eta'_2 + \eta'_3 + \eta'_4$.
Since $d \geq 0$ by assumption, this implies that $j = 0$ and 
$d = 1$.
Therefore, $\Delta$ is the $\Delta_3 $ bundle over $\Delta_1$ associated to $(a_1,a_2,a_3)$, and the blowup is along three of its fiber facets.
\end{proof}

\begin{rmk}\labell{rmk:a2blow}\rm  
Lemma \ref{le:a2blow} shows that 
if $a_2 a_2 (a_2 - a_3) \neq 0$, then the $121$-bundle $\De'$ of
Equation~\eqref{eq:de'}
is minimal unless $d=1$.  
In other words
$\Delta'$ cannot be blown down unless
the sum of the outward conormals to  
its three pervasive facets
is the outward conormal to a fiber facet. 
However, this 
condition
is
not sufficient because condition (iii) in Proposition 
\ref{prop:blowdown} may fail for certain values of $\ka$. 
 Here the base facets of 
$\Tilde F_0$ are given by its intersection with $F_5'$ and $F_6'$, and
these facets may intersect when we remove $\Tilde F_0$;
Figure \ref{fig:2} illustrates a similar $3$-dimensional situation in which $\De_2$ is replaced by $\De_1$.
\end{rmk}

We next consider polytopes of type (a3),
that is, $\De_2$ bundles over polygons.
As we show below, every polytope of type (a3) that admits
an  essential mass linear function  can be obtained from
a minimal polytope of type (a3) that admits an essential mass linear
function by a series of blowups along  symmetric $2$-faces.
However, these minimal polytopes may have arbitrarily many facets.

\begin{lemma}\labell{le:3depblow} Let $\Delta$ be  a $\De_2$ bundle over 
a polygon $\Hat\De$.
Let 
$G_1, 
G_2, \dots, G_N$ be the base facets  of $\De$, and let
$\alpha_j$ be the outward conormal to $G_j$ for all $j$.
Assume that the edges of $\Hat \De$ corresponding to $G_j$ and $G_{j+1}$ are
adjacent for all $j$. 
(We interpret the $G_j$ in cyclic order.)
Then $G_i$ can be blown down exactly if $\alpha_i = \alpha_{i-1} + \alpha_{i+1}$.
In this case, $\Delta$ is the blowup of a polytope $\ov \Delta$ along
the face $P(G_{i+1}) \cap P(G_{i-1}) \cap \ov \De$.
\end{lemma} 

\begin{proof}   
Let $F_1$, $F_2$, and $F_3$ be the fiber facets of $\De$.
The facet $G_i$ is a $\Delta_2$ bundle over $\Delta_1$
with fiber facets 
$F_1 \cap G_i$, $F_2 \cap G_i$, and $F_3 \cap G_i$, 
and with base facets $G_{i-1} \cap G_i$ and $G_{i+1} \cap G_i$.

Assume first that $\alpha_i \neq \alpha_{i+1} + \alpha_{i-1}$.
Since also $\alpha_i \neq 0 = \eta_1 + \eta_2 + \eta_3$,
Proposition~\ref{prop:blowdown} implies that $G_i$ cannot be blown down.

So assume instead that
$\alpha_i = \alpha_{i+1} + \alpha_{i-1}$.
In this case, 
we can also consider $G$ as a (trivial) $\Delta_1$ bundle
over $\Delta_2$, 
and condition (ii)
in
Proposition~\ref{prop:blowdown} is clearly satisfied.
Moreover,
since $P(F_1) \cap P(F_2) \cap P(F_3) = \emptyset$  and since
$F_K \neq  \emptyset$ for any 
$K \varsubsetneq \{1,2,3\}$ 
condition (iii) is also satisfied.
Hence the claim follows from  Proposition~\ref{prop:blowdown}.
\end{proof}

\begin{rmk}\rm  
Note that the edge of $\Hat \De$ associated to $G_i$ can be blown down exactly if
$\Hat\alpha_i = \Hat\alpha_{i-1} + \Hat\alpha_{i+1}$, that is
the vector $\alpha_i -\alpha_{i-1} - \alpha_{i+1}$ lies in the span of the fiber conormals.
The condition
$\alpha_i = \alpha_{i-1} + \alpha_{i+1}$ given above is stronger;
it also implies 
that the bundle $\De\to \Hat\De$ is trivial when restricted to $G_i$.
\end{rmk}

\begin{lemma}\labell{le:a3fiber}
Let $H \in \ft$ be an essential mass linear function on a polytope $\Delta \subset \ft^*$
that is a $\Delta_2$ bundle over a polygon.  Then no fiber facets can be blown down.
\end{lemma}
 
\begin{proof}
Since $H$ is essential,
Proposition~\ref{prop:polybundle} implies that 
the
fiber facets are all asymmetric.
Hence, the claim follows from
part (i) of Lemma~\ref{le:blowcent}.
\end{proof}

\begin{prop}\labell{prop:3depblow} 
Let $H \in \ft$ be an essential  mass linear function on a polytope $\De' \subset \ft^*$ that 
is a $\Delta_2$ bundle over a polygon $\Hat \De'$.
Then there exists a {\em minimal} polytope $\De$ 
so that $H$ is essential on $\De$ and
$\De'$ can be obtained from $\De$ by a series of blowups along 
symmetric $2$-faces;
moreover, 
$\De$ is a $\Delta_2$ bundle over a polygon $\Hat \De$.
\end{prop}

\begin{proof}
Let $G_1,\dots,G_k$ be the base facets of $\De'$.
Assume that they are labelled so that the 
edges of $\Hat \De'$ corresponding to $G_i$ and $G_{i+1}$ are adjacent for all  
$i$, where $G_{k+1}=G_1$.
Let $\alpha_i$ be the outward conormal to $G_i$ for all $i$.

First assume that 
$\alpha_i \neq \alpha_{i-1} + \alpha_{i+1}$ for all $i$.
By Lemma~\ref{le:3depblow}, this implies that none of the base facets can be blown down.
By Lemma~\ref{le:a3fiber}, the fiber facets cannot be blown down either.
Therefore $\De'$ is minimal and the claim holds with $\De=\De'$.

So assume instead that  $\alpha_i = \alpha_{i-1} + \alpha_{i+1}$ for some $i$.
By Lemma~\ref{le:3depblow}, this implies that 
there exists a polytope $\De$ which is a $\Delta_2$ bundle over a polygon
$\Hat \De$ so that $\De'$  can be obtained from $\De$ by blowing up
along the intersection of two base facets.
By Proposition~\ref{prop:polybundle}, since $H$ is essential the  fiber facets of $\De'$
are all asymmetric.   Hence,
by part (ii) of
Lemma~\ref{le:blowcent}, $H$ is
mass linear on $\De$
and the fiber facets of $\De$ are also
asymmetric.
Hence, by Proposition~\ref{prop:polybundle}, $H$ is essential on $\De$.  
By Corollary~\ref{cor:22bundle},
this is impossible if $\Hat \De$ has three edges.
Since this implies that $\Hat \De'$ has at least five edges, 
there are no 
nonzero
inessential mass linear functions on $\Hat \De'$
by Proposition~\ref{prop:2dim}.
Hence,  Proposition~\ref{prop:polybundle} 
and Lemma~\ref{le:blowcent} (ii) imply that the base facets of $\De$ are symmetric.
Thus, the blowup is along a 
symmetric $2$-face, as required.
The result now follows by induction.
\end{proof} 

\begin{proposition}\labell{prop:3depblowN}
For any $N \geq 7$, there exists a {\em minimal} polytope $\Delta$ that has $N$ facets,
is a $\Delta_2$ bundle over a polygon $\Hat \Delta$, and admits
an essential mass linear function.
\end{proposition}

\begin{proof}
Start with the polygon $\Delta_2$ with facets $e_1, e_2$, and $e_3$
with outward conormals $(-1,0)$, $(0,-1)$, and $(1,1)$, respectively.
Then form a polygon $\Hat\De$ with $k = N - 3 \geq 4$ sides by first blowing up  along 
$e_{13} = e_1 \cap e_3$,  
and every subsequent time  
blowing up along the intersection of the new exceptional divisor and $e_1$.
Then  $\Hat\De$ has edges $e_1,\dots,e_k$,  where 
$e_4, e_5,\dots, e_k$ are the  edges that are formed by the successive  blowup operation. 
(Thus they are labelled in order of adjacency.)  
Let $P(\kappa_1,\dots,\kappa_k)$ be the polynomial which gives the area 
of $\Hat \Delta(\Hat \kappa)$  for all $\Hat \kappa \in \Cc_{\Hat \Delta}$.
The blowup which introduces $e_j$ for $j \geq 4$ is performed by
cutting
out a triangle with affine side length 
$\ka_{j-1}-\ka_j + \ka_1$.  Hence
$$
P(0,0,\ka_3,\dots,\ka_k) = 
\tfrac 12 \bigl( \ka_3 \bigr)^2  -\tfrac 12\sum_{j=4}^k \bigl(\ka_{j-1}-\ka_j + \ka_1\bigr)^2.
$$

Let $r_1 = r_2 = 0$, $r_3 = \dots = r_{k-1} = 1$ and $r_k = 2$.
Fix integers $\gamma_1$ and  $\gamma_2$ such
that $\gamma_1 \gamma_2 (\gamma_1 + \gamma_2) \neq 0$, and
define 
\begin{equation}\labell{eqD3}
(b_1^i,b_2^i) = r_i (\gamma_2, -\gamma_1) \in \Z^2 \quad \mbox{for all } 1 \leq i \leq k.
\end{equation}
These can be used to construct a polytope $\Delta$ which is  a $\Delta_2$ bundle over $\Hat \Delta$
with the outer conormals to the fiber and base facets given by
Equations~\eqref{eqD1} and \eqref{eqD2}, respectively.
Since $P(0,0,r_3,\dots,r_k) = 0$ and $\gamma_1 \gamma_2 (\gamma_1 + \gamma_2) \neq 0$,
$\Delta$ admits an essential mass linear function by Proposition~\ref{prop:polybundle}.

Finally, by Equations~\eqref{eqD2} and \eqref{eqD3}, and   Lemma~\ref{le:3depblow}, 
the base facet of $\Delta$ associated to 
the edge $e_i$ of $\Hat \Delta$ cannot be blown down unless the
edge $e_i$ itself can be blown down and
$r_i = r_{i-1} + r_{i+1}$.  (Here, as always, we use cyclic order on the edges.)
However, if $k=5$ then $e_1$ and $e_5$ are the only edges of $\Hat \Delta$ 
that can be blown down; 
if $k \neq 5$ only $e_k$ can be blown down.
Since $r_k = 2  \neq 1 = r_{k-1} + r_1$ and $r_1 = 0 \neq 2 =  r_k + r_2 $ for all $k$,  
this implies that none of the base facets of $\Delta$ can be blown down.
The claim then follows from Lemma~\ref{le:a3fiber}.

\end{proof}

Finally, we consider polytopes of type (b), that is,
double expansions of polygons. 
As we see below, every polytope of type (b) can be obtained from
a minimal polytope of type (b) by a series of blowups.
As in the previous case, these minimal polytopes may have
arbitrarily many
facets.

\begin{lemma}\labell{le:4indepblow}
Let $\Delta'$ be the double
expansion of a polygon $\Tilde \De'$
with edges\footnote{We do
not assume that the edges are labelled in order of adjacency.}
$\Tilde F_1,\dots, \Tilde F_k$
along the edges 
$\Tilde F_1$ and $\Tilde F_2$.
\begin{itemize}
\item[(i)]
If $\Tilde F_i$ cannot be blown down for any $i > 2$, then 
$\De'$ is minimal.
\item [(ii)]
In contrast, if $\Tilde \De'$ is the blowup of a polygon $\Tilde \De$ with
exceptional divisor $\Tilde F_i$ for some $i > 2$, let $\De$ be
the double expansion of $\Tilde \De$ along the edges 
$P(\Tilde F_1) \cap \Tilde \De$
and $P(\Tilde F_2)  \cap \Tilde \De$. 
Then $\De'$ is the blowup of $\De$ along a face 
of one of the following three types:
\begin{itemize}
\item 
the intersection of two fiber-type facets.
\item 
the intersection of a 
fiber-type facet with either
$P(\Hat F_{12}) $ or $P(\Hat F_{34})$.
\item
the vertex $\cap_{i=1}^4 P(\Hat F_i)$.
\end{itemize}
Here, $\Hat F_1$ and $\Hat F_2$ 
($\Hat F_3$ and $\Hat F_4$)
are the base-type facets 
of $\De'$
associated to $\Tilde F_1$
(respectively, 
$\Tilde F_2$).
\end{itemize}
\end{lemma}

\begin{proof}
Let $\io:\R^2\to \R^4$ be inclusion into 
the
first two coordinates, 
and let 
the outward
conormals 
to the edges 
$\Tilde F_1,\dots,\Tilde F_k$  of the 
polygon $\Tilde \De'$
be $\Tilde\al_1,\dots,\Tilde\al_k$.
Then the conormal to the fiber-type facet $F_j$ of 
$\De'$ 
associated to $\Tilde F_j$ is $\al_j = \io(\Tilde\al_j)$ for $j>2$, 
and  the conormals to the base-type facets $\Hat F_1,\dots,\Hat F_4$  are
$$
\eta_1= (0,0,-1,0),\;\;\eta_2 = \io(\Tilde\al_1) + (0,0,1,0),\;\;
\eta_3= (0,0,0,-1),\;\;\eta_4 = \io(\Tilde\al_2) + (0,0,0,1). 
$$
By Remark~\ref{rmk:dint} (ii) and Lemma~\ref{le:noblow},
none of the base-type facets can be blown down.
Now fix $j >2$,
and assume that the
 fiber-type 
 facet
  $F_j$ can be blown down.
Let  $\Tilde F_k$ and $\Tilde F_\ell$ be the edges of $\Tilde \De$ that meet $\Tilde F_j$.

Assume first that  
$k$ and $\ell$ are both greater than $2$.
Then 
$F_j$ is a $\Delta_1$ 
bundle over $\Delta_1 \times \Delta_1$,
with opposite base facets given by its intersections with 
$\Hat F_1$ and $\Hat F_2$,  and with $\Hat F_3$ and $\Hat F_4$, and
with fiber facets $F_{jk}$ and $F_{j \ell}$.
Since  $\Tilde \al_j \neq \Tilde \al_1$, 
we have
$\al_j  \neq  \eta_1 +  \eta_2$; similarly,
$\al_j \neq  \eta_3 +  \eta_4$. 
Therefore, 
by Proposition~\ref{prop:blowdown},
$\al_j$ must be equal to $\al_{k} + \al_{\ell}$, and so  $\Tilde F_j$ 
blows down in 
$\Tilde \De'$.   

Next suppose that $k=1$ and $\ell>2$. 
Then $F_j$ is a
$\De_2$ bundle over $\De_1$ with fiber facets   
$\Hat F_1 \cap F_j, \Hat F_2 \cap F_j,$ and $F_k \cap F_j$, and base facets $\Hat F_3 \cap F_j$ and $\Hat F_4 \cap F_j$. 
Since the equation $\al_j=\eta_3 + \eta_4$
is never satisfied, 
Proposition~\ref{prop:blowdown}  implies that
$\al_j = \eta_1 + \eta_2 + \al_k$. Hence 
$\Tilde \al_j=\Tilde \al_1 + \Tilde \al_k$ so that
 $\Tilde F_j$ blows down in 
$\Tilde \De'$.

Finally suppose that 
$k = 1$ and $\ell = 2$.
Then $F_j$ is a $3$-simplex 
with facets 
$\Hat F_i\cap F_j$ for $i =1,\dots,4$. 
Therefore, in this case Proposition~\ref{prop:blowdown} implies that
$\sum_{i=1}^4 \eta_i = \al_j$.
Once again, this implies that $\Tilde F_j$ blows down in $\Tilde \De'$. 
 This proves (i).
 
To prove (ii) it remains to check that in each of the three cases 
considered above $\De'$ is the blowup of $\De$ along an  
appropriate 
face.  We leave this to the reader.
\end{proof}

\begin{prop}\labell{prop:4indepblow}
Let $\Delta' \subset \ft^*$ be the double expansion of a polygon.
Then $\Delta'$ can be obtained from a
 {\em minimal} polytope $\Delta$ that is
also the double expansion of a polygon  
by a series of blowups. 
Moreover,
if $H \in \ft$ is an inessential function on $\De'$ and
the asymmetric facets are the base-type facets,
then 
\begin{itemize}\item
$H$ is also an inessential function on $\De$, 
\item
the asymmetric
facets of $\De$ are the base-type facets, and 
\item
each blowup is either one
of the two types permitted in Theorem~\ref{thm:4d} or 
is  at the vertex formed by
the intersection of the four asymmetric facets.
\end{itemize}
\end{prop}

\begin{proof}  The first claim is immediate from Lemma 
\ref{le:4indepblow}. 
The second claim follows from Lemmas~\ref{le:dexpan0} and \ref{le:4indepblow}.
\end{proof}

\begin{prop}\labell{prop:4indepblowN}
For any $N \geq 5$ there exists a minimal polytope 
$\Delta \subset \ft^*$
that has $N$ facets, is the double expansion of a polygon $\Tilde \De$,
and admits an inessential function $H \in \ft$ so that the asymmetric
facets are the base-type facets.
Moreover, when 
$N=6$ or $N\ge 8$
we can choose $\Delta$
so that $H$ is essential on a polytope $\Delta' \subset \ft^*$
that can be obtained from $\Delta$ by a sequence of
blowups, where each blowup is either along
a symmetric face or of type $(F_{ij},G)$.
When 
$N$ is $5$ or $7$ 
the previous statement holds only if we also allow blowups
at the point where the four base-type facets intersect.
\end{prop}

\begin{proof}
Start with the simplex $\Delta_2$ with facets $e_1$, $e_2$, and $e_3$.
Then form a polygon $\Tilde \De$ with $k = N - 2$ by first blowing up along
$e_{13} = e_1 \cap e_3$ and every subsequent 
time blowing up along
the intersection of the new exceptional divisor and 
$e_1$ 
as in the proof of Proposition \ref{prop:3depblowN}.
If $k = 5$ (and $N = 7$)
let $\Delta$ be the double expansion of $\Tilde\De$ along 
$e_5$ and  $e_1$; otherwise let $\Delta$ be the  double
expansion of $\Tilde \De$ along $e_k$ and $e_{k-2}$.
If $k = 5$, then $e_1$ and $e_5$ are the only edges of 
$\Tilde\De$
that can be blown down; if $k = 3$ then no edge can be blown down;
otherwise, $e_k$ is the only edge that can be blown down.
Therefore, $\Delta$ is minimal by Lemma~\ref{le:4indepblow},
and it admits an inessential mass linear function so
that the asymmetric facets are the four base-type facets
by Lemma~\ref{le:dexpan0}. 
This proves the first claim.

To prove the second claim, observe that
 if $k\ne 3, 5$ then Proposition~\ref{prop:essblow2} implies that there is a blowup $\De'$ of 
 $\De$ with the required properties.
However, this argument
does not work when $k=5$ since we expanded along adjacent edges $e_1, e_5$ in order to make $\De$ minimal.  
In fact every polygon $\Tilde\De$ with $5$ edges is a blowup 
of a trapezoid and so always has two adjacent edges that can be blown down.
Therefore, if $k = 3$ or $5$ Proposition~\ref{prop:essblow} implies
that there cannot be a blowup $\De'$ of $\De$ with the required properties.
On the other hand, if we allow vertex blowups then
we can find such a blowup for $k = 3$ or $5$ by 
Remark~\ref{rmk:vertexblow}.
This proves the final claim.
\end{proof}

We are now ready to summarize the results of this subsection.
  
\begin{rmk}\labell{rmk:blowess}\rm (i)
In this section, we 
have shown
that each polytope $\ov\De$ described in case (a) of
Theorem~\ref{thm:4d} is the blowup of a minimal polytope of type (a), and  that
each polytope $\ov\De$ described in case (b) is the blowup of a minimal polytope of type (b).
More specifically, Proposition~\ref{prop:a1blow}
shows that  polytopes of type (a1)  with an essential mass linear function 
are minimal, while 
Proposition 
\ref{prop:3depblow} 
shows that 
a polytope of type (a3) that admits an essential mass linear function
can be obtained from  a minimal polytope 
with the same properties
by a series of blowups along symmetric $2$-faces.
In contrast, by
Proposition 
\ref{prop:a2blow},
a polytope of type (a2) with an essential mass linear function may be
the blowup of a minimal polytope 
along the intersection of the
three asymmetric fiber facets -- 
not one of the types permitted in Theorem~\ref{thm:4d}.
Finally, if 
we assume that
$H$ is inessential on $\ov \De$, 
that
the polytope $\ov \De$ is the
double expansion of a polygon, and
that
 the asymmetric facets are the base-type facets 
as in case (b),
 then 
we may conclude from 
Proposition \ref{prop:4indepblow} that
$\ov \De$ is the blowup of a minimal polytope with the same properties.
However, one blowup may be along the vertex formed by the intersection of the four asymmetric
base-type facets, which 
is not one of the types permitted in Theorem~\ref{thm:4d}.
 
\MS

\NI (ii) By Remark \ref{rmk:a2blow},  in the case (a2)  
necessary and sufficient conditions for 
  $\ov\De$ to be minimal 
  depend on $\ka$.  But this is not true  in the
   cases (a3) and (b).  In case (a3) the facets  
$F_1,F_2,$ and $F_3$ 
have a very simple intersection pattern, which forces the troublesome 
  condition (iii) in Proposition \ref{prop:blowdown} to hold.  In case (b), on the other hand,
   there is no choice for the relevant components of $\ka$: in an expansion the 
  components of $\ka$ corresponding to the base facets are determined by the structural constants of the fiber.

\end{rmk}

\subsection{Full mass linearity}\labell{ss:fullML}

We now discuss a  
strengthening of the mass linear condition 
proposed by Shelukhin in \cite{Shel}: 
namely,
``fully mass linear."
In this subsection, we prove that every mass linear function on a polytope of dimension
at most four also satisfies this stronger condition.

As we mentioned in the discussion just before Question \ref{q:ml}, if $H$ generates a loop 
that lies in the kernel of the homomorphism 
$\pi_1(T)\to \pi_1\bigl(\Symp(M_\De,\om_\De)\bigr),$ 
then $H$ is fully mass linear.
Therefore the class of fully mass linear functions may be 
more natural than the class of mass linear functions.

\begin{defn}[Shelukhin \cite{Shel}]\labell{def:full}
 Let $\De \subset \ft^*$  
be an $n$-dimensional polytope, and for $k=0,\dots,n$ denote by $B_k$ the barycenter 
(center of mass)
of the union of the  
$k$-dimensional 
faces of $\De$. Thus 
$B_n(\De) = c_\De$ and $B_0(\De)$ is the average of the vertices of $\De$. 
 Then $H\in \ft$ is said to be {\bf fully mass linear} on $\De$ if
$$
\langle H, B_k(\Delta) \rangle = \langle H, B_n(\Delta) \rangle
\quad \mbox{for all } k = 0, \dots, n.
$$
Further, we say that $H \in \ft$ 
is {\bf generated by the vector}
$\mathbf{ \xi_H \in \ft^*}$ 
if
$$\langle H, c_{\De}(\kappa) \rangle = \sum_{i=1}^N \langle\eta_i,\xi_H\rangle \kappa_i.$$
\end{defn}

The coordinates of $B_0(\De)$ are linear functions of $\ka$.
Thus,  if $\langle H, B_n(\De) \rangle = \langle H, B_0(\De) \rangle$ then $H$ is mass linear on 
$\De$;  in particular,
every fully mass linear function is mass linear.  
A priori, the converse may not hold.  For example,
the three
barycenters $B_0(\De), B_1(\De)$ and $B_2(\De)$ of a generic trapezoid   are  distinct.
However, we do not know any of counterexamples, and
it does hold in dimension at most four.

\begin{prop}\labell{prop:fullML}
Every mass linear function  
on a polytope of dimension 
at most four 
is fully mass linear.
\end{prop}

The converse also holds for inessential functions.

\begin{lemma}\labell{le:fullinessential}
Every inessential function $H \in \ft$ on a polytope $\Delta \subset \ft^*$ is fully mass linear and is generated by some vector $\xi_H$. 
\end{lemma}

\begin{proof}
It suffices to consider the case when $H = \eta_i-\eta_j$ where $F_i\sim F_j$.
By Remark \ref{rmk:equiv}, since  
$F_i\sim F_j$ there is a vector $\xi_H\in \ft^*$ that is parallel to all
facets except for $F_i$ and $F_j$.  It follows easily that the affine reflection in the plane
 $H=\eta_i-\eta_j = 0$ preserves $\De$.  Thus 
all the barycenters $B_k(\De)$ lie on the plane $H=const$  that is fixed by this reflection.
Moreover, the integrality conditions on the conormal vectors $\eta_i$  imply that
$\langle \xi, \eta_i \rangle = - \langle \xi, \eta_j \rangle = 1$; 
see part I, Lemma 3.4.
Hence $H$ is generated by $\xi_H$.
\end{proof}

We prove 
Proposition~\ref{prop:fullML} 
by showing that
all the mass linear  pairs $(\De,H)$  described in
Proposition~\ref{prop:3d} and
Theorem \ref{thm:4d} are fully mass linear. 
The following basic result is 
taken from  McDuff \cite{Mct}.

\begin{lemma}\labell{le:full0}
Let $\Delta \subset \ft^*$ be an $n$-dimensional polytope and fix $H \in \ft$.
\begin{itemize}
\item[(i)] 
The function $H$ 
is mass linear 
exactly if
$$
\langle H, B_0(\De) \rangle = \langle H, B_{n-1}(\De) \rangle   =\langle H, B_n(\De) \rangle.
$$  
\item[(ii)] 
There is a vector
$\xi_H\in \ft^*$ that generates 
$H$ exactly if
$$ \langle H, B_0(\De) \rangle =\langle H, B_{n-2}(\De) \rangle = \langle H, B_n(\De) \rangle  .$$
\end{itemize}
\end{lemma}
\begin{proof}  
Part (i)  
is proved in  \cite[Proposition~4.7]{Mct}, while
(ii)  
holds by \cite[Remark~4.10]{Mct}.
\end{proof}

\begin{cor}\labell{cor:full0}  
If $H \in \ft$ is mass linear on a polytope $\De \subset \ft^*$ 
then $\sum \ga_i=0$, where
$\langle H, c_\De(\ka) \rangle = \sum \ga_i\ka_i$. 
\end{cor}

\begin{proof}
By Lemma~\ref{le:full0}, if $H$ is mass linear
then $\langle H, B_{n-1}(\De) \rangle = \langle H, c_\De(\ka) \rangle$.
It is shown in \cite[Lemma~4.5]{Mct} that 
this implies that 
$\sum \ga_i=0$.
\end{proof}

\begin{rmk}\labell{rmk:sum}\rm 
Since inessential functions have the property that  $\sum \ga_i=0$ one could prove this corollary by induction on the dimension provided that Question \ref{q:1} has a positive answer.  For then, after subtracting an inessential function, we can assume that  every mass linear $H$   has a symmetric facet $G$ and use the fact that the coefficients of $H$ are the same as those for $H|_G$.    This is the approach taken in Part I to prove $\sum\ga_i=0$  in dimension $3$ (cf. Proposition~\ref{prop:3d}), and by [I, Theorem A.9] it works also in dimension $4$.
\end{rmk}

\begin{lemma}\labell{le:full1} 
Let $H \in \ft$ be a mass linear function on an $n$-dimensional
polytope $\De \subset \ft^*$,
where $\Delta$ is a $\De_k$ bundle
over $\De_1$, a $121$-bundle,  or  a  $\De_2$-bundle over a polygon $\Hat\De$. 
Then $$
 \langle H, B_{n-2}(\Delta) \rangle = \langle H, B_n(\Delta) \rangle.
 $$
\end{lemma}

\begin{proof}
By Lemma~\ref{le:full0}, it suffices to
show that $H$ is generated by some $\xi_H \in \ft$.
Moreover, recall from  Lemma~\ref{le:fullinessential} that
every inessential function $H' \in \ft$ is generated by some vector in $\ft$.
Hence, we may subtract any
convenient
inessential function.
We now check case by case that there is a suitable vector $\xi_H$.

First suppose that $\De$ is a $\De_k$ bundle $Y$ over $\De_1$, 
as in \eqref{eq:Yb}. 
By Proposition \ref{DkoverD1}, after possibly subtracting an inessential
function,
we may assume that
$H = \sum_{i=1}^{k+1} \gamma_i \eta_i$, where 
$ \sum_{i=1}^{k+1} \gamma_i =  \sum_{i=1}^k a_i \gamma_i = 0$;
moreover,
$\langle H, c_{\De}(\kappa) \rangle = 
\sum_{i=1}^{k+1} \gamma_i \kappa_i.$
Therefore, if 
$$
\xi_H: = - (\gamma_1,\dots,\gamma_{k},0),
$$
then
$\langle H, c_{\De}(\kappa) \rangle = 
\sum_{i=1}^{k+3} \langle \eta_i,\xi_H\rangle \kappa_i,$
that is, $H$ is generated by $\xi_H$.

Next, let $\De$ is a $121$-bundle, as in \eqref{eq:de'}. 
By Proposition~\ref{prop:doublebundle}, 
we may assume that  $H =  \sum_{i=2}^4 \gamma_i \eta'_i$,
where
$\gamma_2 + \gamma_3 + \gamma_4 = a_2 \gamma_2 + a_3 \gamma_3 = 0;$
moreover, 
$\langle H, c_{\De}(\kappa) \rangle = 
\sum_{i=2}^{4} \gamma_i \kappa_i.$
Therefore, $H$ is generated by $\xi_H = (0, -\gamma_2, - \gamma_3, 0)$.

Finally, let $\De$ be a $\De_2$ bundle over a polygon $\Tilde \De$, as in \eqref{eqD2}.
By Proposition~\ref{prop:polybundle}, 
we may assume that  $H =  \sum_{i=1}^3 \gamma_i \eta_i$,
where
$\gamma_1 + \gamma_2 + \gamma_3 =  0$
and $b_1^j \gamma_1 + b_2^j \gamma_2 = 0$ for each edge $e_j$ of $\Tilde \De$;
moreover, 
$\langle H, c_{\De}(\kappa) \rangle = 
\sum_{i=1}^{3} \gamma_i \kappa_i.$
Hence, $H$ is generated by $\xi_H = (-\gamma_1, -\gamma_2,0)$.
\end{proof}

\begin{cor} \labell{cor:f1}
Every mass linear function on a 
polytope
of dimension at most three
is fully mass linear.
\end{cor}

\begin{proof} 
By Lemma~\ref{le:fullinessential}, every inessential function is fully mass linear. 
Moreover, by Propositions~\ref{prop:2dim} and \ref{prop:3d}, 
the only  polytope  of dimension at most three 
that supports an essential mass linear function $H$ is a $\De_2$ bundle over $\De_1$.
Hence, the claim follows from
part (i) of  Lemma \ref{le:full0} and the first case of Lemma~\ref{le:full1}.
\end{proof}

Given a set of edges $\Ee$ of $\De$,
we denote by $B_1(\Ee)$ the corresponding barycenter.

\begin{lemma} \labell{le:f0}
Let $H \in \ft$ be a mass linear function on an $n$-dimensional polytope $\Delta$. 
 Let $\Ee$ be the set of edges 
that lie on at least one symmetric facet.  Assume that $H|_f$ is fully mass linear on $f$ for every symmetric face
$f$.  Then
$$ \langle  H, B_1(\Ee) \rangle = \langle H, B_n(\De) \rangle.$$
\end{lemma}

\begin{proof}
Let $f$ be any symmetric $k$-face.
Since $\langle H, B_1(f) \rangle = \langle H, B_{k}(f) \rangle$ by assumption, Proposition~\ref{prop:symcent} implies
that $\langle H, B_1(f) \rangle = \langle H, B_n(\Delta) \rangle$.
By induction on $k$, this implies that $\langle H, B_1(\Ee_f) \rangle = \langle H, B_n(\Delta) \rangle$,
where $\Ee_f$ is the set of edges that lie on $f$ but do not lie on any smaller symmetric face.
The result follows immediately.
\end{proof}

\begin{lemma} \labell{le:f1} 
Let $H \in \ft$ be a mass linear function on an $n$-dimensional polytope $\De \subset \ft^*$.
If $\De$ is a $\De_k$ bundle
over $\De_1$, a $121$-bundle,  or  a  $\De_2$-bundle over a polygon $\Hat\De$,  
then $$\langle H, B_1(\De) \rangle = \langle H, B_n(\De) \rangle.$$
\end{lemma}

\begin{proof}
As before,
Lemma~\ref{le:fullinessential} implies that
we may subtract any convenient inessential function.

First suppose that $\De$ is a $\De_k$ bundle $Y$ over $\De_1$, 
as in \eqref{eq:Yb}.
By Proposition~\ref{DkoverD1}, after possibly subtracting an inessential function,
we may assume that 
\begin{equation}\labell{le:f1:e1}
H = \sum_{i=1}^{k+1} \gamma_i \eta_i,
\quad \mbox{where} \quad \sum_{i=1}^{k+1} \ga_i =   \sum_{i=1}^k a_i\ga_i=0;
\end{equation}
moreover
$\langle H, c_\Delta \rangle = \sum_{i=1}^{k+1} \gamma_i \kappa_i$.
In particular, in coordinates 
we have
$$
H = 
\bigl(\ga_{k+1}-\ga_1, \cdots, \ga_{k+1}-\ga_k,0\bigr).
$$

Divide the edges of $Y$ into two groups  $\Ee_1$ and $\Ee_2$, 
where $\Ee_1$ consists of those edges that lie 
in one of the  base facets,  
and $\Ee_2$ consists of the remaining edges, which  are parallel to the last coordinate axis $e_{k+1}$. 
The base facets are symmetric.
Therefore, by Lemma~\ref{le:f0},
\begin{equation}\labell{le:f1:e2}
\langle H, B_1(\Ee_1) \rangle  = \langle H, B_n(\Delta) \rangle .
\end{equation}
Let $\lambda = \sum_{i=1}^{k+1} \kappa_i$ and $h = 
\sum_{i=1}^k a_i \kappa_i + \kappa_{k+2} + \kappa_{k+3}$.
There are $k+1$ edges in $\Ee_2$, one over 
$v_0 = -(\kappa_1,\dots,\kappa_k,\kappa_{k+2})$ of length $h$, 
and one over the vertex 
 $v_i$ at $ v_0 + \la e_i$ 
 of length
$h+a_i\la$  for all $1 \leq i \leq  k$.  Therefore,
$$
B_1(\Ee_2) =  v_0 + \frac{\lambda}{h (k+1) + \lambda \sum a_i} 
(h+a_1\la, \dots ,h+a_k\la,x),
$$
for some constant  $x$.
Thus, 
\begin{equation}\labell{le:f1:e3}
\langle H, B_1(\Ee_2) \rangle = \sum_{i=1}^{k+1} \gamma_i \kappa_i =
\langle H, B_n(\Delta) \rangle.
\end{equation}
Together, \eqref{le:f1:e2} and \eqref{le:f1:e3} imply that
$\langle H, B_1(\Delta) \rangle = \langle H, B_n(\Delta)$, as required.

Next, suppose that $\De$ is a $121$-bundle as in \eqref{eq:de'}.  
By Proposition~\ref{prop:doublebundle}, after possibly
subtracting an inessential function, we may assume that the  $\De$ has only three asymmetric facets:
$F_2', F_3'$, and $F_4'$.  Since these facets 
do not
 intersect, every edge of $\De$ lies on at least one symmetric facet.
Moreover, the
restriction
$H|_f$ is fully mass linear on every symmetric face $f$ by Corollary~\ref{cor:f1}.
Therefore, the claim follows immediately from Lemma~\ref{le:f0}.

Finally, assume that $\De$ is $\De_2$ bundle over $\Tilde \De$.
By Proposition~\ref{prop:polybundle}, after possibly subtracting an inessential function we may
again assume that $\De$ has only three asymmetric facets, and that these facets 
do not intersect.
The argument follows exactly as above.
\end{proof}

\begin{lemma}\labell{le:f4}  Let $H$ 
be a totally mass linear function on a
polytope $\De$  
and suppose that $\De'$ is 
a
blowup of $\De$ either 
of type $(F_{ij},g)$, or
along a symmetric face $f$ such that $H|_f$ is inessential.
Then
$H$ is totally mass linear on $\De'$.
\end{lemma}

\begin{proof}
Let $\De'$ be the blow up of $\De$ along a face $f = F_I$.
Write $\De = \De'\cup W$, where $F_0': = W \cap \De'$ is the exceptional divisor.  
Then $W$ is a $\Delta_{|I|}$ bundle over  $f$.

We claim that 
\begin{enumerate}
\item $\langle H, c_W(\kappa) \rangle = 
\langle H, c_\Delta(\kappa) \rangle$, and
\item  the restriction of $H$ to $W$ is inessential.
\end{enumerate}
If $f$ is 
symmetric,
we saw in the proof of Lemma~\ref{le:symblow} that
$\langle H, c_f(\kappa) \rangle = \langle H, c_\Delta(\kappa) \rangle$ and
the restriction of $H$ to $W$ is the lift of $H|_f$ from $f$ to $W$. 
Hence, both claims follow from Proposition~\ref{prop:lift}.
On the other hand, if $\Delta'$ is a blowup of type $(F_{ij}, g)$,
then both claims are 
explicitly
proved in the penultimate paragraph
of Proposition~\ref{prop:blow4}.

In either case,
by  Proposition~\ref{prop:symcent} and 
Lemma~\ref{le:newblow},
claim (1) implies that
$$
\langle H, B_n(\De') = \langle H, B_n(\Delta) \rangle = 
\langle H, B_n(W) \rangle = \langle H, B_{n-1}(F_0') \rangle. 
$$
Moreover,  since $F_0'$ 
is
symmetric, claim (2) 
and 
Remark \ref{rmk:symcent}
imply that
$H$ is also inessential on $F_0'$,
and so 
$H$ is fully mass linear on both $W$ and $F_0'$
by Lemma~\ref{le:full0}. 
Hence, the equation above implies that
$$
\langle H, B_n(\Delta') \rangle = \langle H, B_1(\Delta) \rangle = 
\langle H, B_1(W) \rangle = \langle H, B_{1}(F_0') \rangle. 
$$

In either case,  let us first
consider the effect of blowing up on $\langle H, B_1(\De) \rangle$.
Each edge of $\De$ which does not meet $f$ is an edge of $\De' \less F_0'$.
The edges of $\De \less f$ which meet $f$ are cut by the hyperplane $P(F_0')$ containing $F_0'$ into two pieces,
one of which is an edge of $W \less F_0'$, and the other an edge of $\De' \less F_0'$.
Finally, each edge of $f$ is an edge of $W \less F_0'$.
The remaining edges of $W$ and $\De'$ lie in $F_0'$.
Therefore
$$V_1(\De)  + 2 V_1(F_0') =  V_1(\De') + V_1(W),$$ 
where $V_1(X)$ denotes the sum of the lengths of the edges of $X$. 
Moreover, by the additivity of the $H$-moment,  we have
$$
\langle H,B_1(\De)\rangle \, V_1(\De) +
2 \langle H,B_1(F_0')\rangle\,V_1(F_0')
= 
\langle H,B_1(\De')\rangle \,V_1(\De') + \langle H,B_1(W)\rangle\,V_1(W).
$$
Since $V_1(\De') \neq 0$, the last three  displayed 
equations 
above imply that
$\langle H, B_1(\Delta') \rangle = \langle H, B_n(\Delta') \rangle$.
A nearly identical argument -- but with $i$-faces instead of edges -- implies that
$\langle H, B_i(\Delta') \rangle = \langle H, B_n(\Delta') \rangle$ for all $i$.
\end{proof}

\NI{\bf Proof of Proposition \ref{prop:fullML}.}\
This holds by combining Theorem \ref{thm:4d} with 
Lemmas 
\ref{le:fullinessential},
\ref{le:full1}, 
\ref{le:f1} and \ref{le:f4}.  
Note that we can always apply Lemma \ref{le:f4} because
Proposition~\ref{prop:2dim} implies that 
the restriction of $H$ to
 a symmetric $2$-face is  inessential.
\QED

\subsection{Blowups and mass linearity}\labell{ss:blml}

We end 
with a general discussion about 
the effect of blowing up on mass linearity. 
Lemma~\ref{le:blowcent}  shows that if a mass linear function 
 on a polytope remains mass linear on its blowup along a face $f$,
the face $f$ must meet all the 
asymmetric facets.
The following example shows that this condition is not sufficient.

\begin{example}\rm  Let $\De_4$ be the $4$-simplex and
let $H = \sum_{i=1}^5 \ga_i\eta_i$,  where 
$\sum_{i=1}^5  \ga_i=0$ and $\eta_1,\dots,\eta_5$
are the outward conormals to the facets of $\De_4$.
The blowup $\De'$ of $\De_3$ along the edge $F_{123}$ is a $\De_2$  bundle over $\De_2$ with base facets
$F'_1, F_2',$ and $F_3'$, where $F_i' =  F_i \cap \Delta'$ for all $i$.
Then 
$H$ is mass linear on $\De_4$ and 
$F_{123}$ meets all asymmetric facets. However, by
Proposition~\ref{prop:polybundle} $H$ is mass linear on $\De'$  
exactly if $\gamma_1 + \gamma_2 + \gamma_3 = 0$.
\end{example}

In the above example it is enough to add the condition 
$\ga_1+\ga_2+\ga_3=0.$
However, 
to get a general result we need 
yet more conditions. 
  
\begin{lemma}\labell{blowsimp}
Let $H \in \ft$ be mass linear on a polytope $\De \subset \ft^*$.
Let
$\De'$ be the blowup of $\De$ along
a face $F_I$ which meets every asymmetric facet of 
$\De$, and assume that  $\sum_{i \in I} \gamma_i = 0$.
Write $\De = \De'\cup W$; if 
$F_I$ is a simplex and
$H$ is mass linear on $W$,
then $H$ is mass linear on $\De'$.
\end{lemma}

\begin{proof} 
We aim to show that 
$\langle H, c_W \rangle$ is equal to
$\langle H,c_\De\rangle$.  The result then follows from
Lemma \ref{le:newblow}.

 Let $\{F_\ell\}_{\ell \in L}$ be the set of facets that  meet $F_I$,
and let $\eta_\ell$ be the outward conormal to $F_\ell$ for all $\ell$.
Since by hypothesis every asymmetric facet meets $F_I$, 
we may write $\langle H,c_\De\rangle = \sum_{\ell\in L}\ga_\ell\ka_\ell$.
Because $\sum_{\ell\in L}\ga_\ell = 0$ by 
Corollary \ref{cor:full0},
our hypotheses imply that 
$$
\sum_{i\in I}\ga_i = \sum_{j\in J}\ga_j = 0,
$$
where $J = L \smallsetminus I$.

There  is a facet $F'_0$ of $W$ with outward conormal 
$\eta_0 = - \sum_{i \in I} \eta_i$ (corresponding to the exceptional divisor
in $\Delta'$);
the remaining facets of $W$ are $\{F_\ell \cap W\}_{\ell \in L}$.
Since $H$ is mass linear on $W$,  we may write
$\langle H,c_W\rangle = \gamma_0' \kappa_0 + \sum_{\ell\in L}\ga_\ell'\ka_\ell$. 
Additionally, observe  that  
$W$ is a $\De_{|I|}$ bundle over $F_I$; cf.\ Remark \ref{rmk:blowint}.
Its fiber facets are
$F_0$ and  $\{F_i\cap W\}_{i\in I}$;  
its base facets are $\{F_j \cap W\}_{j \in J}$. 
Therefore  there is precisely one linear relation among the conormals 
$\{\eta_\ell\}_{\ell\in L}$; it has  the form
\begin{equation}\labell{eq:j0}
 \sum_{j \in J} \eta_j = \sum_{i\in I} a_i \eta_i.
\end{equation}
Moreover, by Proposition \ref{prop:bund}
and Corollary~\ref{cor:full0}
we have
$$ 
\gamma_0' + \sum_{i\in I}\ga_i'= \sum_{j \in J}\ga_j' = 0.
$$ 

Next, note that  by Lemma \ref{le:Hsum}, 
$H = \sum_{\ell \in L} \gamma_\ell \eta_\ell$ and $H = 
\ga'_0 \eta_0 + \sum_{\ell \in L} \ga'_\ell \eta_\ell$.
Hence, 
since $\eta_0 = - \sum_{i \in I} \eta_i$,
if   
we use Equation \eqref{eq:j0} to write
$\eta_{j_0} = - \sum_{j \ne j_0} \eta_j + \sum_{i\in I} a_i \eta_i$
for some $j_0 \in J$ ,
 we see that
$$
\sum_{i \in I} (\ga_i + a_i \gamma_{j_0}) \eta_i + 
\sum_{j \in J \smallsetminus \{j_0\}} (\ga_j - \gamma_{j_0}) \eta_j =  
\sum_{i \in I} (\ga'_i - \ga'_0 + a_i \gamma'_{j_0}) \eta_i + 
\sum_{j \in J \smallsetminus \{j_0\}} (\ga'_j - \gamma'_{j_0}) \eta_j .$$
Since the vectors $\{\eta_\ell\}_{\ell \in L \smallsetminus \{j_0\}}$
are linearly independent, this implies that
$\gamma_j - \gamma_{j_0} = \gamma'_j - \gamma'_{j_0}$ for all $j \in J$,
and 
$\gamma_i + a_i \gamma_{j_0}
= \gamma'_i -\gamma'_0 + a_i \gamma'_{j_0}$ for all $i \in I$.
Since $\sum_{j \in J} \gamma_j = 0 = \sum_{j \in J} \gamma'_j$,
the first equation implies that $\gamma_j = \gamma'_j$ for all $j \in J$. 
Hence, since $\sum_{i \in I} \gamma_i = 0 = \gamma'_0 + \sum_{i \in I}
\gamma'_i$, the second equation implies that $\gamma'_0 = 0$
and $\gamma_i = \gamma'_i$ for all $i \in I$.
Therefore
$\langle H,c_W\rangle = \langle H,c_\De\rangle$ as claimed.
\end{proof}

The difficulty now is to understand when the 
restriction of $H$ to $W$ 
is mass linear.
Here is a simple example.

\begin{cor}\labell{cor:blowv}  Let $H$ be a mass linear 
function
on $\De$
and $v$ any vertex of $\De$. 
Then $H$ is mass linear on the blowup $\De'$ of $\De$ at $v$ exactly if
$v$ lies on all asymmetric facets of $\De$.
Moreover, if $H$ is inessential on $\De$ then it is inessential on the blowup.
\end{cor}

\begin{proof}
If $F_I$ is a vertex then  $W$ is a simplex so that all 
$H$ are mass linear on $W$. Moreover, 
the equality $\sum_{i\in I} \ga_i = 0$ 
holds by Corollary \ref{cor:full0}.  
Therefore if $v$ lies on all asymmetric facets $H$ is 
mass linear on $\De'$ by
 Lemma \ref{blowsimp}.
The converse follows from 
Lemma~\ref{le:blowcent}.
This proves the first claim.
The second follows from
Lemma \ref{le:inessblow}.
 \end{proof}

Now consider the case when $F_I$ is an edge
that meets every asymmetric facet, and assume 
that $\sum_{i \in I} \gamma_i = 0$.
Then
$W$ is a $\De_{|I|}$ bundle over $\De_1$,  
 and so Proposition~\ref{DkoverD1} implies that $H$ is mass linear on
$W$ (and hence $\Delta'$) exactly if  $\sum_{i \in I} a_i\ga_i=0$.
It turns out that this  condition, which involves the relative slope of the 
two facets transverse to $F_I$,
is satisfied whenever $H$ is {\bf generated  by} $\xi_H$ in the sense of Definition \ref{def:full}.

\begin{lemma}\labell{le:blowedge}  Let $H \in \ft$ 
be a mass linear function 
on  a polytope $\De \subset \ft^*$.  Let $\Delta'$ be the blowup
of $\Delta$ along an edge $F_I$ which meets every asymmetric facet of $\Delta$,
and assume that $\sum_{i \in I} \gamma_i = 0$.
If $H$ is generated by some $\xi_H \in \ft^*$,
then $H$ is mass linear on $\De'$,
and is generated on $\De'$ by the same vector $\xi_H$.
\end{lemma}

\begin{proof} Number the facets 
of $\De$ so that $I = \{1,\dots, n-1\}$ 
and so that the two facets
of $\De$
that are transverse to the edge $F_I$ are $F_n$ and $F_{n+1}$.
Choose coordinates so that $\eta_i = -e_i\in \R^n$ for $i\le n$.
Then $\eta_{n+1} = 
(a_1,\dots,a_{n-1},1)$  for some $a \in \Z^{n-1}$.
Moreover, these facets, together with the facet $F_0$ corresponding
to the exceptional divisor of $\Delta'$,
are the facets of $W$. 

By Lemma~\ref{le:Hsum}. $H = \sum_{i\le n+1} \ga_i\eta_i$.
By assumption,
 $\ga_i = \langle\eta_i,\xi_H\rangle$ for all $i$; therefore,
 $\xi_H = (-\ga_1,\dots,-\ga_{n-1}, -\ga_n)$.
Moreover,  since $\sum_{i\le n-1}\ga_i=0$ by hypothesis and
$\sum_i\ga_i = 0$ by 
Corollary \ref{cor:full0}, 
we have 
$$
0=\ga_n+\ga_{n+1} = \langle\eta_n +\eta_{n+1},\xi_H\rangle 
=  -\sum_{i\in I} \ga_i a_i.
$$   
Thus $\langle H, c_W \rangle = \sum_{i-1}^{n+1} \gamma_i \kappa_i$  by 
Proposition \ref{DkoverD1}.
Hence Lemma \ref{blowsimp} implies that 
$\langle H, c_{\De'} \rangle = \sum_{i-1}^{n+1} \gamma_i \kappa_i$. 
The result follows.
\end{proof}

\begin{cor}\labell{cor:blowedge}
Let $H$ be mass linear on a 
 $4$-dimensional polytope $\De$ and let $\De'$ be its blowup along
 an edge $F_{123}$  that meets all asymmetric facets. 
Then $H$ is mass linear on $\De'$ if $\ga_1+\ga_2+\ga_3=0$.
\end{cor}
\begin{proof} By Proposition \ref{prop:fullML}
in dimensions $\le 4$ every mass linear function is fully mass linear.
Therefore, by Lemma~\ref{le:full0}(ii) 
every mass linear $H$ is generated by some $\xi_H$.
\end{proof}

\begin{rmk}\rm    To go further with this question one would obviously need to 
understand more about mass linear functions on the polytopes $W$.  
One could also consider the question of which blowups preserve full mass linearity.  
For example, if one blows up at a vertex then $W$ is a simplex and all linear 
functions on a simplex are inessential and hence
fully mass linear.  In this case the proof of 
Lemma~\ref{le:f4} 
adapts to show that  a
vertex blow up preserves full mass linearity.
We leave further discussion of such questions for the future.
\end{rmk}

Another interesting question concerns which blowups convert inessential functions to essential ones. 
We end by showing (in any dimension) that  if a blow up of type $(F_{ij},g)$ 
has this property, then the underlying polytope is a double expansion.

\begin{proposition}\labell{prop:essblow}
Let $H \in \ft$ be an inessential  mass linear function
on a smooth 
 polytope $\ov\Delta \subset \ft^*$. 
Assume that  $H$ is an essential mass linear function on a polytope 
$\Delta$ that is obtained from $\ov\Delta$ by
a series of blowups.  Moreover, assume that each
blowup is either along a
 symmetric face or of type $(\ov F_{ij},\ov g)$. 
Then 
\begin{itemize}
\item $\ov\Delta$ is the double expansion of a smooth
polytope $\Tilde \Delta$.
\item 
The four base-type facets are the asymmetric facets.
\end{itemize}
\end{proposition}

\begin{proof}
By 
Lemma~\ref{le:symblow} (i)
and 
Proposition~\ref{prop:blow4} 
(ii),
$H$ is mass linear on each intermediate blowup,  
the exceptional divisors are all symmetric,
and the coefficients $\gamma_k$ remain constant under blowup.
Since $H$ is essential on $\Delta$ but not on $\ov{\Delta}$,
there exists a polytope $\De'$ in the sequence so that
$H$ is inessential on $\De'$ but essential on the blowup.
Lemma~\ref{le:symblow} (ii) implies that
this blow up must be of the form
$(F'_{ij}, g')$.
Moreover,  Proposition~\ref{prop:blow4} (iii) implies that
$\De'$ has exactly four asymmetric facets and 
that $F'_i \not \sim F'_j$.
Since $H$ is inessential on $\De'$, we may label the asymmetric
facets so that $F_1 \sim F_2$ and $F_3 \sim F_4$.
Hence, $i \in \{1,2\}$ and $j \in \{3,4\}$.
Since $F'_{ij}\cap \ov g$  meets every asymmetric facet, 
this implies that
$\ov{F}_{12} \neq \emptyset$ and $\ov{F}_{34} \neq \emptyset$.
Therefore, the claim follows from Lemma~\ref{lemma:dexpan}.
\end{proof}

\end{document}